\pdfoutput=1                      

\documentclass{amsart}

\usepackage{mlmodern}
\usepackage[accsupp]{axessibility}    

\usepackage{amssymb,thmtools}
\usepackage{graphicx,stackengine,scalerel}
\usepackage{hyperref}
\usepackage{enumerate}
\usepackage[scr=rsfs]{mathalpha}
\usepackage{bm}
\usepackage{color}
\usepackage{centernot}
\usepackage{enumitem}
\usepackage{esint}

\usepackage[style=numeric, sorting=nyt, giveninits=true, eprint=false, url=false, doi=false,isbn=false, maxnames=50]{biblatex}

\usepackage{biblatex}
\makeatletter
\thmt@define@thmuse@key{tag}{%
  \thmt@suspendcounter{\thmt@envname}{#1}%
}
\makeatother

\newtheorem{theorem}{Theorem}[section]
\newtheorem{lemma}[theorem]{Lemma}
\newtheorem{prop}[theorem]{Proposition}
\newtheorem{cor}[theorem]{Corollary}

\newtheorem*{question}{Question}

\theoremstyle{definition}
\newtheorem{definition}[theorem]{Definition}
\newtheorem{example}[theorem]{Example}

\newtheorem{hyp}{Hypothesis}
\theoremstyle{remark}
\newtheorem{remark}[theorem]{Remark}
\newtheorem*{notation}{Notation}

\numberwithin{equation}{section}
\newcommand{\olsi}[1]{\,\overline{\!{#1}}} 

\newcommand{\M}{\olsi{M}}

\newcommand{\di}{\mathrm{d}}
\newcommand{\supp}{\mathrm{spt} \, }
\newcommand{\lie}{\mathscr{L}}
\newcommand{\g}{\bar{g}}

\newcommand{\bu}{\bar{u}}
\newcommand{\ric}{\operatorname{Ric}}

\newcommand{\Ric}{\overline{\ric}}
\newcommand{\bnabla}{\overline{\nabla}}
\newcommand{\balpha}{\bar{\alpha}}
\newcommand{\tr}{\operatorname{tr}}

\newcommand{\R}{\mathbb{R}}
\newcommand{\N}{\mathbb{N}}
\newcommand{\LL}{\mathbb{L}}
\newcommand{\Y}{\mathscr{Y}_\tau(\Sigma)}
\newcommand{\YO}{\mathscr{Y}_\varphi(\Omega)}
\newcommand{\loc}{\mathrm{loc}}

\newcommand{\lip}{\operatorname{Lip}}
\newcommand{\Hc}{\overline{\mathrm{H}}}
\newcommand{\cM}{\mathscr{M}}
\newcommand{\eps}{\varepsilon}

\newcommand{\norm}[1]{\left\Vert #1 \right\Vert}

\newcommand{\abs}[1]{\left\vert #1 \right\vert}

\newcommand{\second}{\operatorname{II}}
\newcommand{\X}{\mathfrak{X}}
\newcommand{\bpart}{\bar{\partial}}

\newcommand{\set}[2]{\left\{ \ #1 \ : \ #2 \ \right\}}

\newcommand{\diver}{\operatorname{div}}

\newcommand{\disp}{\displaystyle}

\newcommand{\taueq}{\overset{\tau}{=}}
\newcommand{\ee}{\g_{\mathrm{E}}^\tau}

\def\top{\ThisStyle{\abovebaseline[0pt]{\scalebox{-1}{$\SavedStyle\perp$}}}}

\newcommand{\nocontentsline}[3]{}
\let\origcontentsline\addcontentsline
\newcommand\stoptoc{\let\addcontentsline\nocontentsline}
\newcommand\resumetoc{\let\addcontentsline\origcontentsline}

\begin{document}

\title[Prescribing the mean curvature as a measure]{Prescribing the mean curvature of an achronal hypersurface as a measure: the case of 3D spacetimes.}



\author{Lorenzo Maniscalco}
\address{Dipartimento di Matematica ``Giuseppe Peano", 
Universit\`a degli Studi di Torino,
10123 Torino, Italy}
\curraddr{}
\email{lorenzo.maniscalco@unito.it}
\thanks{}

\author{Luciano Mari}
\address{Dipartimento di Matematica ``Federigo Enriques",
Universit\`a degli Studi di Milano,
20133 Milano, Italy}
\curraddr{}
\email{luciano.mari@unimi.it (corresponding author)}
\thanks{The second author is supported by the PRIN 20225J97H5 ``Differential-geometric aspects of manifolds via Global Analysis"}



\date{\today}

\dedicatory{}


\begin{abstract}
We study the existence problem for achronal hypersurfaces $M \hookrightarrow \overline{M}$ in a globally hyperbolic spacetime, whose mean curvature is a prescribed -- possibly singular -- source, and whose boundary is a given smooth spacelike submanifold. Since $M$ is allowed to go null somewhere, the mean curvature prescription is to be understood in the distributional sense. We prove a general existence and regularity theorem  for surfaces in ambient dimension $3$. Although most of our estimates hold in any dimension, recent counterexamples show that some of our conclusions fail in ambient dimension at least $5$. The case of $4$D-spacetimes is an open problem. Our theorems have application to Born-Infeld electrostatics in general static spacetimes. 
\end{abstract}

\subjclass{Primary: 
58J32, 
53C40, 
53C42, 
83C50 
Secondary:
35Q75, 
53C50
}

\keywords{Lorentzian mean curvature, Born-Infeld model, light segment, singularity, maximal hypersurface}

\maketitle

\tableofcontents

\section{Introduction}



In the present work we consider the existence problem for spacelike hypersurfaces 
\[
M^m \subseteq \M^{m+1} 
\]
with prescribed mean curvature and fixed boundary in a (boundaryless) spacetime $(\M,\g)$.  
It is known that such hypersurfaces are important  both from a physical and from a mathematical viewpoint: for instance, they can be used as convenient initial data for the Einstein Field Equations \cite{choquetbruhat}, or in connection with positive mass theorems and related geometric inequalities \cite{mt,huisken_yau}. The existence and regularity problem was extensively studied in the 1980s, when general results were obtained by Bartnik \& Simon \cite{BS}, Gerhardt \cite{gerhardt83} and Bartnik \cite{bartnik84,bartnik88} if the prescribed mean curvature is locally bounded and sufficiently regular (see also Flaherty \cite{flaherty} and Audonet \& Bancel \cite{ab}). In particular, Bartnik in \cite[Theorem 4.1]{bartnik88} considered a given spacelike hypersurface $\Sigma$ in a globally hyperbolic spacetime $\M$, and a $C^1$ function $\Hc$ on the tangent bundle $T\M$. Under mild assumptions on $\M,\Sigma$ and $\Hc$, he showed that problem
\begin{align}\label{prob_general_BI}
    \begin{cases}
        H_M(x) = \Hc(N_M(x)) \qquad \text{for } \, x \in M \\
        \partial M = \partial \Sigma
    \end{cases}
\end{align}
admits a spacelike solution $M$, where $N_M$ is the future-pointing unit normal field of $M$ and $H_M$ the mean curvature in direction $N_M$.

From a mathematical perspective, \eqref{prob_general_BI} poses several challenges. A major issue is that a candidate solution, obtained for instance by energy minimization or as a limit of approximating solutions, might fail to be spacelike somewhere (``go null", in the terminology of \cite{mt}). In regions where $M$ goes null equation \eqref{prob_general_BI} loses its sense pointwise, and it is interesting to understand which kind of singularities may appear. Among them, isolated point singularities were probably the most studied. In Lorentz-Minkowski's space $\LL^{m+1}$, by explicit integration Born \cite{born} constructed a radially symmetric maximal hypersurface ($\Hc \equiv 0$) away from a point $o \in M$ where $M$ is asymptotically a light cone, and in fact the resulting singularity at $o$ is a multiple of a Dirac-delta measure. Such behaviour of isolated, nonremovable singularities was shown to be typical by Ecker \cite{ecker}, see also \cite{bcf} and \cite[Thm. 1.17]{BIMM}. 
Maximal surfaces with point singularities were  studied in bounded domains by Klyachin \& Miklyukov \cite{KM95}, and in the whole of $\LL^3$ were classified by Kobayashi \cite{koba}, Klyachin \cite{klyachin_desc} and Fern\'andez, L\'opez \& Souam \cite{fls}, while G\'alvez, Jim\'enez \& Mira \cite{galvez_jimenez_mira} characterized the behaviour near isolated singularities of surfaces in $\LL^3$ with smooth, possibly non-constant mean curvature. 

Less is known about more complicated singularities. Of particular importance for our investigation is the case where $M$ contains \emph{light segments}. To our knowledge, the first example of a maximal surface in $\LL^3$ with a (compact) light segment is due to Pryce, see \cite[Example XI]{pryce}, and further examples in $\LL^3$ containing an entire light line were later produced in \cite{various,uy_light,auy}. With no claim of completeness, we also refer the reader to:
\begin{itemize}
    \item \cite{esturome,fls}, for an investigation of the structure of maximal surfaces with singularities;
    \item \cite{uy,fsuy}, for a detailed description of some classes of maximal surfaces with controlled singular set;
    \item \cite{klyachin_mixed,uy_light_0,uy_light}, for a local study near points where the tangent space is lightlike.
\end{itemize}

The fact that Dirac measures appear as the mean curvature of isolated point singularities may suggest that more complex singularities may also give rise to measures instead of more singular distributions. However, whether this is true, and what kinds of measures might arise, are questions that, to our knowledge, have not yet been addressed.

In our paper, we study problems like \eqref{prob_general_BI} for singular $\Hc$ modelled by a measure. We shall consider globally hyperbolic spacetimes $\M$, and refer to \cite{oneill,he} for the physical terminology used here. For the sake of simplicity, we only consider hypersurfaces $M,\Sigma$ which are achronal, i.e. such that every inextendible timelike curve intersects $M$ (respectively, $\Sigma$) at most once. We recall that the future Cauchy development $D^+(M)$ of an achronal set $M$ is the set of all points $p\in\M$ such that every past-pointing inextendible causal (i.e. timelike or null) curve meets $M$. The past Cauchy development $D^-(M)$ is accordingly defined, and the Cauchy development is defined as
\begin{align*}
    D(M)\doteq D^+(M)\cup D^-(M).
\end{align*}

For our purposes, we shall require that the ``Dirichlet data" $\Sigma$ is, loosely speaking, neither too large nor too close to a singularity, a property codified by the following 
\begin{hyp}[tag=(C)]\label{cauchy compatto}
    $D(\Sigma)$ is precompact in $\M$. 
\end{hyp}
This assumption already appears in \cite{bartnik88} and is there extensively commented. A connected subset $M \subseteq\M$ will be named 
\begin{itemize}
    \item[-] a \emph{weakly spacelike hypersurface} if it is an embedded, achronal locally Lipschitz hypersurface with boundary, which is closed in $\M$;
    \item[-] a \emph{spacelike hypersurface} if it is a weakly spacelike, $C^1$ hypersurface with boundary  and has timelike unit normal vector up to the boundary.
    \item[-] a \emph{smooth spacelike hypersurface} if it is spacelike and it is a $C^\infty$ hypersurface with boundary.
    \end{itemize} 
\begin{remark}\label{rem_cauchycompact_intro}
    Note that in Lorentz-Minkowski's space $\LL^{m+1}$ any compact, spacelike $\Sigma$ automatically satisfies Hypothesis \ref{cauchy compatto}. More generally, this is also the case for compact spacelike $\Sigma$ in static spacetimes
    \[
    \M = \R \times S \qquad \text{with  metric} \quad \g = - (\pi^*f)^2d\tau^2 + \bar \sigma, 
    \]
    where $\pi : \M \to S$ is the projection onto the second factor, $f \in C^1(S)$, $\bar \sigma = \pi^* \sigma$ and $\sigma$ is a Riemannian metric on $S$.
\end{remark}

We assume that $\Sigma$ is a smooth compact spacelike hypersurface, and seek for hypersurfaces $M$ in the set 
\begin{align*}
        \mathscr{Y}(\Sigma) \doteq \{M\subseteq \M \ : \ \text{$M$ weakly spacelike, $D(M) = D(\Sigma)$} \}.
\end{align*}
Then, \eqref{prob_general_BI} can be rephrased as a more familiar Dirichlet problem for functions. By \cite{Bernal:2005aa}, choosing a suitable time function $\tau$ (hereafter called a \emph{splitting time function}) we may write the globally hyperbolic spacetime $\M$ as 
\[
\M \taueq \R \times S 
\]
with metric 
\[
\g = \balpha^2 (-d\tau^2 + \bar \sigma),
\]
for a lapse function $0 < \balpha \in C^\infty(\M)$ and a non-negative tensor $\bar \sigma$ whose kernel is generated by $\bnabla \tau$. Let 
\[
T = - \frac{\bnabla \tau}{|\bnabla \tau|}
\]
be the future pointing normalization of $\bnabla \tau$, and let 
\[
\pi = \M \to S 
\]
be the projection onto the second factor. The global hyperbolicity of $\M$ and  condition $D(M) = D(\Sigma)$ imply both $\pi(M) = \pi(\Sigma)$ and $\partial M = \partial \Sigma$, see Lemmas \ref{lem_goodproj} and \ref{lem_stessafront}. Therefore, to each $M \in \mathscr{Y}(\Sigma)$ one may associate its height, defined as the unique function $u:\Sigma\to \R$ such that the graph map
    \begin{align*}
        F_u: \Sigma \to \R\times S\taueq\M \qquad F_u(x) = (u(x),\pi(x))
    \end{align*}
is an embedding with $F_u(\Sigma) = M$. 
The height of $\Sigma$ will be called $\varphi$. Writing $g_u\doteq F_u^*\g$, we have
\begin{align*}
    g_u = \alpha_u^2\left( \sigma_u - du^2\right)
\end{align*}
where $\alpha_u = F_u^*\balpha$ and $\sigma_u = F_u^*\bar\sigma$. Note that $M$ with its induced metric is identified via $F_u$ with $(\Sigma,g_u)$, that $\sigma_u$ is a Riemannian metric on $\Sigma$ and that
\begin{align*}
    \text{$M$ is spacelike} \quad &\iff \quad u\in C^1(\Sigma)  \quad \text{and}\quad \abs{du}_{\sigma_u} < 1 \\
    \text{$M$ is weakly spacelike} \quad &\iff \quad u\in \lip(\Sigma) \quad \text{and} \quad \abs{du}_{\sigma_u} \leq 1 \quad \text{a.e.}
\end{align*}
Denote by $\mathring\Sigma$ the relative interior of $\Sigma$. The choice of $\tau$ allows to identify $\mathscr{Y}(\Sigma)$ with
\begin{align*}
    \Y\doteq \set{u\in \lip(\Sigma)}{|du|_{\sigma_u}\leq 1 \ \, \text{a.e. on $\Sigma$}, \  \  \varphi=u \ \text{on $\partial\Sigma$} }.
\end{align*}



%
%
%
Wherever $|du|_{\sigma_u} < 1$ one can define the future pointing unit normal $N_u$ to $M$ and the associated \emph{tilt function}
\[
w_u = - \g(T,N_u) = \frac{1}{\sqrt{1-|du|^2_{\sigma_u}}}, 
\]
which plays a relevant role in existence theory. If $H_u : \mathring{\Sigma} \to \R$ is the scalar mean curvature of $M$ in direction $N_u$, \eqref{prob_general_BI} leads to study
\begin{equation}\label{eq_general_BI}
    \begin{cases}
        H_u= \Hc(F_u,N_u) & \text{in $\mathring\Sigma$,}   \\
        u = \varphi & \text{on $\partial\Sigma$}
    \end{cases}
\end{equation}
for suitable source terms $\Hc$ which are possibly singular. Equation \eqref{eq_general_BI} does not make sense pointwise where $|du|_{\sigma_u} = 1$, and so it needs a suitable weak definition.

\begin{remark}\label{rem_Omega}
  Problem \eqref{eq_general_BI} is a quasilinear equation for $u : \Sigma \to \R$. However, $\pi$ projects $\Sigma$ diffeomorphically onto the closure of a relatively compact domain $\Omega \Subset S$, and it is easy to see that each $M \in \mathscr{Y}(\Sigma)$ projects onto $\overline \Omega$ as well (see Lemma \ref{lem_goodproj}). Therefore, up to composing with $\pi_{|\Sigma}$ we can consider \eqref{eq_general_BI} as a Dirichlet problem in $\Omega$. For instance, in Lorentz-Minkowski's space
  \[
\LL^{m+1} = \R \times \R^m, \qquad \bar g = - \di x_0^2 + \sum_{j=1}^m \di x_j^2
  \]
  the first in \eqref{eq_general_BI} becomes
  \[
  \diver \left( \frac{Du}{\sqrt{1-|Du|^2}} \right) = \Hc\left(x,u(x), \frac{\partial_0 - u^j \partial_j}{\sqrt{1-|Du|^2}}\right) \qquad \text{in } \, \Omega \subseteq \R^m = \{x_0 = 0\}, 
  \]
  where $Du = u^j\partial_j$ is the Euclidean gradient of $u$. Since the mean curvature prescription problem is independent of the choice of the splitting time function $\tau$, we preferred to refer to the geometric Dirichlet data $\Sigma$ instead of the $\tau$-dependent set $\Omega$.
\end{remark}

A further reason to consider singular $\Hc$ comes from nonlinear electrodynamics. Indeed, hypersurfaces with prescribed mean curvature appear in the Born-Infeld theory for electromagnetism, see the surveys \cite{BI,Birula,Kiessling-legacy,Yang} and Appendix \ref{appe_BI} below. The model, proposed in \cite{BI} as an alternative to Maxwell’s theory, successfully solves the issue of infinite energy associated with point charges in Maxwell's description and possesses further significant properties from the physical viewpoint \cite{boillat,plebanski}. When considering the electrostatic case in $\R^m$, according to Born-Infeld's theory the electric potential $u$ generated by a charge distribution $\rho$ in $\Omega \subseteq \R^m$ (typically, a measure) minimizes the energy 
\begin{equation}\label{eq_energy_BI}
I_\rho(\psi) = \int_\Omega \left( 1- \sqrt{1-|D\psi|^2} \right)d V_\delta + \langle \rho, \psi \rangle 
\end{equation}
in $\YO$, the set of $\psi \in W^{1,\infty}(\Omega)$ attaining the boundary data $\varphi$ on $\partial \Omega$ and satisfying $|D\psi| \le 1$. Here, $dV_\delta$ is the Lebesgue measure on $\R^n$, $\langle \, , \, \rangle$ is the natural pairing between measures and continuous functions, and physical units are set to $1$ (a word of warning: in this paper, we adopt a sign convention on the electric potential $u$ which is opposite to that of \cite{BIMM}, resulting in the different choice of the sign of $\langle \rho, \psi \rangle$ in $I_\rho$). The Euler-Lagrange equations of $I_\rho$ are formally 
\begin{align}\label{BI}\tag{$\mathcal{BI}$}
    \begin{cases}
    \diver\left(\frac{Du}{\sqrt{1-|Du|^2}}\right) = \rho & \quad \text{in } \Omega, \\
    u = \varphi & \quad \text{on } \, \partial \Omega.
    \end{cases}
\end{align}
hence the graph of $u$ can be seen as a spacelike hypersurface in $\LL^{m+1}$ with prescribed mean curvature $\rho$ in the future pointing direction. Moreover, the tilt function $w_u$ relates to the energy density $T_{00}$ (i.e. the $(0,0)$ component of the stress-energy tensor associated to $I_\rho$) via the identity
\[
T_{00} = w_u -1 - \rho u,
\]
see the Appendix of \cite{BIMM} (note the different sign of $\rho u$ due to our convention). As we shall see later and in Appendix \ref{appe_BI}, Born-Infeld's electrostatics in more general static spacetimes still leads to an equation like \eqref{BI}, now set in a relatively compact domain $\Omega$ of a fixed spacelike slice. 

The functional $I_\rho$ is convex and admits a (unique) minimizer $u_\rho$, which is therefore the only candidate as a solution to \eqref{BI}. However, $I_\rho$ is not $C^1$ due to the singularity of the Lagrangian where $|D\psi|=1$, and this fact may prevent $u_\rho$ to actually solve \eqref{BI}. In recent years, the existence problem for \eqref{BI} received increasing attention since the influential paper \cite{bpd} by Bonheure, d'Avenia and Pomponio. Among others, we quote works by Klyachin-Miklyukov \cite{KM95,klyachin_desc}, Kiessling \cite{kiessling}, Bonheure, Colasuonno \& Foldes \cite{bcf} for $\rho$ a sum of Dirac sources:
\begin{equation}\label{eq_diracdeltas}
\rho = \sum_{j=1}^k a_j \delta_{x_j}, \qquad a_j \in \R, 
\end{equation}
which models interacting point charges, those by Bonheure \& Iacopetti \cite{boniaco_1,boniaco_2} and Haarala \cite{haa} for sources $\rho \in L^q(\R^m)$ with $q > m$, and the one by Byeon, Ikoma, Malchiodi \& Mari \cite{BIMM} for charges in $L^2_\loc$ away from a compact set of singularities with vanishing $1$-dimensional measure.  
%
%
%
%

%

To describe our main result and put it into context, we first observe that $\Y$ can be seen as a closed, convex subset of $W^{1,p}(\Sigma)$ with the induced topology, for some fixed $p \in (2,\infty)$. By Morrey's embedding, $\Y \hookrightarrow C(\Sigma)$ compactly.

\begin{notation}
Here and in what follows all metric quantities, Sobolev spaces and Hausdorff measures on $\Sigma$ will implicitely be taken with respect to the metric $\sigma_\varphi$, unless specified otherwise. This is not a restriction, since by Lemma \ref{conseguenze di C} below any choice of metric $\sigma_v$ with $v \in \Y$ leads to equivalent distances, measures, $L^p$ and $W^{1,p}$ norms. Moreover, boundedness in $W^{2,2}$ is independent of $\sigma_v$.
\end{notation}


We shall be interested in mean curvature prescriptions $H_u$ depending on pairs $(\rho,X)$, where:
\begin{itemize}
    \item $X$ is a continuous vector field on $\overline{D(\Sigma)}$;
    \item $\rho$ is a continuous map 
    \[
    \rho \ \  : \ \ (\Y, \|\cdot \|_{C(\Sigma)}) \ \to \ \mathscr{M}(\Sigma),
    \]
    where $\mathscr{M}(\Sigma)$ is the space of signed, finite Radon measures in  $\Sigma$ with the weak$^*$ topology,
\end{itemize}
and given by the identity 
\begin{equation}\label{eq_our_prescription}
H_u \, dV_{\sigma_u} = \rho(u) + \g(X, N_u) \, dV_{\sigma_u} \qquad \text{in } C_c^1(\mathring\Sigma)^*,
\end{equation}
that is, by integrating against $\eta \in C_c^1(\mathring\Sigma)$. Here, $dV_{\sigma_u}$ is The Riemannian volume measure of $\sigma_u$. 
Note that the presence of the unit timelike normal vector $N_u$ needs, at least, that the graph of $u$ is spacelike almost everywhere. If $\rho(u)$ is absolutely continuous with respect to some (hence all) volume measures $dV_{\sigma_v}$, $v\in \Y$, then $u$ solves \eqref{eq_BI} if and only if
    \begin{align*}
        H_u = \frac{d\rho(u)}{dV_{\sigma_u}} + \g(X,N_u) \qquad \text{as functions,}
    \end{align*}
where $\frac{d\rho(u)}{dV_{\sigma_u}}$ is the Radon-Nykodim derivative of $\rho(u)$ with respect to $dV_{\sigma_u}$

\begin{example}
Taking into account Remark \ref{rem_Omega}, in Minkowski space $\LL^{m+1}$ equation  \eqref{eq_our_prescription} becomes
\[
\diver \left( \frac{Du}{\sqrt{1-|Du|^2}} \right) d V_\delta = \rho(u) + \frac{X^j u_j-X^0}{\sqrt{1-|Du|^2}} d V_\delta \qquad \text{in } \, \Omega \subseteq \R^m,
\]
where $\delta$ is the Euclidean metric and $X = X^0 \partial_0 + X^j \partial_j$. 
In particular, for $a \in C^1(\Omega)$ the equation 
\[
\diver \left(\frac{e^{-a} Du}{\sqrt{1-|Du|^2}} \right) d V_\delta = e^{-a}\rho(u) \qquad \text{in } \, \Omega \subseteq \R^m
\]
is obtained by choosing $X^0 = 0$ and $X^j =\partial_j a$. 
\end{example}

\begin{example}
    The above class of $\rho$ includes maps of the type $u \mapsto \Hc(F_u) dV_{\sigma_u}$, where $F_u : \Sigma \to \M$ is the graph of $u$ and $\Hc$ is a continuous function in $\overline{D(\Sigma)}$. More generally, one can consider 
    \[
    \rho(u) = \vartheta + \Hc(F_u) dV_{\sigma_u}, 
    \]
where $\vartheta \in \cM(\Sigma)$ is a fixed measure. 
\end{example}

\begin{remark}
    It would be interesting to find an alternative description of the singular data $\rho$ in terms of a measure in $\M$, and the Hausdorff measure constructed through causal diamonds by McCann and S\"amann in \cite{mccannsamann} seems a natural candidate. A subtle point may be to properly formalize the fact that the measure is ``mildly varying in the time variable", here codified by the continuity of $\rho$.     
\end{remark}

One motivation to include the term $\g(X,N_u)$ in our mean curvature prescription is due to Born-Infeld's electrostatics in a general static spacetime
\[
V = \R \times S
\]
with metric 
\[
\langle \cdot,\cdot \rangle = - (\pi^*\alpha)^2d \tau^2 + \pi^*\sigma,
\]
where $\alpha \in C^\infty(S)$ and $\sigma$ is a smooth metric on $S$. In fact, as we shall see in Appendix \ref{appe_BI}, the electric potential generated by a charge $\rho$ on $(S,\sigma)$ satisfies
\begin{align}\label{eq_born_infeld_electrostatics_intro}
    \diver_\sigma \left(\frac{\alpha^{-1}Du}{\sqrt{1 - \alpha^{-2}|Du|^2}}\right) = \rho \qquad \text{in $S$.}
\end{align}
%
Whence, if we endow $V$ with the metric 
\[
\g = -(\pi^*\alpha)^{-2}d \tau^2 + \pi^*\sigma,
\]
the graph of $u$ in $(V,\g)$ is a spacelike hypersurface with mean curvature 
    \begin{align*}
        H_u = \rho + \g(X,N_u), 
    \end{align*}
where $X = \bnabla_{\alpha \partial_\tau}(\alpha \partial_\tau)$.\\[0.1cm]



We now describe our main theorem. First, observe that if $\Hc$ and $X$ are of class $C^1$ in $\overline{D(\Sigma)}$ the equation 
\[
H_u = \Hc(F_u) + \g(X,N_u) 
\]
is included in the existence theory developed by Bartnik in \cite{bartnik88}. However, as he pointed out at the end of the Introduction of \cite{bartnik88}, even when $X \equiv 0$ treating the problem with $\Hc$ less regular than $C^1$ needs new ideas. In Minkowski space $\LL^{m+1}$, Bartnik and Simon \cite{BS} were able to treat sources with $X \equiv 0$ and $\Hc$ continuous (in fact, $\Hc$ locally bounded and Carathéodory) by using different tools, some of them specific to the flat ambient space. However, even in this case the local boundedness of $\Hc$ is essential.

In Theorem \ref{main result} below, we allow $\rho(u)$ to be singular in a set $E \Subset \mathring\Sigma$ with vanishing $1$-dimensional Hausdorff measure, and locally $L^2$ in the complement. The request $\mathscr{H}^1(E)=0$ is to prevent that the graph of a possible solution contains light segments projecting into $E$. For reasons that will be discussed below, we shall restrict to surfaces in $(2+1)$-dimensional ambient spaces.

\begin{theorem}\label{main result}
    Let $\M$ be a globally hyperbolic spacetime of dimension $2+1$ and let $\Sigma$ be a smooth, compact spacelike hypersurface satisfying \ref{cauchy compatto}. Choose a splitting time function $\tau$ and let $\varphi \in C^\infty(\Sigma)$ be the height function of $\Sigma$. Consider a pair $(\rho,X)$, and assume that there exists $E \subseteq \mathring{\Sigma}$ compact with $\mathscr{H}^1
    (E)=0$ such that for each $\Sigma' \Subset \mathring{\Sigma} \backslash E$ the restriction
    \begin{equation}\label{ipo_L2omega'}
    \rho \llcorner \Sigma' \ : \  u \mapsto \rho(u) \llcorner \Sigma'
    \end{equation}
    is valued and bounded in $L^2(\Sigma')$. Then, the Dirichlet problem
    \begin{equation}\label{eq_BI}\tag{{\rm PMC}}
    \begin{cases}
        H_u dV_{\sigma_u} = \rho(u) + \g(X, N_u)dV_{\sigma_u} & \text{in $\Sigma$} \\
        u = \varphi & \text{on  $\partial\Sigma$}
    \end{cases}
    \end{equation}
    has a weak solution $u\in\Y\cap W^{2,2}_\loc(\Sigma\backslash E)$ with the following properties:
    \begin{itemize}
        \item[(i)] the graph of $u$ has no light segments;
        \item[(ii)] the tilt function $w_u$ satisfies
        \[
        w_u \in L^1(\Sigma) 
        \]
    and, for each $\Sigma' \Subset \mathring{\Sigma} \backslash E$,
        \begin{equation}\label{eq_higherinteg}
        \begin{array}{l}
        w_u \ln w_u \in L^1(\Sigma'), \\[0.3cm]
        w_u|D^2u|^2 + w_u^3 |D^2u(Du,\cdot)|^2 + w_u^5 |D^2u(Du,Du)|^2 \in L^1(\Sigma').
        \end{array}
        \end{equation}
        where $D$ and norms are taken in the metric $\sigma_u$. In particular, 
        \[
        w_u^{-1/2} \in W^{1,2}(\Sigma'), \quad |D \log w_u| \in L^1(\Sigma')
        \]
        \item[(iii)] If $\Sigma'\Subset \mathring{\Sigma}\backslash E$ is a domain such that $X$ is $C^1$ in $\overline{D(\Sigma)}\cap(\Sigma'\times\R)$ and $\rho\llcorner\Sigma'$ is valued in $\mathscr{C}^1(\overline{\Sigma'})$ and there bounded, then $u\in C^{2,\alpha}_\loc(\Sigma')$ and is there spacelike. Furthermore, if $\rho$ and $X$ are smooth in $\Sigma'$ then so is $u$.
    \end{itemize}    
\end{theorem}

\begin{remark}
    For the application of Theorem \ref{main result} to Born-Infeld's electrostatics, please see Theorem \ref{te_forBI}.
\end{remark}
As an example, Theorem \ref{main result} applies to the Dirichlet problem \eqref{BI} for $\Omega \subseteq \R^2$ and $\rho \in \cM(\overline\Omega)$, which represents hypersurfaces of mean curvature $\rho$ in Minkowski's space $\LL^3$. In this case, $\sigma_u$ is the Euclidean metric for each $u$,  $\rho$ is a constant map and $X \equiv 0$. Setting the problem in $\Omega$ instead of $\Sigma$ as in Remark \ref{rem_Omega}, our assumptions on $\rho$ can be rewritten as follows: there exists a compact subset $E \subseteq \Omega$ such that 
\[
\rho = \rho_{\rm S} + \frac{d\rho}{d V_\delta} d  V_\delta, \qquad \text{with} \quad \left\{ \begin{array}{l}
\supp \rho_{\rm S} \subseteq E, \\[0.2cm]
\frac{d\rho}{d  V_\delta} \in L^1(\Omega) \cap L^2_\loc(\Omega \backslash E).
\end{array}\right. 
\]
and $\mathscr{H}^1(E) = 0$. This case was considered in \cite[Theorem 1.11]{BIMM}. Note also that the choices $E = \{x_1,\ldots x_k\}$ and $\rho$ a sum of Dirac deltas as in \eqref{eq_diracdeltas} enable to recover \cite[Theorem 2]{KM95}.

\begin{remark}
In \cite[Theorem 1.11]{BIMM}, conclusion \textit{(ii)} is a bit stronger than in \eqref{eq_higherinteg}, as both the  integrability relations in \eqref{eq_higherinteg} hold with an extra factor $\ln^q w_u$ for any $q \ge 0$. On the other hand, the mean curvature prescription in \eqref{eq_BI} allows for a larger class of sources even in $\LL^3$. We point out that, in \cite{BIMM}, $\Sigma$ is not assumed to have a regular spacelike boundary. Up to weakening the first conclusion in \textit{(ii)} to $w_u \in L^1_\loc(\Sigma)$, it is likely that the spacelike and smoothness assumptions on $\partial \Sigma$ in Theorem \ref{main result} could be removed as well, but for the sake of simplicity we do not address the issue here.   
\end{remark}

The strategy for proving Theorem \ref{main result}, summarized in Section \ref{strategy}, is inspired by \cite{BIMM}. However, significant changes have to be made to overcome the problems arising in our setting. 
For instance, some of the results in \cite{BIMM} are strongly based on variational arguments and on the convexity of $I_\rho$ in \eqref{eq_energy_BI}, which in particular implies the uniqueness of solutions. In contrast, \eqref{eq_BI} is, in general, non-variational and uniqueness fails. This affects, among others, the proofs of the tilt estimate $w_u \in L^1(\Sigma)$ and of the absence of light segments, which need new ideas. More subtle is also the proof of the second in \eqref{eq_higherinteg}, which can be rephrased as a bound on the second fundamental form of the graph of $u$. The argument relies on a careful rearrangement of the Jacobi equation, see Subsection \ref{sec_secondfund}.

Theorem \ref{main result} fails for spacetimes of dimension $m+1 \ge 5$, even for $\LL^{m+1}$. A counterexample was found in \cite[Corollary 1.10]{BIMM}: there exists a domain $\Omega \Subset \R^m$ and $u \in C_c^\infty(\Omega)$ spacelike outside of a segment $\overline{xy}$ and satisfying:
\begin{itemize}
    \item $u$ minimizes $I_\rho$ in \eqref{eq_energy_BI}  with boundary value $\varphi \equiv 0$ and 
\[
\rho = \delta_x - \delta_y + \rho_{\rm AC}, \qquad \rho_{\rm AC} \in L^q(\Omega) \ \  \forall q < m-1,
\]
(hence, $u$ is the only candidate to solve \eqref{BI} by the convexity of $I_\rho$), 
    \item $u$ does not solve \eqref{BI}.
\end{itemize}
Here, $E = \{x,y\}$, $X \equiv 0$ and (the constant map) $\rho$ satisfies all the conditions of Theorem \ref{main result}, but \eqref{eq_BI} admits no solution. Also, by construction the graph of $u$ has a light segment over $\overline{xy}$, so \textit{(i)} in Theorem \ref{main result} fails. Note that $\rho_{\rm AC} \in L^2(\Omega)$ if $m \ge 4$, and barely fails to be $L^2$ if $m=3$, leading us to formulate the following
\begin{question}
    Is Theorem \ref{main result} valid in ambient dimension $3+1$?
\end{question}
A positive answer to the question depends on solving the following two problems:

\begin{itemize}
    \item[(A)] prove that (the graph of)  $u$ has no light segments;
    \item[(B)] even assuming that $u$ has no light segments in a domain $\Sigma' \Subset \mathring\Sigma \backslash E$, prove the higher integrability $w_u \ln w_u \in L^1(\Sigma')$. This latter serves to guarantee that $u$ weakly solves \eqref{eq_BI} for test functions supported in $\Sigma'$.  
\end{itemize}

The validity of (A) is unknown even in $\LL^4$, and likely needs a refined analysis of the behaviour of $u$ and its level sets near a light segment. To our knowledge, no such result is available in the existing literature. On the other hand, (B) holds for \eqref{BI} in $\Omega \Subset \R^m$ by \cite[Theorem 1.14]{BIMM}. However, the estimates in \cite{BIMM} are based on the second order properties of the Lorentzian distance function
\[
\ell \ : \ \M \times \M \to \R,
\]
in particular on the fact that the $(1,1)$-Hessian of $\ell^2$ in Minkowski space is a multiple of the identity. As observed in \cite[Remarks at p.158]{bartnik88}, bounding $\bnabla^2 \ell$ in a way that is effective for our purposes is not easy in general globally hyperbolic spacetimes, and we are not aware of results in this direction. In this respect, the comparison theory developed in \cite{erke_garcia_kupe,alias_hurtado_palmer,impera} seems difficult to apply.

\vspace{0.5cm}

\noindent \textbf{Acknowledgements.} The authors thank Alessandro Iacopetti for several interesting discussions. The second author is partially supported by the PRIN project no. 20225J97H5 (Italy) ``Differential-geometric aspects of manifolds via Global Analysis''.





\section{Preliminaries}\label{sec_prelim}

\subsection*{Globally hyperbolic spacetimes} Let $(\M,\g)$ be an $(m+1)$-dimensional spacetime, that is, a Lorentzian time-oriented manifold, and denote by $\bnabla$ its metric connection. The following characterization result for globally hyperbolic spacetimes was proved in \cite{Bernal:2005aa}.

\begin{lemma}
    For an $(m+1)$-dimensional spacetime $(\M,\g)$ the following are equivalent.
    \begin{enumerate}
        \item $(\M,\g)$ is globally hyperbolic (see \cite[Chapter 14]{oneill})
        \item\label{space-time_splitting}
        There exists a smooth submersion $\tau:\M\to \R$, called \emph{splitting time function} such that the following hold.
        \begin{enumerate}[label = (\roman*)]
            \item The gradient $\bnabla\tau$ is time-like and past-pointing everywhere.
            \item Each level set $S_t \doteq \{\tau = t\}$ is a smooth, spacelike, Cauchy hypersurface diffeomorphic to an $m$-dimensional manifold $S$.
            \item The flow of $\bnabla\tau$ induces a smooth submersion $\pi:\M\to S$ and the product map
            \begin{align*}
                \M \to \R\times S \qquad p \mapsto (\tau(p),\pi(p))
            \end{align*}
            is a diffeomorphism. In this case we write $\M \taueq \R\times S$.
            \item The metric writes as
            \begin{align*}
                \g = \balpha^2 (-d\tau^2 + \bar\sigma)
            \end{align*}
            where $\balpha = (-\g(\bnabla\tau,\bnabla\tau))^{-1/2}$ is called the \emph{lapse function} and $\bar\sigma$ is a $2$-covariant symmetric tensor whose kernel is generated by $\bnabla\tau$ and that restricts to a Riemannian metric on every $S_t$. 
        \end{enumerate}
    \end{enumerate}
\end{lemma}

When a splitting time function $\tau$ is chosen, we will always denote by
\begin{align}\label{eq_observer}
    T = -\balpha\bnabla\tau
\end{align}
the \emph{future pointing} normalization of $\bnabla\tau$. Physically, $T$ is  an observer and the level sets $S_t$ are to interpret as the rest spaces of $T$.

\subsection*{Spacelike hypersurfaces} 
Let $\Sigma\subset\M$ be a compact spacelike hypersurface with precompact Cauchy development $D(\Sigma)$ and set
    \begin{align*}
        \mathscr{Y}(\Sigma) \doteq \{M\subset \M \ | \ \text{$M$ weakly spacelike, $D(M) = D(\Sigma)$} \}.
    \end{align*}

\begin{lemma}\label{lem_goodproj}
    For each splitting time function $\tau$ and associate projection $\pi:\M\to S$, it holds
    \begin{align*}
        D(\Sigma) = D(M) \quad \Longrightarrow \quad \pi(\Sigma) = \pi(M). 
    \end{align*}
   Moreover, the restriction $\pi : M \to \pi(M)$ is a homeomorphism, and a $C^1$ diffeomorphism if $M$ is spacelike.
\end{lemma}
\begin{proof}
    Assume by contradiction the existence of $p\in \Sigma$ such that $\pi(p)\in\pi(\Sigma)\backslash\pi(M)$. Then, the integral curve $\gamma$ of $T$ passing through $p$ is a timelike inextensible curve that does not meet $M$. But this would mean that $p\in D(\Sigma)\backslash D(M)$, a contradiction. Hence, $\pi(\Sigma)\subseteq\pi(M)$, and the reverse inclusion follows by switching the roles of $\Sigma$ and $M$. The map $\pi$ restricted to $M$ is continuous and bijective, hence a homeomorphism since $M$ is compact. The last statement follows by the inverse function theorem. 
\end{proof}

Assume a splitting time function $\tau$ has been chosen. For a given $M\in\mathscr{Y}(\Sigma)$ define the \emph{height function} of $M$ with respect to $\tau$ as the only function $u:\Sigma\to \R$ such that the graph map
    \begin{align*}
        F_u: \Sigma \to \R\times S\taueq\M \qquad F_u(x) = (u(x),\pi(x))
    \end{align*}
is an embedding with $F_u(\Sigma) = M$. Lemma \ref{lem_goodproj} implies that $u$ is well defined on the whole of $\Sigma$. We will always call $\varphi \in C^1(\Sigma)$ the height function of $\Sigma$. If we let $g_u\doteq F_u^*\g$, we have
\begin{align*}
    g_u = \alpha_u^2\left( \sigma_u - du^2\right)
\end{align*}
where $\alpha_u = F_u^*\balpha$ and $\sigma_u = F_u^*\bar\sigma$. Note that $M$ with its induced metric is identified via $F_u$ with $(\Sigma,g_u)$, that $\sigma_u$ is always a (Lipschitz continuous) Riemannian metric on $\Sigma$ and that
\begin{align*}
    \text{$M$ is spacelike} \quad &\iff \quad u\in C^1(\Sigma)  \quad \text{and}\quad \abs{du}_{\sigma_u} < 1 \\
    \text{$M$ is weakly spacelike} \quad &\iff \quad u\in \lip(\Sigma) \quad \text{and} \quad \abs{du}_{\sigma_u} \leq 1.
\end{align*}

\begin{lemma}\label{lem_stessafront}
    For each splitting time function $\tau$ there exists a one to one correspondence
    \begin{align*}
        \mathscr{Y}(\Sigma) \quad &\longleftrightarrow \quad \Y \doteq\{u\in\lip(\Sigma) \ | \ |du|_{\sigma_u}\leq 1 \text{ and } u = \varphi \text{ on $\partial\Sigma$}\}
    \end{align*}
    given by associating to every weakly spacelike hypersurface its height function with respect to $\tau$.
\end{lemma}
\begin{proof}
    Assume that $M\in\mathscr{Y}(\Sigma)$ and let $u$ be its height function. Then by \cite[Proposition 6.3.1]{he} $u$ is Lipschitz and since $M$ is achronal it must be $|du|_{\sigma_u}\leq 1$. Otherwise, there would be a tangent vector $V\in T\Sigma$ such that $g_u(V,V)<0$. Then, by continuity, any curve $\gamma:(-\epsilon,\epsilon)\to\Sigma$ with $\gamma'(0) = V$ would be timelike for small enough $\epsilon>0$, which is impossible since $\Sigma$ is achronal. On the other hand, if $u\in\lip(\Sigma
    )$ satisfies $|du|_{\sigma_u} \le 1$, then for any absolutely continuous curve $\gamma$ on $\Sigma$ we have $g_u(\gamma',\gamma')\leq 0$, that is, $\gamma$ is not timelike and hence $M$ is achronal.

    We only now need to show that $\partial\Sigma = \partial M$ for each $M\in\mathscr{Y}(\Sigma)$. Since $\pi:\Sigma\to \pi(\Sigma)$ is a diffeomorphism, $\pi(\Sigma)$ has $C^1$ boundary, and we know by Lemma \ref{lem_goodproj} that $\pi(\Sigma) = \pi(M)$ and that $\pi : M \to \pi(\Sigma)$ is a homemorphism. In particular, both $\partial M$ and $\partial \Sigma$ project onto $\partial \pi(\Sigma)$. Suppose by contradiction that $\partial \Sigma \neq \partial M$, so that there exist $p \in \partial \Sigma$, $q \in \partial M$ with $p \neq q$ and $\pi(p) = \pi(q) = x$. Without loss of generality, we can assume that $\tau(q) > \tau(p)$, and set $\delta = \tau(q) - \tau(p)$. Consider the outer unit normal $n$ to $\partial\pi(\Sigma)$ in $S$ and the curve $\gamma: [  2 \tau(p) - \tau(q), \tau(q)]\to \R \times S$ given by 
    \[ 
    \gamma(\tau)= \left\{ \begin{array}{ll}
    (\tau,x) & \quad \text{for } \, \tau < 2 \tau(p)-\tau(q) \  \text{or} \, \tau > \tau(q), \\[0.3cm] 
    \disp \left(\tau,  \exp^S_x\left(\epsilon\sin\left( \frac{\pi(\tau - \tau(p) +\delta)} {2\delta}\right)n\right) \right) & \quad \text{if } \, \tau \in [ 2 \tau(p)- \tau(q), \tau(q)]
    \end{array}\right.
    \]
    with $\epsilon>0$ small enough to make $\gamma$ timelike. Then, $\gamma$ is an inextendible piecewise $C^1$ timelike curve that meets $q$ at time $\tau(q)$ but does not meet $p$, hence $q \in D(M)\backslash D(\Sigma)$, a contradiction.
\end{proof}

\begin{notation}
    When $\tau,\Sigma$ have been specified, if there is no risk of confusion we will drop the subscript $u$ and simply write $F = F_u$, $g = g_u$, and so on. We will use the symbols $\bnabla$ and $\nabla$ to denote the Levi-Civita connections of $\g$ and  $g_u$ respectively, while $D$ will stand for that of $\sigma_u$.
    Tensor norms will be denoted as follows:
    \begin{align*}
        \abs{A}_{\g}^2\doteq\g(A,A) && \abs{A}_{g_u}\doteq\sqrt{g_u(A,A)} && \abs{A}_{\sigma_u}\doteq\sqrt{\sigma_u(A,A)}
    \end{align*}
    but when the norm and the covariant derivatives are computed with respect to the same metric the subscript will be omitted, for instance
    \begin{align*}
        |Du|=|Du|_{\sigma_u} = |du|_{\sigma_u}, && |\nabla u| = |\nabla u|_{g_u} = |du|_{g_u}.
    \end{align*}
\end{notation}

The volume densities of the metrics $g$ and $\sigma$ write in terms of the tilt function 
\[
w \doteq - \g(T,N)
\]
as follows:
\begin{align}\label{volumi}
    dV_g = w^{-1}\alpha^m \, dV_\sigma.
\end{align}
By a direct computation we see that 
\begin{equation}\label{eq_nu_extri}
    w = \frac{1}{\sqrt{1-\abs{Du}^2}} \qquad N= w\left(U+T\right), \qquad U\doteq\balpha\bnabla\bu, \qquad \bu \doteq \pi^* u
\end{equation}
and that, in a local frame $\{\partial_i\}$ on $\Sigma$, the metric $g = g_u$ and its inverse write as
\begin{align}
    g_{ij} = \alpha^2(\sigma_{ij}-u_iu_j), \qquad g^{ij}=\alpha^{-2}(\sigma^{ij}+w^2u^iu^j)
\end{align}
where $\alpha = \alpha_u$, $\sigma = \sigma_u$, $u_i = \partial_i u$, and $u^i \doteq \sigma^{ij}u_j$. Note also that by \eqref{eq_observer} the tangential component of $T$ is
    \begin{align}\label{tangente di T}
    T^{\top} = -\alpha\nabla u,
\end{align}
hence, from $-1 = \abs{T}_{\g}^2 = \abs{T^{\top}}^2 - \g(T,N)^2$ we have
\begin{align}\label{eq_w_nabla_u}
    w^2 = \alpha^2\abs{\nabla u}^2 + 1.
\end{align}

\subsection*{Mean curvature} Fix a spacelike $M\in\mathscr{Y}(\Sigma)$ with height function $u$, graph map $F = F_u: \Sigma\to \M$ and future unit normal $N$. The \emph{second fundamental form} $\second$ of $F$ in direction $N$ is defined by the identity
\begin{align*}
    \bnabla_{F_*X}F_*Y=F_*(\nabla_XY)+\second(X,Y)N&& \forall X,Y\in\X(\Sigma),
\end{align*}
that is, 
\begin{align*}
    \second(X,Y)= -\g(\bnabla_{F_*X}F_*Y, N) = \g( F_*Y, \bnabla_{F_*X}N).
\end{align*}
The \emph{mean curvature} of $F$ in direction $N$ is $H\doteq \tr_g\second$. By an elementary computation
we see that
\begin{align}\label{seconda}
        w\second = \alpha\nabla^2u + du\odot d\alpha + \frac{1}{2}F^*\lie_T\g
    \end{align}
where $du\odot d\alpha=\frac{1}{2}(du\otimes d\alpha+d\alpha\otimes du)$ and $\lie_T\g$ is the Lie derivative of $\g$ in direction $T$.

For a vector field $X\in\X(\M)$ we define the tangential divergence as
\begin{align*}   
    \diver_M X \doteq \tr_{M}\bnabla X = \sum_{j=1}^m \g(\bnabla_{F_*e_j} X,F_*e_j) \in C^\infty(\Sigma)
\end{align*}
where $\{e_j\}$ is any local $g$-othonormal frame. Therefore,
\begin{align}
    \diver_M X = \diver_{\g}X + \g(\bnabla_NX,N),
\end{align}
and decomposing $X$ along $M$ into its tangential and normal components as
\begin{align*}
    X = F_*X^{\top} - \g(X,N)N
\end{align*}
one has
\begin{align}\label{eq_local_diverg_thm}
    \diver_MX = \diver_g X^{\top} - \g(X,N)H.
\end{align}
In particular $H = \diver_MN$ and the following integration by parts rule holds for any test function $\eta\in C_c^1(\mathring\Sigma)$:
\begin{align*}
    \int_\Sigma \eta\diver_MX \, dV_g = -\int_{\Sigma} g(\nabla\eta,X^{\top}) \, dV_g - \int_\Sigma \eta g(X,N)H \, dV_g.
\end{align*}
Setting $X=T$ and using 
\eqref{tangente di T} we deduce the following expression for the mean curvature:
\begin{align}\label{eq_mean_curvature}
    wH = \diver_g(\alpha\nabla u) + \diver_M T.
\end{align}
Therefore, for any $\eta\in C_c^1(\mathring\Sigma)$ it holds
\begin{align}\label{eq_smooth_weak_solution}
    \int_\Sigma \eta\diver_MT \, dV_g = \int_\Sigma g(\alpha\nabla u, \nabla\eta) \, dV_g + \int_\Sigma \eta wH \, dV_g.
\end{align}
Recalling \eqref{volumi} and the relation $g(\nabla u,\nabla\eta) = w^2\alpha^{-2}\sigma(Du,D\eta)$ this can be written as an integral identity with respect to the volume $dV_\sigma$ as follows:
\begin{align}\label{eq_smooth_weak_solution_sigma}
    \int_\Sigma \eta w^{-1}\alpha^m \diver_MT \, dV_\sigma = \int_\Sigma w\alpha^{m-1}\sigma(du,d\eta) \, dV_\sigma + \int_\Sigma \eta \alpha^m H \, dV_\sigma.
\end{align}
Here, we have made a slight abuse of notation by writing $\sigma(du,d\eta)$ instead of $\sigma^{-1}(du,d\eta)$.

The identities \eqref{eq_smooth_weak_solution} and \eqref{eq_smooth_weak_solution_sigma} allow us to define a weak solution for \eqref{eq_BI}. Suppose we are given a continuous vector field $X$ on $\overline{D(\Sigma)}$ and a continuous map $\rho: (\Y, \|\cdot \|_{C(\Sigma)}) \to\cM(\Sigma)$. 

\begin{definition}
We say that $u\in\Y$ is a \emph{weak solution to} 
\begin{equation}\tag{{\rm PMC}}\label{eq_BI_section}
    \begin{cases}
        H_u dV_{\sigma_u} = \rho(u) + \g(X, N_u)dV_{\sigma_u} & \text{in $\Sigma$} \\
        u = \varphi & \text{on  $\partial\Sigma$}
    \end{cases}
\end{equation}
if $w_u\in L^1_\loc(\mathring\Sigma)$ and for every $\eta\in C_c^1(\mathring\Sigma)$ it holds
\begin{align*}
    \int_\Sigma \eta \frac{\alpha^m}{w} \diver_MT \, dV_\sigma = \int_\Sigma w\alpha^{m-1}\sigma(du,d\eta) \, dV_\sigma +\int_\Sigma \eta\alpha^m \, d\rho + \int_\Sigma
        \eta\alpha^m\g(X,N) \, dV_\sigma. \nonumber
\end{align*}
where $\rho = \rho(u)$, $w = w_u$, $\sigma = \sigma_u$, $\alpha = \alpha_u$ and so on.
\end{definition}

\begin{remark}
    Observe that, since $u$ is Lipschitz, assumption $w_u\in L^1_\loc(\mathring\Sigma)$ implies that $N_u$ is well defined and timelike almost everywhere.
\end{remark}

\begin{remark}
    For later reference, we write  the expression of the mean curvature with an explicit dependence on the metric $\sigma$.
    
    Computing the relation between the Levi-Civita connections $\nabla$ of and $D$, one deduces the relation between the corresponding Hessians:
\begin{align}\label{eq_relation_hessians}
    \nabla^2u = w^2 D^2 u + w^2\sigma(Du,D\ln\alpha)(\sigma-du^2) - 2d \ln\alpha\odot du.
\end{align}
Taking traces with respect to $g$ one has
\begin{align}\label{eq_relations_laplacians}
    \Delta_g u = \frac{w^2}{\alpha^2}\left( \Delta_\sigma u + w^2 D^2u(Du,Du) + (m-2)\sigma(Du,D\ln\alpha) \right).
\end{align}
On the other hand, at a given $p\in M$ consider the slice $S= \{\tau = \tau(p)\}$. Using \eqref{eq_nu_extri} and $d\balpha = \balpha^3 \bnabla^2\tau(\bnabla\tau,\cdot)$, we have at $p$
\begin{align*}
    \diver_MT &= \diver_{\g}T + \g(\bnabla_NT,N) \\ &= \diver_ST + w^2\left(\g(\bnabla_UT,U) + \g(\bnabla_TT,U)\right) \\
    &= H^S + w^2 \left( \balpha^2\second^S(\bnabla\bu,\bnabla\bu) + \g(\bnabla\balpha,\bnabla\bu) \right)
\end{align*}
where $\second^S$ and $H^S$ are the second fundamental form and mean curvature of $S$. 
By \eqref{eq_mean_curvature} and \eqref{eq_relations_laplacians} we therefore deduce that in local charts
\begin{align}\label{eq_mean_curvature_sigma_ij}
H_u = \alpha_u w_u g_u^{ij} u_{ij} + w_u B^k(x,u,du) u_k + w_u^{-1}H^S(x,u),   
\end{align}
where $u_{ij} = \partial^2_{ij} u$, $u_k = \partial_k u$ and the functions $B^k(x,r,p)$ are smooth in $T^*\Sigma$, in particular, they are nonsingular where $|du|_{\sigma_u} = 1$. Thus, the Lorentzian mean curvature operator is a quasi-linear elliptic second order operator. However, it is not uniformly elliptic because the eigenvalues of the principal symbol are controlled by $w$, which explodes as $|Du|\to 1$.
\end{remark}

\section{Consequences of hypothesis \ref{cauchy compatto}, and approximation}

\stoptoc

\subsection{Some consequences of \ref{cauchy compatto}}

The choice of a splitting time function $\tau$ induces a Riemannian metric on $\M$ given by 
\begin{equation}\label{eq_eucl_metric}
\ee = \g + 2 T_\flat \otimes T_\flat, 
\end{equation}
which, as in \cite{bartnik88}, we use to measure $C^k$ norms of tensors on $\M$: for a tensor $A$ on $\overline{D(\Sigma)}$ we set
    \begin{align*}
        \norm{A} = \max_{\overline{D(\Sigma)}} \sqrt{\ee(A,A)}, && \norm{A}_k = \sum_{j=0}^n \norm{\bnabla^jA}.
    \end{align*}
Notice that if $X$ is any vector field we have
\begin{align}\label{eq_bound_euclidean_norm}
    |\g(X,N)| \leq \norm{N}\norm{X} = \sqrt{2w^2 - 1}\norm{X} \leq \sqrt{2}w\norm{X}.
\end{align}
The following Lemma establishes some simple but important consequences of Hypothesis \ref{cauchy compatto}, that will be repeatedly used throughout the paper. 

\begin{lemma}\label{conseguenze di C}
    Suppose $\Sigma$ satisfies Hypothesis \ref{cauchy compatto} and choose a spitting time function $\tau$. Then there exists a constant $C= C(\Sigma,\tau)$ such that, for each $u,v\in\Y$, the following \emph{a priori} estimates hold.
    \begin{enumerate}[label=(P\arabic*), ref=P\arabic*]
        \item\label{C_1} $\norm{u}_{L^\infty(\Sigma)}\leq C$ and $C^{-1}\leq\norm{\alpha_u}_{L^\infty(\Sigma)}\leq C$.
        \item\label{C_2} $\norm{\sigma_u-\sigma_v}_{W^{1,\infty}(\Sigma,\sigma_v)}\leq C$ and $C^{-1}\sigma_v \le \sigma_u \le C\sigma_v \ $  as quadratic forms.
        \item\label{C_3} For any $1\leq p\leq \infty$ the spaces $L^p(\Sigma,\sigma_u)$ and $W^{1,p}(\Sigma,\sigma_u)$ do not depend on $u \in \Y$ and the respective norms are all equivalent, namely, there exists $C_p = C(\Sigma,\tau,p)$ such that for each measurable function $f$ and $u,v \in \Y$
        \[
        \begin{array}{c}
            C_p^{-1}\norm{f}_{L^p(\Sigma,\sigma_v)} \leq \norm{f}_{L^p(\Sigma,\sigma_u)} \leq C_p \norm{f}_{L^p(\Sigma,\sigma_v)} \\[0.3cm]
            C_p^{-1}\norm{f}_{W^{1,p}(\Sigma,\sigma_v)} \leq \norm{f}_{W^{1,p}(\Sigma,\sigma_u)} \leq C_p \norm{f}_{W^{1,p}(\Sigma,\sigma_v)}.
        \end{array}
        \]
        \item\label{C_4} $\Y$ is uniformly bounded in $C(\Sigma)$ and uniformly $\sigma_v$-Lipschitz for any $v\in\Y$.
        \item\label{C_5} $\Y$ is compact in $C(\Sigma)$.
        \item\label{C_6} $\Y$ is weakly compact in $W^{1,p}(\Sigma,\sigma_\varphi)$ for any $1<p<\infty$.
        \item\label{C_7} For each measurable $f$, 
        \begin{align}\label{eq_W22_norms}
        C^{-1} \left(\norm{f}^2_{W^{2,2}(\Sigma,\sigma_v)} - 1\right) \leq \norm{f}^2_{W^{2,2}(\Sigma,\sigma_u)} \leq C \left(\norm{f}^2_{W^{2,2}(\Sigma,\sigma_u)} + 1 \right).
        \end{align}
        In particular, boundedness in $W^{2,2}(\Sigma,\sigma_u)$ and $W^{2,2}_\loc(\Sigma, \sigma_u)$ does not depend on $u \in \Y$. If $\g$ is conformal to the product metric $-d\tau^2 + \pi^*\sigma$ on $\R \times S$, the spaces $W^{2,2}(\Sigma, \sigma_u)$ all have equivalent norms.
        \item\label{C_8} The Lie derivative $\lie_T\g$ and the Ricci tensor $\Ric$ of $\g$ enjoy the following bounds on $D(\Sigma)$:
        \begin{align*}
        |\lie_T\g(V,V)|&\leq C\g(T,V)^2 \\
        |\bnabla\lie_T\g(V,V,V)|&\leq C\g(T,V)^3 \\
        \Ric(V,V)&\geq -C \g(T,V)^2
        \end{align*}
        for any timelike $V\in\X(\overline{D(\Sigma)})$.
        \item\label{C_9} If $X$ is a continuous vector field on $\overline{D(\Sigma)}$, then 
        \begin{align*}
            \forall V\in \X(\overline{D(\Sigma)}) \ \text{timelike,} \qquad \abs{\g(X,V)} \leq C\abs{\g(T,V)}.
        \end{align*}
        In particular, if $u\in\Y\cap C^\infty(\Sigma)$ and $V=N_u$ we have
        \begin{align*}
            \abs{\g(X,N_u)} \leq C w_u.
        \end{align*}
    \end{enumerate}
\end{lemma}

\begin{notation}
    Hereafter we will denote any constant depending on $\Sigma$ and on the choice of a splitting time function $\tau$ with $C$. Any other dependency will be denoted by subscripts. If a constant does not depend on $\Sigma$, $\tau$ it will be denoted by $c$.
\end{notation}

\begin{proof}[Proof of Lemma \ref{conseguenze di C}]
    By the very definition of $\mathscr{Y}(\Sigma)$, every $M\in\mathscr{Y}(\Sigma)$ must lie within $\overline{D(\Sigma)}$, the compactness of which readily implies \eqref{C_1}.
    Choose a local chart $\{x^i\}$ on $\Sigma$ and define $\bar x^i \doteq \pi^* x^i$, $\bar x^0 \doteq \tau$. Then
    \begin{align}\label{eq_sigma_ij}
        \sigma_{ij} = F^*\bar\sigma_{ij} \qquad \partial_k\sigma_{ij} = F^*(\bpart_k\bar\sigma_{ij}+ \bu_k\bpart_0\bar\sigma_{ij})
    \end{align}
    readily yields
    \begin{align*}
        \norm{\sigma_v - \sigma_u}_{C(\Sigma)} \leq C, \qquad |u_k|\leq C \quad \forall k=1, \dots, m.
    \end{align*}
    and the second in \eqref{C_2}. Moreover, from
    \begin{align*}
        |\partial_k(\sigma^u_{ij}-\sigma^v_{ij})| &\leq |F^*_u(\bpart_k\bar\sigma_{ij}) - F^*_v(\bpart_k\bar\sigma_{ij})| + |F^*_u(\bu_k\bpart_0\bar\sigma_{ij}) - F^*_v(\bu_k\bpart_0\bar\sigma_{ij})|\leq C
    \end{align*}
    the first in \eqref{C_2} follows. Item  \eqref{C_3} is a direct consequence of \eqref{C_2}. Assertion \eqref{C_4} follows from \eqref{C_1} and the second in \eqref{C_2} applied to $D_v u$, where $u,v\in \Y$ and $D_v$ is the Levi-Civita connection of $\sigma_v$. Assertion \eqref{C_5} follows from the previous one by Ascoli-Arzelà.  Concerning \eqref{C_6}, since $\Y$ is bounded and closed in $W^{1,p}(\Sigma,\sigma_\varphi)$, which is reflexive for $1<p<\infty$, it is also weakly compact.
    To prove \eqref{C_7}, from  \eqref{eq_sigma_ij} it follows
    \begin{align*}
        (\Gamma^\sigma)^k_{ij} = F^*\left((\Gamma^{\bar\sigma})^k_{ij} + \frac{1}{2}\bar\sigma^{kl}\left(\bu_i \bpart_0\bar\sigma_{lj} + \bu_j\bpart_0\bar\sigma_{il} - \bu_l\bpart_0\bar\sigma_{ij}\right)\right).
    \end{align*}
    as a consequence, unless $\M$ is conformal to a product $\R\times S$, in which case $\bpart_0\bar\sigma_{ij} = 0$, the $W^{2,p}$ spaces are not equivalent. However, from the second in \eqref{eq_sigma_ij}, we deduce that for any measurable function $f$ and for $u,v\in\Y$
    \begin{align*}
        C^{-1} \left(|D^2_v f|_{\sigma_v}^2 - 1\right)\leq |D^2_u f|_{\sigma_u}^2 \leq C \left(|D^2_v f|_{\sigma_v}^2 + 1\right).
    \end{align*}
    Using \eqref{C_3}, the bound \eqref{eq_W22_norms} is thus satisfied.
    To prove \eqref{C_8}, by the compactness of $D(\Sigma)$ we have $\norm{\lie_T\g}_1\leq C$ and the same holds for $\Ric$. The inequalities in \eqref{C_8} readily follow from it. Finally, property \eqref{C_9} follows from \eqref{eq_bound_euclidean_norm}. 
\end{proof}

\begin{notation}
    For any $1<p<\infty$ and $0\le \beta<1$ we will write $L^p(\Sigma)$, $W^{1,p}(\Sigma)$, $W^{2,2}(\Sigma)$ and $C^{1,\beta}(\Sigma)$ leaving implicit the fact that we are considering the metric $\sigma_\varphi$: 
    \begin{align*}
        \norm{\cdot}_{W^{k,p}(\Sigma)} \doteq \norm{\cdot}_{W^{k,p}(\Sigma,\sigma_\varphi)} \qquad \norm{\cdot}_{C^{1,\beta}} \doteq \norm{\cdot}_{C^{1,\beta}(\Sigma,\sigma_\varphi)}.
    \end{align*}
\end{notation}

A useful remark is in order.

\begin{remark}
    Applying \eqref{C_8} to the unit normal of $M\in\mathscr{Y}(\Sigma)$ yields
    \begin{align}\label{eq_bounds_lie}
        |\lie_T\g(N,N)| + |\tr_M\lie_T\g| + |\diver_M T| \leq Cw^2 \qquad |\tr_M\bnabla_N\lie_T\g| \leq C w^3,
    \end{align}
    and
    \begin{align}\label{eq_bounds_ric}
        \Ric(N,N) \geq -C w^2.
    \end{align}
\end{remark}


\subsection{Approximating the mean curvature} For any $k\in\N$ and $\Sigma'\subseteq \Sigma$ define the set of Radon measure of class $C^k$ on $\Sigma'$ as
\begin{align*}
    \mathscr{C}^k(\Sigma') &\doteq \left\{ \mu \in \mathscr{M}(\Sigma) \quad : \quad \mu << dV_{\sigma_u} \quad \text{and} \quad \frac{d\mu}{dV_{\sigma_u}} \in C^k(\Sigma') \quad \forall u\in\Y\cap C^k(\Sigma') \right\}
\end{align*}
and the set of smooth measures as
\begin{align*}
    \mathscr{C}^\infty(\Sigma') \doteq \bigcap_{k\in\N} \mathscr{C}^k(\Sigma').
\end{align*}
Notice that $\tfrac{d\mu}{dV_{\sigma_u}}$ is $C^k(\Sigma')$ for one $u\in \Y\cap C^k(\Sigma')$ if and only it belongs to $C^k(\Sigma')$ for all $u\in\Y\cap C^k(\Sigma')$. We also define
\begin{align*}
    \mathscr{L}^p(\Sigma') &\doteq \left\{ \mu \in \mathscr{M}(\Sigma) \quad : \quad \mu << dV_{\sigma_u} \quad \text{and} \quad \frac{d\mu}{dV_{\sigma_u}} \in L^p(\Sigma') \quad \forall u\in\Y \right\}.
\end{align*}
For any $\delta>0$ and $\Sigma'\subseteq\Sigma$ set
\begin{align*}
    \Sigma'_\delta = \{ x\in\Sigma' \ : \ \di_{\sigma_\varphi}(x,\partial\Sigma') > \delta \}.
\end{align*}

We have the following approximation result, whose proof, depending on standard convolution procedures, is deferred to  Appendix \ref{Appe_approx}.

\begin{prop}\label{prop_appendix_sec}
    Assume \ref{cauchy compatto}. Let $\rho: (\Y, \|\cdot \|_{C(\Sigma)}) \to\cM(\Sigma)$ be a continuous map. There exists a sequence of functions
    \begin{align*}
        \rho_j \ : \ \Y \longrightarrow \mathscr{C}^\infty(\Sigma)
    \end{align*}
    such that the following hold for $j>>1$.
    \begin{enumerate}[label=(\roman*)]
        \item
        For any $\{u_j\}\subseteq\Y$ we have
        \begin{align*}
            u_j \to u \quad \text{in $C(\Sigma)$} \quad \Longrightarrow \quad \rho_j(u_j)\overset{*}{\rightharpoonup} \rho(u) \quad \text{in $\cM(\Sigma)$}.
        \end{align*}
        \item
        There exists a constant $C_\rho$ such that
        \begin{align*}
            \norm{\rho_j(u)}_{\cM(\Sigma)} \leq C_\rho \qquad \forall u\in \Y.
        \end{align*}
        \item
        For each $j$ there exists a constant $C_{\rho,j}$ such that
        \begin{align*}
            \norm{\frac{d\rho_j(u)}{dV_{\sigma_u}}}_{C^1(\Sigma)} \leq C_{\rho,j} \qquad \forall u\in\Y\cap C^1(\Sigma)
        \end{align*}
        and if $\rho$ is valued in $\mathscr{C}^1(\Sigma')$ for some $\Sigma'\subseteq\Sigma$, then  for every $\delta>0$ there exists a constant $C_{\Sigma',\delta}$ such that
        \begin{align*}
            \norm{\frac{d\rho_j(u)}{dV_{\sigma_u}}}_{C^1(\Sigma'_\delta)} \leq C_{\Sigma',\delta} \norm{\frac{d\rho(u)}{dV_{\sigma_u}}}_{C^1(\Sigma')} \qquad \forall u\in\Y\cap C^1(\Sigma).
        \end{align*}
        \item
        If $\rho$ takes values in $\mathscr{L}^p(\Sigma')$ for some $\Sigma' \subseteq \Sigma$ and $p\in[1,\infty)$, then for every $\delta>0$ there exists a constant $C_{p,\Sigma',\delta}$ such that 
        \begin{align*}
            \norm{\frac{d\rho_j(u)}{dV_{\sigma_u}}}_{L^p(\Sigma'_\delta)} \leq C_{p,\Sigma',\delta}\norm{\frac{d\rho(u)}{dV_{\sigma_u}}}_{L^p(\Sigma')} \qquad\forall u\in\Y.
        \end{align*}
    \end{enumerate}
    Furthermore if $X\in\X(\overline{D(\Sigma)})$ is a continuous vector field, there exists a sequence of smooth vector fields $\{X_j\}$ and a constant $\Lambda\geq 0$ such that 
    \begin{align}
        X_j\to X \quad \text{in $C(\overline{D(\Sigma)})$}
    \end{align}
    and
    \begin{align}\label{eq_Hp_Lambda}
        \abs{\g(X_j,N_{u})} \leq \Lambda w_{u} \qquad \forall u\in\Y
    \end{align}
    where $w_{u} = - \g(T,N_{u})$ and $N_{u}$ is the future pointing unit normal to the graph of $u$. Moreover, if $X$ is $C^1$ on a compact subset $K\subseteq\overline{D(\Sigma)}$ then the $C^1$ norm of each $X_j$ on $K$ satisfies
    \begin{align}
        \norm{X_j}_{1,K} \leq \Lambda.
    \end{align}
        
\end{prop}

\resumetoc

\section{Outline of the proof of Theorem \ref{main result}}\label{strategy}

To produce a weak solution to \eqref{eq_BI} we will proceed as follows. Let $(\rho_j,X_j)$ the approximation sequence given by Proposition \ref{prop_appendix_sec} for the pair $(\rho,X)$. \\[0.2cm]
\textbf{Step 1.} \textit{Existence of approximating solutions.} \\ We show that the approximate Dirichlet problem
\begin{align}\tag{$\mathrm{PMC}_j$}\label{eq_approx_problem}
    \begin{cases}
        H_{u_j} = \rho_j(u_j) + \g(X_j,N_{u_j}) \\
        u_j\in \Y.
    \end{cases}
\end{align}
admits a smooth, spacelike, weak solution $u_j$. This follows from the gradient estimates of Bartnik in \cite{bartnik84} and a fixed point argument.
\\[0.2cm]
\textbf{Step 2.} \textit{Existence of a limit.} \\ We show that there exists a function $u\in\Y\cap W_\loc^{2,2}(\mathring\Sigma\backslash E)$ such that, up to a subsequence, $u_j\to u$ in $W^{1,p}(\Sigma)$ for each $p<\infty$. The proof of this fact relies on two different estimates:
\begin{itemize}
    \item first, we obtain the bound $\norm{w_j}_{L^1(\Sigma)}\leq C_{\rho,X}$ for the tilt function $w_j$ of $u_j$, see Proposition \ref{stima di energia}. In view of the link between $w_j$ and the energy density in the context of Born-Infeld's theory, we often call $\norm{w_j}_{L^1(\Sigma)}\leq C_{\rho,X}$ the \emph{energy estimate};
    \item second, in every domain $\Sigma\subseteq \mathring \Sigma\backslash E$ where $\rho\in L^2$ we show the bound 
    \begin{align*}
        \int_{\Sigma'_\epsilon} \abs{\second_j}^2 \, dV_{g_j} \leq C_{\rho,X,\epsilon}\left(\norm{\rho}_{L^1(\Sigma)} + \norm{\rho}_{L^2(\Sigma')} + 1\right),
    \end{align*}
    for the second fundamental form $\second_j$ of the graph of $u_j$, see Proposition \ref{stima di seconda}.
\end{itemize}
The latter provides local uniform $W^{2,2}_\loc$ estimates on $\mathring\Sigma\backslash E$, necessary to guarantee the $W^{1,p}$ convergence of $u_j$ to $u$. In fact, the second fundamental form estimate leads to much stronger bounds than merely $W^{2,2}_\loc$ ones, as  it accounts for the second in \eqref{eq_higherinteg}.
\\[0.2cm]
\textbf{Step 3. }\textit{Convergence of the integral identity.} \\ Every $u_j$ being a weak solution, the following integral identity holds for any $\eta\in C_c^1(\mathring\Sigma)$:
\begin{align*}
    \int_\Sigma \eta \frac{\alpha_j^m}{w_j}\diver_{M_j} T \, dV_j = \int_\Sigma w_j\alpha_j^{m-1}\sigma_j(du_j,d\eta) \, dV_j + \int_\Sigma \eta\alpha_j^m \, d\rho_j + \int_\Sigma\eta\alpha_j^m\g(X_j,N_j) \, dV_j
\end{align*}
where $M_j = F_{u_j}(\Sigma)$, $dV_j = dV_{\sigma_j}$, $\rho_j = \rho_j(u_j)$ and so on. To conclude that $u$ is a weak solution with respect to $(\rho,J)$, we shall prove that, along a subsequence,
\begin{align*}
    \lim_{j\to \infty} \int_\Sigma \eta \frac{\alpha_j^m}{w_j}\diver_{M_j}T \, dV_j &= \int_\Sigma \eta w^{-1}\alpha^m \diver_MT \, dV_\sigma \\ 
    \lim_{j\to\infty} \int_\Sigma w_j\alpha_j^{m-1}\sigma_j(du_j,d\eta) \, dV_j &= \int_\Sigma w\alpha^{m-1}\sigma(du,d\eta) \, dV_\sigma \\
    \lim_{j\to\infty} \int_\Sigma \eta\alpha_j^m \, d\rho_j + \int_\Sigma \eta\alpha_j^m\g(X_j,N_j) \, dV_j &= \int_\Sigma \eta\alpha^m \, d\rho + \int_\Sigma \eta \alpha^m\g(X,N) \, dV_\sigma.
\end{align*}
While for the first limit the convergence is straightforward, the matter is more subtle for the last two. Consider for instance the second limit. By Vitali convergence theorem, its validity follows once we know that the family $\{w_j\}$ is locally uniformly integrable. This is achieved in two steps:

\textit{Higher integrability.} We prove that, \emph{in dimension} $m+1 = 3$, the second fundamental form estimate in Step 2 guarantees the uniform integral bound
\begin{align}\label{eq_strategy_high_int}
    \int_{\Sigma_\epsilon'} w_j \ln (1+w_j) \, dV_j \leq C_{\rho,X,\epsilon}\left(\norm{\rho}_{L^1(\Sigma)} + \norm{\rho}_{L^2(\Sigma')} + 1\right)
\end{align}
on every $\Sigma'\Subset\mathring\Sigma$ on which $\rho$ is $L^2$. By de la Vallée-Poussin's theorem, \eqref{eq_strategy_high_int} is enough to conclude the local uniform integrability of $\{w_j\}$ on $\mathring\Sigma\backslash E$.

\textit{Removable singularities.} In order to have local uniform integrability on the whole of $\mathring\Sigma$ we will prove the following Removable Singularity Theorem \ref{singolarità rimovibili}: if $E\Subset \mathring\Sigma$ is a compact set of vanishing $1$-dimensional Hausdorff measure, then
\begin{align*}
        \parbox{4cm}{$\{w_j\}$ is locally uniformly integrable in $\mathring\Sigma\backslash E$} \ \ \iff \ \ \parbox{4cm}{$\{w_j\}$ is locally uniformly integrable in $\mathring\Sigma$.} 
\end{align*}
This completes the proof of the existence statement in Theorem \ref{main result}. 
\\[0.2cm]
\textbf{Step 4. }\textit{Absence of light segments.} \\
As pointed out at the end of the Introduction, this step depends on the dimension restriction and indeed it fails in ambient dimension $m+1 \geq 5$. We prove Step 4 by exploiting the higher integrability inequality \eqref{eq_strategy_high_int}, thus avoiding the variational arguments in \cite{BIMM} which are hardly applicable to our setting. 
\\[0.2cm]
\textbf{Step 5.} \textit{Higher regularity.} \\ This step, which is \textit{(iv)} in Theorem \ref{main result}, exploits one of the main Theorems by Bartnik \cite{bartnik88}.\\[0.2cm]

We underline that the higher integrability estimate in Step 3 and the no-light-segment Step 4 and the only conclusions in the proof which require a dimensional restriction. We think that establishing their validity (or failure) in ambient dimension $4$ is an intriguing open problem.

\section{Main estimates}

From now on, any constant depending on the chosen smooth, compact, spacelike hypersurfaces $\Sigma$ satisfying \ref{cauchy compatto} and on the splitting time function $\tau$ will be denoted by $C$. Further relevant dependencies will be marked by subscripts. In the course of the proofs their value may vary from line to line.

Consider $\rho\in C^\infty(\Sigma)$ and a smooth vector field $X\in\X(\M)$. All the results in this section will regard smooth classical solutions to the Dirichlet problem
\begin{align}\tag{PMC'}\label{eq_SP}
    \begin{cases}
        H_u = \rho + \g(X,N_u)\\
        u\in\Y
    \end{cases}
\end{align}
which is the model for the approximate problems \eqref{eq_approx_problem}.



\subsection{Energy bound (tilt function estimate)} As a first step, we establish an integral estimate for the tilt function that will be of basic importance for the next arguments. The estimate holds in any dimension $m \ge 2$.


\begin{prop}\label{stima di energia}
    There exists constants $C_X$ such that for each smooth spacelike solution $u \in \Y \cap C^\infty(\Sigma)$ to \eqref{eq_SP} it holds
    \begin{align*}
        \int_\Sigma w_u \, dV_{\sigma_u} \le C_X \left[ 1 + \| \rho\|_{L^1(\Sigma)}\right].
    \end{align*}
\end{prop}

\begin{proof}
    Write $\sigma$, $g$, $H$, $\alpha$ to denote, respectively, $\sigma_u$, $g_u$ $H_u$ and $\alpha_u$. By \eqref{eq_smooth_weak_solution} with  $H = \rho + \g(X,N)$ and by density, for any $\eta\in \lip_c(\Sigma)$ we have
    \begin{align*}
        \int_\Sigma g(\alpha\nabla u, \nabla\eta) \, dV_g = \int_\Sigma \eta \diver_MT \, dV_g -\int_\Sigma \eta \rho w \, dV_g - \int_\Sigma\eta\g(X,N) \, dV_g.
    \end{align*}
    Using \eqref{eq_w_nabla_u}, \eqref{volumi} and \eqref{eq_bounds_lie}, that is,
    \begin{align*}
        w^2 = \alpha^2|\nabla u|^2 + 1 \qquad \alpha^m \, dV_\sigma = w \, dV_g \qquad \diver_MT \leq Cw^2,
    \end{align*}
    and Lemma \ref{conseguenze di C} we have
    \begin{align*}
        \int_\Sigma g(\alpha\nabla u,\nabla\eta) \, dV_g &\leq \int_\Sigma |\eta|(\alpha^2|\nabla u|^2 + 1) \, dV_g + \int_\Sigma|\eta\rho| \, dV_g + C_X\int_\Sigma |\eta| w \, dV_g \\ &\leq C_{X}\left(\int_\Sigma \alpha|\eta||\nabla u|^2 \, dV_g + \norm{\eta}_{L^\infty(\Sigma)}\left(\norm{\rho}_{L^1(\Sigma)} + 1\right)\right).
    \end{align*}
    We choose the test function
    \begin{align*}
        \eta \doteq e^{\lambda u}(u-\varphi)_{+}
    \end{align*}
    for some $\lambda>0$ to be specified later so that, again by \eqref{C_1} in Lemma \ref{conseguenze di C} we can find a constant $C_{X,\lambda}$, such that
    \begin{align}\label{prima ineq}
        \int_\Sigma
         g(\alpha \nabla u, \nabla\eta) \, dV_g - C_{X}\int_\Sigma \eta\alpha|\nabla u|^2 \, dV_g \leq C_{X,\lambda}\left(\norm{\rho}_{L^1(\Sigma)} + 1\right).
    \end{align}
    Let $\theta = \theta(\varphi)$ satisfy
    \begin{align*}
        |d\varphi|_{\sigma_\varphi}\leq 1-4\theta \qquad \text{on $\Sigma$.} 
    \end{align*}
    Since the coefficients of $\sigma_u$ depend continuously on $u$, there exists $\delta=\delta(\Sigma,\tau)$ such that for every $u\in\Y$ it holds
    \begin{align}\label{next to u}
        |D\varphi|=|d\varphi|_{\sigma_u}\leq 1-2\theta \qquad \text{on $\{|u-\varphi|\leq \delta\}$ }.
    \end{align} 
    Define
    \begin{align*}
        \Sigma^+ \doteq \{u\geq\varphi\}, \qquad \Sigma^+_\delta\doteq\{u\geq \varphi +\delta\}, \qquad U_\delta\doteq\{\varphi\leq u\leq \varphi+\delta\},
    \end{align*}
    so that $\Sigma^+ = \Sigma_\delta^+  \cup  U_\delta$, we estimate the left-hand-side of the inequality \eqref{prima ineq} as follows
    \begin{align*}
        ({\rm LHS}) \doteq\lambda \int_{\Sigma^+} \eta\alpha|\nabla u|^2 \, dV_g &+ \int_{\Sigma^+} \alpha e^{\lambda u}(|\nabla u|^2 - g(\nabla u,\nabla \varphi)) \, dV_g \\ &- C_X\int_{\Sigma^+}\eta\alpha|\nabla u|^2  \, dV_g\\
        \geq (\lambda-C_X)\int_{\Sigma^+} \eta\alpha|\nabla u|^2 \, dV_g &+ \int_{U_\delta}\alpha e^{\lambda u}(|\nabla u|^2-|\nabla u||\nabla \varphi|) \, dV_g \\
        &+\int_{\Sigma^+_\delta} \alpha e^{\lambda u}(|\nabla u|^2-|\nabla u||\nabla \varphi|) \, dV_g.
    \end{align*}
    We introduce the set $E_\theta\doteq\{|Du|\geq 1-\theta\}$ to estimate the three integrals in the right-hand side as follows: assuming $\lambda \geq C_X$, the first one becomes
    \begin{align*}
\int_{\Sigma^+} \eta\alpha|\nabla u|^2 \, dV_g 
&= \int_{U_\delta} \eta\alpha|\nabla u|^2 \, dV_g + \int_{\Sigma^+_\delta} e^{\lambda u}(u-\varphi)\alpha|\nabla u|^2 \, dV_g \\
&\geq \delta \int_{\Sigma^+_\delta} \alpha e^{\lambda u}|\nabla u|^2 \, dV_g \\
&\geq \delta \int_{\Sigma^+_\delta \cap E_\theta} \alpha e^{\lambda u}|\nabla u|^2 \, dV_g .
\end{align*}
    while, for the second one, from 
    \begin{align*}
        |\nabla \varphi|\leq |D\varphi|\frac{w}{\alpha}\leq (1-2\theta)\frac{w}{\alpha} &\qquad \text{on $U_\delta$ by \eqref{next to u}}, \\ |\nabla u|=|Du|\frac{w}{\alpha}\geq (1-\theta)\frac{w}{\alpha} &\qquad \text{on $E_\theta$}, 
    \end{align*}
    we have the inequalities
    \begin{align*}
|\nabla u|^2 - |\nabla u||\nabla \varphi|
  &\geq \left(1 - \frac{1-2\theta}{1-\theta}\right)|\nabla u|^2
   = \frac{\theta}{1-\theta}|\nabla u|^2
  &\qquad \text{on $U_\delta \cap E_\theta$,} \\[6pt]
|\nabla u|^2 - |\nabla u||\nabla \varphi|
  &\geq -\frac{1}{\alpha^2}\frac{(1-2\theta)(1-\theta)}{1-(1-\theta)^2}\geq -C_\theta
  &\qquad \text{on $U_\delta \cap E_\theta^c$.} 
\end{align*}
    and therefore
    \begin{align*}
        \int_{U_\delta}\alpha e^{\lambda u}(|\nabla u|^2-|\nabla u||\nabla \varphi|) \, dV_g & \ge \frac{\theta}{1-\theta}\int_{U_\delta\cap E_\theta} \alpha e^{\lambda u}|\nabla u|^2 \, dV_g \\ &-C_\theta\int_{U_\delta\cap E_\theta^c} \alpha e^{\lambda u} \, dV_g.
    \end{align*}
    Similarly,
    \begin{align*}
        |\nabla u|^2 - |\nabla u||\nabla\varphi| \geq \left(1-\frac{|D\varphi|}{1-\theta}\right)|\nabla u|^2 =- C_{\varphi,\theta}|\nabla u|^2 &\qquad \text{on $\Sigma^+\cap E_\theta$} \\
        |\nabla u|^2 - |\nabla u||\nabla\varphi| \geq -\frac{1-\theta}{1-(1-\theta)^2}\frac{|D\varphi|}{\alpha^2} \geq -C_{\varphi,\theta} &\qquad \text{on $\Sigma^+\cap E_\theta^c$}
    \end{align*}
    and the third integral can be estimated as follows:
    \begin{align*}
        \int_{\Sigma^+_\delta} \alpha e^{\lambda u}(|\nabla u|^2-|\nabla u||\nabla \varphi|) \, dV_g \geq &-C_{\varphi,\theta}\int_{\Sigma^+_\delta\cap E_\theta}\alpha e^{\lambda u}|\nabla u|^2 \, dV_g \\
        &-C_{\varphi,\theta}\int_{\Sigma^+_\delta\cap E_\theta^c} \alpha e^{\lambda u} \, dV_g.
    \end{align*}
    Summarizing, we infer the following inequality:
    \begin{align*}
        ({\rm LHS}) &\geq \left(\delta(\lambda-C_X)-C_{\varphi,\theta}\right)\int_{\Sigma^+_\delta\cap E_\theta} \alpha e^{\lambda u}|\nabla u|^2 \, dV_g  \\ &+ \frac{\theta}{1-\theta}\int_{U_\delta\cap E_\theta} \alpha e^{\lambda u}|\nabla u|^2 \, dV_g
        -C_{\varphi,\theta}\int_{\Sigma^+ \cap E_\theta^c} \alpha e^{\lambda u} \, dV_g.
    \end{align*}
    Choosing $\lambda$ sufficiently large to satisfy $\delta(\lambda-C_X) - C_{\varphi,\theta} \geq \frac{\theta}{1-\theta}$, rearranging we deduce from \eqref{prima ineq} the inequality 
    \begin{align*}
        \int_{\Sigma^+\cap E_\theta
        }\alpha |\nabla u|^2 \, dV_g &\leq e^{-\lambda\inf u}\frac{1-\theta}{\theta}\left(({\rm LHS})+C_{\varphi,\theta}\int_{\Sigma^+_\delta\cap E_\theta^c} \alpha e^{\lambda u} \, dV_g \right) \\
        &\leq C_{X,\theta} \left(\norm{\rho }_{L^1(\Sigma)} + 1\right).
    \end{align*}
    Repeating the argument with the choice $\eta=e^{-\lambda u}(u-\varphi)_-$, we eventually get
    \begin{align*}
        \int_{\Sigma^-\cap E_\theta} \alpha|\nabla u|^2 \, dV_g \leq C_{X,\theta} \left(\norm{\rho }_{L^1(\Sigma)} + 1\right)
    \end{align*}
    where $\Sigma^-=\{u\leq \varphi\}$, while on $E_\theta^c$ we have the straightforward estimate
    \begin{align*}
        \int_{E_\theta^c}\alpha|\nabla u|^2 \, dV_g \leq \int_{E_\theta^c} \alpha^{m+1}|\nabla u|^2 \, dV_\sigma \leq C \frac{(1-\theta)^2}{1-(1-\theta)^2}|\Sigma| = C_\theta|\Sigma|
    \end{align*}
    thus we obtain
    \begin{align*}
        \int_\Sigma \alpha|\nabla u|^2 \, dV_g \leq C_{X,\theta} \left(\norm{\rho }_{L^1(\Sigma)} + 1\right)
    \end{align*}
    and the conclusion follows by $w^2 = \alpha^2|\nabla u|^2 + 1$ and Lemma \ref{conseguenze di C}.
\end{proof}

\subsection{Removable singularities}

Recall that a sequence $\{w_j\}$ is \emph{uniformly integrable} in $\Sigma'\subseteq \Sigma$ with respect to a measure $\mu$ if for every $\epsilon > 0$ there exists $\delta > 0$ such that
for every measurable set $A \subseteq \Sigma'$ we have
\[
  \mu(A)<\delta \quad \Longrightarrow \quad \sup_j \int_A |w_j| \, d\mu < \epsilon.
\]
The sequence $\{w_j\}$ is said to be \emph{locally uniformly integrable} in an open subset $U\subseteq\mathring\Sigma$ if it uniformly integrable in every compact subset of $U$. 

Note that, by Proposition \ref{conseguenze di C}, (local) uniform integrability with respect to $dV_{\sigma_u}$ is equivalent to that with respect to $dV_{\sigma_v}$ for all $u,v \in \Y$. Therefore, hereafter, we shall refer to any of these notions simply as (local) uniform integrability.

In this subsection we show that if a sequence $\{w_j\}$ as in Section \ref{strategy} is locally uniformly integrable in $\mathring\Sigma\backslash E$ and the compact subset $E$ has vanishing $1$-dimensional Hausdorff measure, then local uniform integrability holds in the whole of $\mathring\Sigma$. The result holds in any dimension $m \ge 2$, and relies on the following Proposition.

\begin{prop}\label{crescita sulle bolle}
    Let $u\in \Y\cap C^\infty(\Sigma)$ be a smooth solution to \eqref{eq_SP}. For any fixed $x\in\Sigma$, set
    \begin{align*}
        I(s)\doteq \int_{B_s(x)} w_u \, dV_{\sigma_u}
    \end{align*}
    where $B_s(x)$ is the ball with respect to the metric $\sigma_u$. Then there exists a constant $C_X$ such that
    \begin{align*}
        I(s)\leq C_Xs\left(\frac{I(r)}{r}+ \int_{B_r(x)} |\rho| \, dV_{\sigma_u}+r^{m-1}\right)
    \end{align*}
    for $s<r < \di_\sigma(x,\partial\Sigma)$. 
\end{prop}
\begin{proof}
    Throughout the proof we will omit the subscript $u$. In \eqref{eq_smooth_weak_solution_sigma} chose as a test function $\eta (u - u(x))$, with $\eta\in C_c^\infty(\mathring\Sigma)$ non-negative: using $w^2|Du|^2 = w^2 - 1$ we have
    \begin{align*}
        \int_\Sigma \eta w\alpha^{m-1} \, dV_\sigma = &- \int_\Sigma (u - u(x)) \sigma(du,d\eta)w\alpha^{m-1} \, dV_\sigma - \int_\Sigma \eta (u - u(x)) H\alpha^m \, dV_\sigma \\ &+ \int_\Sigma \eta (u - u(x)) w^{-1}\alpha^m \diver_MT \, dV_\sigma + \int_\Sigma \eta w^{-1}\alpha^{m-1} \, dV_\sigma.
    \end{align*}
    Let $R = \operatorname{d}_\sigma(x,\cdot)$ be the distance from $x$ in the metric $\sigma = \sigma_u$. Since \ref{cauchy compatto} holds and $|Du|<1$, we can apply the estimates
    \begin{align*}
        C^{-1}\leq \alpha\leq C \qquad \diver_M T \leq Cw^2 \qquad |u - u(x)| \leq R
    \end{align*}
    and write
    \begin{align*}
        \int_\Sigma \eta w\alpha^{m-1} \, dV_\sigma &\leq \int_\Sigma R|D\eta| w\alpha^{m-1} \, dV_\sigma + \int_\Sigma \eta R |H|\alpha^m \, dV_\sigma \\ 
        &+ C\left(\int_\Sigma \eta R w\alpha^{m-1} \, dV_\sigma + \int_\Sigma \eta \, dV_\sigma\right).
    \end{align*}
    For fixed \(0 < \zeta < r < \operatorname{d}_\sigma(x,\partial\Sigma)\) and \(\epsilon>0\) small enough so that $\zeta+\epsilon < r$ we make the following choice for \(\eta\):
    \begin{align*}
        \eta_\epsilon(y)\doteq\left(\min\left\{1,\frac{\zeta+\epsilon-R(y)}{\epsilon}\right\}\right)_+
    \end{align*}
    so that
    \begin{align*}
        \supp \eta_\epsilon = B_{\zeta+\epsilon}, \qquad |D\eta_\epsilon| = \frac{1}{\epsilon} \chi_{B_{\zeta+\epsilon}\backslash B_\zeta}, \qquad \eta_\epsilon \overset{\epsilon\to 0}{\longrightarrow} \chi_{B_\zeta},
    \end{align*}
    where balls are centred at $x$. Observe that, by 
    \eqref{C_9} in Lemma \ref{conseguenze di C}, $\abs H \leq \abs\rho + C_X w$ and by the choice of $\eta_\epsilon$ we have $R\leq r$, whence
    \begin{align*}
        \int_{B_{\zeta+\epsilon}}\eta_\epsilon w \alpha^{m-1} \, dV_\sigma &\leq \int_{B_{\zeta+\epsilon}} R|D\eta| w\alpha^{m-1} \, dV_\sigma + \int_{B_{\zeta+\epsilon}} \eta R\abs\rho \alpha^m \, dV_\sigma \\ &+ C_X\left(\int_{B_{\zeta+\epsilon}} \eta R w \alpha^{m-1} \, dV_\sigma + |B_{\zeta+\epsilon}|\right).
    \end{align*}
    Letting \(\epsilon\to 0\) and applying the coarea formula we have
    \begin{align*}
        \int_{B_\zeta} w\alpha^{m-1} \, dV_\sigma &\leq \zeta \int_{\partial B_\zeta} w\alpha^{m-1} \, d\mathscr{H}^{m-1}_\sigma + \int_0^\zeta t\int_{\partial B_t} |\rho|\alpha^m \, d\mathscr{H}_\sigma^{m-1} dt \\ 
        &+ C_X\left(\zeta\int_{B_\zeta} w\alpha^{m-1} \, dV_\sigma + |B_\zeta|\right)
    \end{align*}
        so if we set
    \begin{align*}
        f(\zeta) = \int_{B_\zeta} w\alpha^{m-1} \, dV_\sigma,\qquad
        h(t) = \int_{\partial B_t} |\rho|\alpha^m \, d\mathscr{H}^{m-1}_\sigma
    \end{align*}
    and observe that \(f'(\zeta) = \int_{\partial B_\zeta} w\alpha^{m-1} \, d\mathscr{H}_\sigma^{m-1}\) by the coarea formula, the above inequality rewrites as
    \begin{align*}
        f(\zeta) \leq \zeta f'(\zeta) + \int_0^\zeta th(t) \, dt + C_X\left(\zeta f(\zeta) + |B_\zeta| \right).
    \end{align*}
    Rearranging terms and writing $C=C_X$ for convenience, 
    \begin{align*}
        -\frac{d}{d\zeta}\left(e^{C\zeta}\frac{f(\zeta)}{\zeta}\right) \leq \frac{e^{C\zeta}}{\zeta^2}\left(\int_0^\zeta th(t) \, dt + C|B_\zeta|\right),
    \end{align*}
    thus integrating on $[s,r]$ yields
    \begin{align*}
        -e^{Cr}\frac{f(r)}{r} + e^{Cs}\frac{f(s)}{s} &\leq \int_s^r \frac{e^{C\zeta}}{\zeta^2}\left(\int_0^\zeta th(t) \, dt + C|B_\zeta|\right) \, d\zeta \\ 
        &\leq Ce^{C\operatorname{diam}_\sigma(\Sigma)} \left(\int_0^r t h(t) \int_{\max\{s,t\}}^r \frac{d\zeta}{\zeta^2} \, dt + \int_s^r \zeta^{m-2} \, d\zeta\right) \\
        &\leq C \left(\int_0^r th(t)\left[-\frac{1}{\zeta}\right]_t^r \, dt + r^{m-1} \right) \\
        &\leq C\left(\int_0^r th(t)\frac{1}{t} \, dt + r^{m-1}\right) \\
        &= C\left(\int_0^r\int_{\partial B_t}|\rho|\alpha^m \, d\mathscr{H}_\sigma^{m-1} \, dt + r^{m-1}\right).
    \end{align*}
    Estimating again $e^{Cr}$ with $e^{C\operatorname{diam}_\sigma\Sigma}$ and $e^{-Cs}$ with $1$ we obtain
    \begin{align*}
        \frac{f(s)}{s} \leq C_Xs \left(\frac{f(r)}{r} + \int_{B_r} |\rho| \, dV_\sigma  + r^{m-1}\right),
    \end{align*}
    the constant $C$ depending only on $\Sigma$ and $\tau$. This gives the thesis, for
    \begin{align*}
        C^{-1}\int_{B_s} w \, dV_\sigma  \leq \int_{B_s} w\alpha^{m-1} \, dV_\sigma \leq C\int_{B_s} w \, dV_\sigma
    \end{align*}
    holds for all $s$ thanks to \ref{cauchy compatto}.
\end{proof}

We are ready to prove our main removable singularity theorem. 

\begin{theorem}\label{singolarità rimovibili}
    Assume \ref{cauchy compatto} holds, let $(\rho_j,X_j)$ as in Proposition \ref{prop_appendix_sec} and let $u_j \in \Y \cap C^\infty(\Sigma)$ be a spacelike solution to 
    \begin{equation*}
        H_{u_j}= \frac{d\rho_j(u_j)}{dV_{\sigma_{u_j}}} + \g( X_j, N_{u_j}).
    \end{equation*}
    Set
    \begin{align*}
        \sigma_j = \sigma_{u_j} \qquad w_j = \frac{1}{\sqrt{1-|du_j|_{\sigma_j}^2}}.
    \end{align*}
    Then, there exists a constant $C_{\rho,X}$ such that for each $x\in \mathring\Sigma$ and each $r<r_x\doteq\frac{1}{2}\inf_j\di_{\sigma_j}(x,\partial\Sigma)$ it holds
    \begin{align*}
        \int_{B_{r}^{\sigma_j}(x)} w_j \, dV_j \leq \frac{C_{\rho,X}}{r_x}r \qquad \forall \,  r\in(0,r_x).
    \end{align*}
    Moreover, if $E\Subset\mathring\Sigma$ is a compact subset such that $\mathscr{H}^1(E)=0$, then
    \begin{align*}
        \parbox{4cm}{$\{w_j\}$ is locally uniformly integrable in $\mathring\Sigma\backslash E$} \ \ \iff \ \ \parbox{4cm}{$\{w_j\}$ is locally uniformly integrable in $\mathring\Sigma$.} 
    \end{align*}
\end{theorem}


\begin{proof}
    Since $B_{r_x}(x)\Subset \mathring\Sigma$ we can apply Proposition \ref{crescita sulle bolle}, the energy estimate in Proposition \ref{stima di energia} and \textit{(ii)} in Proposition \ref{prop_appendix_sec} to deduce
    \begin{align*}
        \int_{B^{\sigma_j}_{r}(x)} \, w_j \, dV_j &\leq r C_X \left(\frac{1}{r_x}\int_{B^{\sigma_j}_{r_x}(x)} \,  w_j \, dV_j + \int_{B_{r_x}(x)} |\rho_j(u_j)| \, dV_j + r_x^{m-1} \right) \\
        &\leq \frac{C_{\rho,X}}{r_x}r
    \end{align*}
    for every $r\in (0,r_x)$. This proves the first statement in the theorem. Next, let $\Sigma'\Subset \mathring{\Sigma}$ and define
    \begin{align*}
        r_0 = \frac{1}{8}\inf_{u\in\Y} \di_{\sigma_u}(\Sigma',\partial\Sigma).
    \end{align*}
    which is positive by Lemma \ref{conseguenze di C}. Combining the first part of our theorem with Lemma \ref{conseguenze di C} we also deduce that
    \begin{align*}
        \int_{B_{r}(x)} \, w_j \, dV_j \leq Cr \qquad \forall x\in B_{r_0}(\Sigma') , \,\forall r<r_0,
    \end{align*}
    where $B_r$ is the geodesic ball in the metric $\sigma_\varphi$ and $C=C_{\rho, X, \Sigma'}$. Let $\epsilon>0$ and choose $\lambda\in(0,\epsilon/C)$. Since $\mathscr{H}^1(E)=0$, we can pick a finite covering $\{B_i\}_{i=1}^N = \{B_{r_i}(x_i)\}_{i=1}^N$ of $E\cap \overline\Sigma'$ with $\sigma_\varphi$-balls of radii $r_i\leq r_0$ and centred at points $x_i\in B_{r_0}(\Sigma')$ satisfying $\sum_i r_i \leq \lambda$. Also, let $\delta=\delta(\epsilon)>0$ be as in the definition of uniform integrability of $\{w_j\}$ in $\Sigma'\backslash\bigcup_i B_i$. Then for every measurable subset $A\subseteq \Sigma'$ such that $|A|_{\sigma_\varphi}<\delta$ it holds
    \begin{align*}
        \int_{A} w_j \, dV_{\sigma
        _\varphi} = \int_{A\backslash\bigcup_i B_i } w_j \, dV_{\sigma
        _\varphi} + \int_{A\cap \bigcup_iB_i} w_j \, dV_{\sigma_\varphi} \leq \epsilon + \int_{A\cap \bigcup_i B_i} w_j \, dV_{\sigma
        _\varphi}.
    \end{align*}
    On the other hand, 
    \begin{align*}
        \int_{A\cap \bigcup_i B_i} w_j \, dV_{\sigma_\varphi} \leq \sum_{i=1}^N \int_{B_{r_i}(x_i)} w_j \, dV_{\sigma
        _\varphi} \leq C\sum_{i=1}^N r_i \leq  C\lambda < \epsilon.
    \end{align*}
    Summarizing
    \begin{align*}
        |A|_{\sigma_\varphi}<\delta \quad \Longrightarrow \quad \int_{A} w_j \, dV_{\sigma
        _\varphi} \leq 2\epsilon
    \end{align*}
    which was to be proved.
\end{proof}

\subsection{Second fundamental form estimate}\label{sec_secondfund}

In this section we establish a local integral estimate for the squared norm of the second fundamental form. This is a key tool for it ensures a $W^{2,2}_\loc$ estimate for the approximating sequence, Corollary \ref{Dlnw e D^2u}, and plays a role in proving uniform integrability as well (Proposition \ref{higher integrability}). All that appears in this section holds in any dimension $m \ge 2$.

Our starting point is the Jacobi equation computed in \cite[Proposition 2.1]{bartnik84}:
\begin{align*}
    \Delta_g w = w\left(|\second|^2+\Ric(N,N)\right)-g(\nabla H,T^{\top})+T(H_T)
\end{align*}
where $T(H_T)$ is the derivative of the mean curvature of the deformations of $M$ in direction $T$, and satisfies
\begin{align}\label{T(H_T)}
    T(H_T)&=\frac{1}{2}(\bnabla_N\lie_T\g)(e_i,e_i)-(\bnabla_{e_i}\lie_T\g)(N,e_i)+\\\nonumber &-\frac{1}{2}H\lie_T\g(N,N)-\lie_T\g(e_i,e_j)\second(e_i,e_j),
\end{align}
where $\{e_i\}$ is an orthonormal tangent frame. The first step is to rewrite this equation in a more convenient way.

\begin{lemma}
    Let $Z\in\X(\Sigma)$ be the tangential component of $(\iota_N \mathscr{L}_T\g)^\sharp$ and set
    \begin{align*}
        Y\doteq\frac{\nabla w+HT^{\top}+Z}{w}
    \end{align*}
    then
    \begin{align}\label{jacobi}
        \diver_g Y=|\second|^2&-H^2-g\left(\frac{\nabla w}{w},Y\right)+\Ric(N,N) \\&+Hw^{-1}\left(\lie_T\bar g(N,N)+\diver_{\g}T \right) \nonumber +\frac{1}{2}w^{-1}\tr_M(\bnabla_N\mathscr{L}_T\g).
    \end{align}
\end{lemma}
\begin{proof}
First observe that
\begin{align*}
    \bnabla_{e_i}(\iota_N \lie_T\g)(e_i) &= e_i(\lie_T\g(N,e_i))-\lie_T\g(N,\bnabla_{e_i}e_i) \\ &= (\bnabla_{e_i}\lie_T\g)(e_i) -\lie_T\g(\bnabla_{e_i}N,e_i)
\end{align*}
thus recalling \eqref{eq_local_diverg_thm} we have
\begin{align*}
    (\bnabla_{e_i}\lie_T\g)(N,e_i)&=\bnabla_{e_i}(\iota_N \lie_T\g)(e_i)-\lie_T\g(\bnabla_{e_i}N,e_i)\\
    &=\diver_M(\iota_N \lie_T\g)-\second(e_i,e_j)\lie_T\g(e_j,e_i)\\
    &=\diver_g(\iota_N \lie_T\g)^{\top}-H\lie_T\g(N,N)-\second(e_i,e_j)\lie_T\g(e_j,e_i).
\end{align*}
Thus the derivative of the mean curvature reduces to
\begin{align*}
    T(H_T)=\frac{1}{2}(\tr_M\bnabla_N\lie_T\g+H\lie_T\g(N,N))-\diver_gZ,
\end{align*} 
as a consequence
\begin{align*}
    \Delta_g w =& w\left(|\second|^2+\Ric(N,N)\right)-g(\nabla H,T^{\top})\\&+\frac{1}{2}(\tr_M\bnabla_N\lie_T\g+H\lie_T\g(N,N))-\diver_gZ.
\end{align*}
By \eqref{eq_local_diverg_thm} we have
\begin{align*}
    g(\nabla H,T^{\top}) &= \diver_g(HT^{\top})-H\diver_g T^{\top} = \diver_g(HT^{\top}) + wH^2 - H\diver_M T \\ &= \diver
    _g(HT^{\top}) + wH^2 - H\left(\frac{1}{2}\lie_T\g(N,N)+\diver_{\g}T\right)
\end{align*}
thus substituting we have
\begin{align*}
    \Delta_g w &= w(|\second|^2-H^2+\Ric(N,N))-\diver_g(HT^{\top}+Z)+\\&+H(\lie_T\g(N,N)+\diver_{\g}T)+\frac{1}{2}\tr_M\bnabla_N\lie_T\g.
\end{align*}
Dividing by $w$ and noticing that $\frac{\Delta_g w}{w}=\left\vert\frac{\nabla w}{w}\right\vert^2+\diver_g\left(\frac{\nabla w}{w}\right)$ we obtain the desired formula.
\end{proof}

It is worth noting already that, thanks to the energy estimate in Proposition \ref{stima di energia}, if the hypersurface in question satisfies \ref{cauchy compatto}, which we will assume from now on, the last three terms in \eqref{jacobi} are of no concern when integrating the equation against a test function $\eta^2$, so, after integrating by parts,
\begin{align*}
    \int_\Sigma \eta^2|\second|^2 \, dV_g \leq -2\int_\Sigma \eta g(\nabla\eta,Y) \, dV_g +\int_\Sigma g\left(\frac{\nabla w}{w},Y\right) + \int_\Sigma \eta^2 H^2 + \text{bounded}.
\end{align*}
We can see at this point that the problematic terms are those containing $\frac{\nabla w}{w}$, which necessarily need to be absorbed in the left-hand side. Note that from \eqref{eq_w_nabla_u}
\begin{align}\label{gradiente di w}
        \nabla w = \frac{|\nabla u|^2\alpha}{w}\nabla\alpha + \frac{|\nabla u|\alpha^2}{w}\nabla|\nabla u|
    \end{align}
    hence, using $|\nabla\alpha|\leq Cw$ and $|\nabla u|=\frac{w}{\alpha}|Du|\leq \frac{w}{\alpha}$, we have
    \begin{align}\label{eq_nablaw_ealtro}
        \left\vert\frac{\nabla w}{w}\right\vert^2
        &\leq (1+\epsilon)\frac{w^2}{\alpha^2}|\nabla|\nabla u||^2+ C_\epsilon w^2.
    \end{align}
We see now that in order to absorb $\frac{\nabla w}{w}$, the norm $|\second|^2$ shall exceed $c\frac{w^2}{\alpha^2}|\nabla|\nabla u||^2$, for some constant $c$ greater than one, up to some further terms we can control. Such a refined Kato inequality is achieved in the following

\begin{prop}
    For all $\epsilon>0$ small enough there exists a positive constant $C_{m,\epsilon}$ such that on $\{\nabla u \neq 0\}$
\begin{align}\label{stima seconda da sotto}
    \abs{\second}^2 &\geq \frac{\alpha^2}{w^2}\left[\left(\frac{m}{m-1}-\epsilon\right)|\nabla|\nabla u||^2 + (1-\epsilon)\abs{\nabla u}^2|\mathring A|^2 \right] - C_{m,\epsilon} (H^2+w^2)
\end{align}
    where $\mathring A$ is the traceless  part of the second fundamental form $A$ of the level sets of $u$ in $(\Sigma,g)$. Moreover, there exists $C>0$ depending only on $\Sigma$ and $\tau$ such that on $\{\nabla u \neq 0\}$
    \begin{align}\label{stima seconda da sopra}
        |\second|^2\leq C\left(\frac{\alpha^2}{w^2}|\nabla|\nabla u||^2+\frac{\alpha^2|\nabla u|^2}{w^2}|\mathring A|^2+H^2+w^2\right).
    \end{align}
Furthermore, the two inequalities hold a.e and in $L^\infty(\Sigma)$ by setting 
\[
0 = |\nabla|\nabla u||^2 + |\mathring A|^2|\nabla u|^2 \qquad \text{on } \, \{\nabla u = 0\}.
\]
\end{prop}

\begin{proof}
Taking norms in \eqref{seconda} and denoting by $L\doteq F^*\lie_T\g$, we get
\begin{align}\label{norma di seconda}
    w^2|\second|^2 &= \alpha^2|\nabla^2 u|^2 + \frac
    {1}{2}g(\nabla u,\nabla\alpha)^2+ \frac{1}{2}|\nabla u|^2|\nabla\alpha|^2 + |L|^2/4 + \\ \nonumber &+ L(\nabla u,\nabla\alpha) + 2\alpha\nabla^2 u(\nabla u,\nabla\alpha) +\alpha g(\nabla^2u,L).
\end{align}
First, observe that by Stampacchia's theorem $\nabla^2 u \equiv 0$ a.e. on $\{\nabla u = 0\}$, so by \eqref{norma di seconda} $|\second| \equiv \frac{\abs{L}}{2}$ a.e. on the same set. Hence, \eqref{stima seconda da sopra} holds in $L^\infty(\Sigma)$ once it holds pointwise on $\{\nabla u \neq 0\}$. Our choice for $|\nabla|\nabla u||^2 + |\mathring A|^2|\nabla u|^2$ on $\{\nabla u = 0\}$ guarantees that the same happens for  \eqref{stima seconda da sotto}, so we only need to check the inequalities on $\{\nabla u \neq 0\}$. Using the Cauchy and Young inequalities $ab\leq \frac{\epsilon}{2}a^2+\frac{1}{2\epsilon}b^2$, for any $\epsilon>0$ we can write
\begin{align}
    w^2\abs{\second}^2\geq (1-\epsilon)\alpha^2|\nabla^2u|^2-C_\epsilon(|L|^2+|\nabla\alpha|^2|\nabla u|^2).
\end{align}
By \eqref{eq_bounds_lie}, we can estimate
\begin{align*}
    |L|^2 &= F^*|\lie_T\g|^2_{\g} + 2|F^*\iota_N\lie_T\g|_{g}^2 - \lie_T\g(N,N) \leq Cw^2
\end{align*}
so
\begin{align}\label{prima stima da sotto}
    |\second|^2 \geq (1-\epsilon)\frac{\alpha^2}{w^2}|\nabla^2 u|^2 - C_\epsilon w^2
\end{align}
and we only need to study the norm of the Hessian of $u$. Pick a point $x \in \{\nabla u \neq 0\}$, choose a local $g$-orthonormal local frame $\{e_j\}$ around $x$ with $e_1\doteq\frac{\nabla u}{\abs{\nabla u}}$ and denote by $u_{ij}$ the associated components of the Hessian of $u$. Noting that $g(\nabla\abs{\nabla u},e_j)= u_{1j}$, we have at $x$
\begin{align*}
    \abs{\nabla^2 u}^2 &= u_{11}^2 + 2\sum_{a=2}^m u_{1a}^2 + \sum_{a,b=2}^m u_{ab}^2 \\
    &= u_{11}^2 + 2\abs{\nabla^{\top}|\nabla u|}^2 + \sum_{a,b=2}^m u_{ab}^2
\end{align*}
where $\nabla^{\top}$ is the tangential gradient on $\{u = u(x)\}$. Since the second fundamental form of the level sets has components $\frac{u_{ab}}{\abs{\nabla u}}$, by definition of $\mathring A$ we have
\begin{align*}
    \sum_{a=2}^m u_{ab}^2 = \abs{\nabla u}^2|\mathring A|^2 + \frac{1}{m-1}\left(\sum_{a=2}^m u_{aa}^2 \right) = \abs{\nabla u}^2|\mathring A|^2 + \frac{(\Delta_g u - u_{11})^2}{m-1}
\end{align*}
which leads to
\begin{align}\label{espressione hessiana u}
    |\nabla^2u|^2&=u_{11}^2+2|\nabla^{\top}|\nabla u||^2+|\nabla u|^2|\mathring A|^2+\frac{(\Delta_gu-u_{11})^2}{m-1}.
\end{align}
Using the inequality $(a+b)^2 \geq (1-\epsilon)a^2 + \frac{\epsilon-1}{\epsilon}b^2$ to the last term we have
\begin{align*}
    |\nabla^2u|^2
    &\geq \frac{m-\epsilon}{m-1}|\nabla|\nabla u||^2+\frac{m-2+\epsilon}{m-1}|\nabla^{\top}|\nabla u||^2+|\nabla u|^2|\mathring A|^2+\frac{\epsilon-1}{\epsilon(m-1)}(\Delta_g u)^2 \\
    &\geq \frac{m-\epsilon}{m-1}|\nabla|\nabla u||^2+|\nabla u|^2|\mathring A|^2-\frac{1-\epsilon}{\epsilon(m-1)}(\Delta_g u)^2.
\end{align*}
Now by \eqref{eq_mean_curvature} and \eqref{eq_bounds_lie}, $ (\Delta_g u)^2\leq w^2 C(H^2+w^2)$. Thus
substituting into \eqref{prima stima da sotto} we obtain
\begin{align*}
    \abs{\second}^2 &\geq \frac{\alpha^2}{w^2}\left(\frac{(m-\epsilon)(1-\epsilon)}{m-1}|\nabla|\nabla u||^2 + (1-\epsilon)\abs{\nabla u}^2|\mathring A|^2 \right) - C_{m,\epsilon}(w^2+H^2).
\end{align*}
Notice finally that
\begin{align*}
    \frac{(m-\epsilon)(1-\epsilon)}{m-1}\geq \frac{m}{m-1}+\frac{m+1}{m-1}\epsilon
\end{align*}
so up to renaming epsilon we obtained the first inequality.
The second is readily obtained by applying Cauchy and Young inequalities to \eqref{norma di seconda} and \eqref{espressione hessiana u}.
\end{proof}

Combining the last Lemma with \eqref{jacobi}, \eqref{eq_bounds_ric} and \eqref{eq_bounds_lie} we obtain the following key inequality:
\begin{align}\label{key inequality}
    \diver_g Y &\geq \frac{\alpha^2}{w^2}\left(\left(\frac{m}{m-1}-\epsilon\right)|\nabla|\nabla u||^2 + (1-\epsilon)|\nabla u|^2|\mathring A|^2\right) \\\nonumber &- g\left(\frac{\nabla w}{w},Y\right) - C_{m,\epsilon}(H^2+w^2).
\end{align}

\begin{prop}\label{stima di seconda}
    There exists a constant $C$ depending only on $\Sigma$ and $\tau$ such that, for any $u\in\Y\cap C^\infty(\Sigma)$, the estimate
    \begin{align}\label{stima di seconda integrale}
        \int_\Sigma \eta^2|\second|^2 \, dV_g \leq C\left(\int_\Sigma|\nabla\eta|^2 \, dV_g+\int_\Sigma\eta^2H^2 \, dV_g +\int_\Sigma \eta^2w^2 \, dV_g \right)
    \end{align}
    holds for any $\eta\in C_c^1(\mathring\Sigma)$.
\end{prop}
\begin{proof}
    Integrating \eqref{key inequality} against a test function $\eta^2\in\lip_c(\Sigma,\sigma)$ we obtain
    \begin{align}\label{prima diseguaglianza}
        \int_\Sigma&\eta^2\frac{\alpha^2}{w^2}\left[\left(\frac{m}{m-1}-\epsilon\right)|\nabla|\nabla u||^2 +(1-\epsilon)|\nabla u|^2|A|^2\right] \, dV_g \leq \\ \nonumber \leq -2\int_\Sigma&\eta g(\nabla\eta,Y) \, dV_g + \int_\Sigma \eta^2 g\left(\frac{\nabla w}{w},Y\right) \, dV_g + C_{m,\epsilon}\int_\Sigma \eta^2(H^2+w^2) \, dV_g
    \end{align}
    The last integral is easily managed thanks to the energy estimate, Proposition \ref{stima di energia}.
    Next notice that, by \eqref{gradiente di w}, using $|\nabla\alpha|\leq Cw$ and $|\nabla u|=\frac{w}{\alpha}|Du|\leq \frac{w}{\alpha}$, we have
    \begin{align*}
        \left\vert\frac{\nabla w}{w}\right\vert^2
        &\leq (1+\epsilon)\frac{w^2}{\alpha^2}|\nabla|\nabla u||^2+ C_\epsilon w^2.
    \end{align*}
    On the other hand, recalling $|Z|=|\iota_N \lie_T\g|\leq Cw^2$ and \eqref{tangente di T}, we have
    \begin{align*}
        \left\vert\frac{HT^{\top}+Z}{w}\right\vert\leq C(|H|+w)
    \end{align*}
    and thus
    \begin{align*}
        g\left(\frac{\nabla w}{w},Y\right) &= \left\vert\frac{\nabla w}{w}\right\vert^2 + g\left(\frac{\nabla w}{w},\frac{HT^{\top}+Z}{w}\right) \leq (1+\epsilon)\left\vert\frac{\nabla w}{w}\right\vert^2 + C_\epsilon \left\vert\frac{HT^{\top}+Z}{w}\right\vert^2 \\
        &\leq (1+\epsilon)^2\frac{w^2}{\alpha^2}|\nabla|\nabla u||^2 + C_\epsilon(H^2+w^2).
    \end{align*} 
    Inserting this into \eqref{prima diseguaglianza} and using again Proposition \ref{stima di energia} we obtain a positive $C_\epsilon$ such that
    \begin{align*}
        \int_\Sigma \eta^2\frac{\alpha^2}{w^2}\left[\left(\frac{1}{m-1}-3\epsilon-\epsilon^2\right)|\nabla|\nabla u||^2+(1-\epsilon)|\nabla u|^2 |A|^2 + H^2\right] \, dV_g \leq \\
        \leq C_\epsilon \left(\int_\Sigma \eta^2 H^2 \, dV_g + \int_\Sigma |\nabla\eta|^2 \, dV_g + \int_\Sigma
         \eta^2 w^2 \, dV_g \right)
    \end{align*}
    which is the thesis once one recalls \eqref{stima seconda da sopra} and chooses $\epsilon$ small enough.
\end{proof}

\begin{cor}\label{Dlnw e D^2u}
    Assume that on some $\Sigma'\Subset \mathring\Sigma$ it holds
    \begin{align*}
        \norm{\rho}_{L^1(\Sigma)}\leq \mathcal{I}_1 \qquad \norm{\rho}_{L^2(\Sigma')}\leq \mathcal{I}_2.
    \end{align*}
    Then for every $\epsilon>0$ there exists a constant $C$ depending on $\epsilon$, $X$, $\Sigma'$ (and on $\Sigma$ and $\tau$) such that, for any smooth solution $u\in\Y\cap C^\infty(\Sigma)$ to \eqref{eq_SP},
    \begin{gather*}
        \int_{\Sigma'_\epsilon}\left(w_u|D^2 u|^2 + w_u^3|D^2u(Du,\cdot)|^2+w_u^5 D^2u(Du,Du)^2\right) \, dV_{\sigma_u} \leq C\left(\mathcal{I}_2^2 + \mathcal{I}_1 + 1 \right) \\
        \int_{\Sigma'_\epsilon}\left( |D(w_u^{-1/2})|^2 + |D\ln w_u|\right) \, dV_{\sigma_u} \leq C\left(\mathcal{I}_2^2 + \mathcal{I}_1 + 1 \right)
    \end{gather*}
    where $\Sigma'_\epsilon\doteq\{x\in \Sigma \ | \ \di_{\sigma_u}(x,\partial\Sigma')>\epsilon\}$ and $D$ and $|\cdot|$ are taken with respect to $\sigma = \sigma_u$.
\end{cor}
\begin{proof}
    By applying Young inequality to \eqref{eq_relation_hessians} we have
    \begin{align*}
        |\nabla^2 u|_g^2 \geq (1-\epsilon) w^4|D^2 u|_g^2 - C_\epsilon\left(w^4\sigma(Du,D\ln\alpha)^2|\sigma-du^2|_g^2 +|2d\ln\alpha\odot du|^2_g\right)
    \end{align*}
    and since
    \begin{align*}
        |\sigma-du^2|_g^2 &= \frac{m}{\alpha^4}\leq C \\
        |2d\ln\alpha\odot du|^2_g &= \frac{w^4}{2\alpha^2}\left(|D\ln\alpha|^2|Du|^2 + \sigma(D\ln\alpha,Du)^2\right) \leq Cw^4
    \end{align*}
    we have
    \begin{align*}
        |\nabla^2 u|^2_g \geq (1-\epsilon)w^4|D^2u|_g^2 - C_\epsilon w^4
    \end{align*}
    and using \eqref{prima stima da sotto} we obtain
    \begin{align*}
        |\second|^2_g &\geq (1-\epsilon)\frac{\alpha^2}{w^2}|\nabla^2u|_g^2 - C_\epsilon w^2 \geq (1-\epsilon)^2 \alpha^2w^2|D^2u|_g^2 - (1-\epsilon)C_\epsilon \alpha^2w^2 - C_\epsilon w^2 \\ &\geq C_\epsilon w^2\left(|D^2u|^2_g - 1\right).
    \end{align*}
    Now choosing $\eta\in C^\infty_c(\Sigma'_{\epsilon})$ in \eqref{stima di seconda integrale} with $\eta=1$ on $\Sigma'_\epsilon$ and $|D\eta|\leq \frac{2}{\epsilon}$, applying Proposition \ref{stima di seconda} we have
    \begin{align*}
        \int_\Sigma \eta^2w^2|D^2u|^2_g \, dV_g &\leq C_\epsilon\left(\int_\Sigma \eta^2|\second|^2_g \, dV_g + \int_\Sigma\eta^2 w^2 \, dV_g\right) \\ & \leq C_\epsilon\left(\int_\Sigma \eta^2 H^2 \, dV_g + \int_\Sigma |\nabla\eta|^2_g \, dV_g + \int_\Sigma\eta^2 w^2 \, dV_g \right)
    \end{align*}
now since $H_u = \rho + \g(X,N)$ with $|\g(X,N)|\leq C_X w$, using $|\nabla\eta|\leq w\alpha^{-1}|D\eta|$, \eqref{volumi} and Proposition \ref{stima di energia} we have
    \begin{align*}
        \int_{\Sigma'_\epsilon} w|D^2u|^2_g \, dV_\sigma &\leq  C_{X,\epsilon}\left(\int_\Sigma \eta^2\rho^2 \, dV_g + \frac{2}{\epsilon}\int_\Sigma w^2 \, dV_g + \int_\Sigma \eta^2 w^2 \, dV_g  \right) \\ &\leq C_{X,\epsilon} \left(\norm{\rho}_{L^2(\Sigma')}^2 + \norm{\rho}_{L^1(\Sigma)} + 1 \right).
    \end{align*}
    Finally observe that, since $dw = w^3 D^2u(Du,\cdot)$,
    \begin{align*}
        w|D^2u|^2_g  &= \alpha^{-4}\left(w|D^2 u|^2 + 2w^3|D^2u(Du,\cdot)|^2+w^5 D^2u(Du,Du)^2\right) 
        \\ &\geq C \left(w|D^2 u|^2 + w^{-1}|D\ln w|^2\right)
    \end{align*}
    To conclude simply observe that $w^{-1}|D\ln w|^2 = |D(w^{-1/2})|^2$ and that, by the Cauchy inequality,
    \begin{align*}
        \int_{\Sigma'_\epsilon} |D\ln w| \, dV_\sigma \leq \left(\int_{\Sigma'_\epsilon} w \, dV_\sigma\right)^{\frac{1}{2}}\left(\int_{\Sigma'_\epsilon} w^{-1}|D\ln w|^2 \, dV_\sigma \right)^{\frac{1}{2}}
    \end{align*}
    and use once again Proposition \ref{stima di energia}.
\end{proof}

\subsection{Higher integrability}

In this section we show that, \emph{if $\Sigma$ has dimension $m=2$,} in regions where  $\rho$ is $L^2$ the tilt function $w$ enjoys better integrability properties, namely, that the integral of $w \ln (w + 1)$ is bounded by a uniform constant. 

\begin{prop}\label{higher integrability}
    Assume $m=2$ and that on some $\Sigma'\subseteq \Sigma$ it holds
    \begin{align*}
        \norm{\rho}_{L^1(\Sigma)}\leq \mathcal{I}_1 \qquad \norm{\rho}_{L^2(\Sigma')}\leq \mathcal{I}_2.
    \end{align*}
    Then for every $\epsilon>0$ there exists a constant depending on $\epsilon$, $X$, $\Sigma'$, $\mathcal{I}_1$, $\mathcal{I}_2$ (and on $\Sigma$ and $\tau$) such that, for any smooth solution $u\in\Y\cap C^\infty(\Sigma)$ to \eqref{eq_SP},
    \begin{align*}
        \int_{\Sigma_\epsilon'} w_u \ln (w_u + 1) \, dV_{\sigma_u} \leq C(\epsilon,X,\Sigma',\mathcal{I}_1,\mathcal{I}_2)
    \end{align*}
    where $\Sigma'_\epsilon\doteq\{x\in \Sigma' \ | \ \di_{\sigma_u}(x,\partial\Sigma')>\epsilon\}$.
\end{prop}
\begin{proof}
    Take a test function $\eta \in C^1_c(\Sigma')$ such that
    \begin{align*}
        \eta \equiv 1 \quad \text{on $\Sigma_\epsilon'$,} \qquad \operatorname{spt} \eta \subseteq \Sigma_{\epsilon/2}'.
    \end{align*}
    Set
    \begin{align*}
        f = \eta \ln (w_u +1)\in C^\infty_c(\Sigma_{\epsilon/2}'), \qquad \mu = w_u \, dV_{u}\llcorner\Sigma_{\epsilon/2}',
    \end{align*}
    so that 
    \begin{align*}
        \int_{\Sigma_\epsilon'} w_u\ln(w_u + 1) \, dV_{\sigma_u} \leq \int_\Sigma \eta w_u\ln(w_u + 1) \, dV_{\sigma_u} = \int_{\Sigma} f \, d\mu.
    \end{align*}
    Since $\overline{\Sigma_{\epsilon/2}'}$ is compact, we can cover it with finitely many, say $N$, charts $\psi_i:U_i\to \R^m$ with $\psi_i(U_i) = \R^m$ and $U = \bigcup_i U_i\Subset \Sigma'_{\epsilon/4}$. Pick a partition of unit $\{\zeta_i\}$ subordinated to the covering $\{U_i\}$, that is
    \begin{align*}
        \zeta_i \in C_c^\infty(U_i), \qquad 0\leq \zeta_i \leq 1, \qquad \operatorname{spt}\zeta_i \subseteq U_i, \qquad \sum_{i=1}^N \zeta_i = 1 \quad \text{on $\Sigma_{\epsilon/2}'$},
    \end{align*}
    and let
    \begin{align*}
        \tilde f_i = f \circ \psi_i^{-1} \in C^\infty(\R^m), \qquad \tilde\zeta_i = \zeta \circ \psi_i^{-1} \in C^\infty_c(\R^m),
    \end{align*}
    so that $\tilde f_i \tilde \zeta_i \in C^1_c(\R^m)$. Now remember that the following trace inequality holds on $\R^m$ for any radon measure $\nu$ and any positive function $h\in C^1_c(\R^m)$, see \cite[Corollary 1.1.2]{maz2008theory}:
    \begin{align*}
        \int_{\R^m} h \, d\nu \leq c_m \sup_{\substack{r>0 \\ x\in\R^m}} \frac{\nu(B_r^\delta(x))}{r^{m-1}} \int_{\R^m} |\tilde D h| \, dx
    \end{align*}
    where $\tilde D$ is the Euclidean gradient and $B^\delta_r$ is an Euclidean ball. Applying this inequality to each measure $\mu_i = (\psi_i)_*(\mu\llcorner \operatorname{spt}\zeta_i)$ we have
    \begin{align*}
        \int_\Sigma f \, d\mu &= \sum_{i=1}^N \int_{\psi_i^{-1}(\R^m)}  \zeta_i f_i \, d\mu = \sum_{i=1}^N \int_{\R^m} \tilde\zeta_i \tilde f_i \, d\mu_i \\
        &= c_m\sum_{i=1}^N \sup_{\substack{r>0 \\ x\in\R^m}} \frac{\mu_i(B_r^\delta(x))}{r^{m-1}} \int_{\R^m} |\tilde D(\tilde \zeta_i \tilde f_i)| \, dx \\
        &\leq c_m \sum_{i=1}^N \sup_{\substack{r>0 \\ x\in\R^m}} \frac{\mu_i(B_r^\delta(x))}{r^{m-1}} \left( \int_{\R^m} |\tilde D\tilde\zeta_i|\tilde f_i \, dx + \int_{\R^m} \tilde\zeta_i |\tilde D \tilde f_i| \, dx \right)
    \end{align*}
    Now observe that $\delta$ and $(\psi_i^{-1})^*\sigma_u$ control each other as bilinear forms on the support of $\tilde\zeta_i$, hence for each $i$ there exists a constant $c_i$ such that
    \begin{gather*}
        \int_{\R^m} |\tilde D \tilde\zeta_i||\tilde f_i| \, dx \leq c_i \int_{U_i} |D\zeta_i||f| \, dV_{\sigma_u},  \qquad \int_{\R^m} \tilde\zeta_i |\tilde D \tilde f_i| \, dx \leq c_i \int_{U_i} \zeta_i |D f| \, dV_{\sigma_u}, \\
        \psi^{-1}_i(B_r^\delta (x)\cap \operatorname{spt}\tilde\zeta_i) \subseteq B_{c_i r}^{\sigma_u} (\psi_i^{-1}(x))\cap\operatorname{spt}\zeta_i, \qquad \forall x\in\R^m, \forall 
        r>0
    \end{gather*}
    thus, recalling that $f$ is supported in $\Sigma_{\epsilon/2}'$, we have
    \begin{align*}
        \int_{\Sigma} f \, d\mu &\leq c_m\sum_{i=1}^N \sup_{\substack{r>0 \\ x\in\R^m}}\frac{\mu_i(B_r^\delta(x))}{r^{m-1}} c_i\left( \int_{U_i} |D\zeta_i|f \, dV_{\sigma_u} + \int_{U_i} \zeta_i |D f| \, dV_{\sigma_u}  \right) \\ &\leq c_m \sum_{i=1}^N \sup_{\substack{r>0 \\ x\in\R^m}}\frac{\mu_i(B_r^\delta(x))}{r^{m-1}} c_i\left(\norm{D\zeta_i}_{L^\infty(\Sigma)}\int_{\Sigma'_{\epsilon/2} } f \, dV_{\sigma_u} + \int_{\Sigma'_{\epsilon/2}} |D f| \, dV_{\sigma_u} \right) 
        \\ &\leq c_{m,\zeta}\norm{f}_{W^{1,1}(\Sigma'_{\epsilon/2},\sigma_u)} \sum_{i=1}^N \sup_{\substack{r>0 \\ x\in\R^m}} \frac{\mu_i(B_r^\delta(x))}{r^{m-1}}, 
    \end{align*}
    where $c_{m,\zeta} = Nc_m \left( 1+ \max_i c_i\norm{D\zeta_i}_{L^\infty(\Sigma)}\right)$ depends on the chosen partition of unity. Note that 
    \begin{align*}
        \mu_i(B_r^\delta(x)) &= \mu(\psi_i^{-1}(B_r^\delta(x)\cap\operatorname{spt}\tilde\zeta_i)) \leq \mu \big(B_{c_ir}^{\sigma_u}(\psi_i^{-1}(x)) \cap \operatorname{spt}\zeta_i \big) \\
        &\leq \mu(B_{cr}^{\sigma_u}(\psi_i^{-1}(x))) = \int_{B_{cr}^{\sigma_u}(\psi_i^{-1}(x))\cap \Sigma_{\epsilon/2}'} w_u \, dV_{\sigma_u}
    \end{align*}
    where $c = \max_i c_i$. Now, recalling that $U\subseteq \Sigma_{\epsilon/4}'$ if $r<r_0 \doteq\frac{\epsilon}{16c}$ then for any $i$ and $x\in \R^m$, $B_{cr}^{\sigma_u}(\psi^{-1}(x))\Subset \Sigma$, so by combining Proposition \ref{crescita sulle bolle}, the energy estimate in Proposition \ref{stima di energia} and using $m = 2$, for any $x\in\R^m$ and $r < r_0$ we have 
    \begin{align*}
        \frac{\mu_i(B_r^\delta(x))}{r^{m-1}} &\leq \frac{1}{r}\int_{B_{cr}^{\sigma_u}(\psi_i^{-1}(x))} w_u \, dV_{\sigma_u} \\ &\leq cC_X  \left( \frac{1}{r_0}\int_{B_{r_0}^{\sigma_u}(\psi_i^{-1}(x))} w_u \, dV_{\sigma_u} + \int_{B_{r_0}^{\sigma_u}(\psi_i^{-1}(x))} |\rho| \, dV_{\sigma_u} + r_0 \right) \\
        &\leq C(\epsilon, \zeta, X, \mathcal{I}_1).
    \end{align*}
    On the other hand, if $r \geq r_0$ then, again by Proposition \ref{stima di energia},
    \begin{align*}
        \frac{\mu_i(B_r^\delta(x))}{r^{m-1}} \leq \frac{1}{r_0^{m-1}}\int_{\Sigma_\epsilon'} w_u \, dV_{\sigma_u} \leq \frac{16c}{\epsilon}C_X(1 + \mathcal{I}_1).
    \end{align*}
    We have establish
    \begin{align*}
        \sup_{\substack{r>0 \\ x\in \R^m}}\frac{\mu_i(B_\delta(x))}{r} \leq C(\epsilon,\zeta,X,\mathcal{I}_1) \qquad \forall i = 1,\dots N
    \end{align*}
    hence
    \begin{align*}
        \int_{\Sigma_\epsilon'} w_u \ln(w_u + 1) \, dV_{\sigma_u} \leq NC(m,\epsilon,\zeta, X, \mathcal{I}_1) \norm{f}_{W^{1,1}(\Sigma_{\epsilon/2}',\sigma_u)}.
    \end{align*}
    The $W^{1,1}$ norm is finally estimated by Corollary \ref{Dlnw e D^2u} and Proposition \ref{stima di energia}:
    \begin{align*}
        \int_{\Sigma_{\epsilon/2}'} (f + |Df|) \, dV_{\sigma_u} &\leq \int_{\Sigma'_{\epsilon/2}} (\eta \ln(w_u + 1) + \ln(w_u + 1)|D\eta| + \eta|D\ln w_u|) \, dV_{\sigma_u}  \\
        &\leq \norm{D\eta}_{L^\infty(\Sigma)} \int_{\Sigma_{\epsilon/2}'} \left( w_u + 1 + |D\ln w_u| \right) \, dV_{\sigma_u} \\
        &\leq \norm{D\eta}_{L^\infty(\Sigma)} C(\epsilon,X, \Sigma',\mathcal{I}_1,\mathcal{I}_2)
    \end{align*}
    and the Proposition is proved.
\end{proof}

\section{Proof of the Theorem \ref{main result}}\label{sect_proof}

\stoptoc

We consider the sequence of pairs $(\rho_j,X_j)$ guaranteed by Proposition \ref{prop_appendix_sec}.


\subsection*{Step 1: Existence of approximating solutions} \textit{For each $j\in\N$ there exists a smooth classical solution $u_j$ to the Dirichlet problem}
\begin{align}\tag{${\rm PMC}_j$}
    \begin{cases}
        H_{u_j} = \frac{d\rho_j(u_j)}{dV_{\sigma_{u_j}}} + \g(X_j,N_{u_j}) \\
        u_j\in \Y.
    \end{cases}
\end{align}

\begin{proof}
The proof relies on a fixed-point argument due to Bartnik, and we include full details for the  reader's convenience. First, recall that, by \eqref{eq_mean_curvature_sigma_ij} the mean curvature is
\begin{align*}
    H_u = \alpha_u w_u g_u^{ij} u_{ij} + w_u B^k(x,u,du) u_k + w_u^{-1}H_u^S(x,u).
\end{align*}
Hereafter, as usual, all Sobolev and H\"older spaces are considered with respect to the metric $\sigma_\varphi$. Freezing the coefficients at some spacelike $v\in \Y\cap C^{1,\alpha}(\Sigma)$, we define the \emph{linear} operator
\begin{gather*}
    L_v \ \ : \quad \Y\cap C^{2,\alpha}(\Sigma)\to C^{0,\alpha}(\Sigma) \\
    L_v u \doteq w_v\alpha_vg_v^{ij} u_{ij}  + w_v B^k(x,v,dv)u_k.
\end{gather*}
The coefficients of $L_v$ are Hölder-continuous, and the operator is uniformly elliptic since $w_v \in L^\infty$. Recall also that, by the construction of $\rho_j$, since $v$ is a $C^{1,\alpha}(\Sigma)$ function, so is the density $\tfrac{d\rho_j(v)}{dV_{\sigma_v}}$, in particular it is $L^2(\Sigma)$. Hence, by standard Fredholm's theory \cite[Theorem 6.4]{evans} the existence and uniqueness of $u \in C^{2,\alpha}(\Sigma)$ solving
\begin{align}\label{eq_freezed_dirichlet}
    \begin{cases}
        L_v u = f_v & \quad \text{in } \, \Sigma,\\
        u=\varphi & \quad \text{on } \, \partial\Sigma,
    \end{cases}
    \qquad f_v = \frac{d\rho_j(v)}{dV_{\sigma_v}} + \g(X_j, N_{v}) - w^{-1}_vH^S(x,v) \in C^{0,\alpha}(\Sigma)
\end{align}
is guaranteed provided that the homogeneous problem
\begin{align*}
    \begin{cases}
        L_v u = 0 & \quad \text{in $\Sigma$} \\
        u = 0 & \quad \text{on $\partial\Sigma$}
    \end{cases}
\end{align*}
only admits the trivial solution. This follows from the comparison principle, since $L_v$ is a uniformly elliptic operator with vanishing zero-order term.

Let $\mathcal{B} = C^{1,\alpha}(\Sigma)$, and consider the closed convex subset
\begin{align*}
    \mathscr{C} \doteq \set{v \in \mathcal{B} \cap \Y} {\norm{v}_{\mathcal{B}} \leq K, \quad \norm{w_v}_{L^\infty} \leq K},
\end{align*}
for some constant $K > 0$ to be chosen later.  
The above discussion can be summarized by saying that the operator $P: \mathscr{C} \to \mathcal{B}$, which maps $v$ to the unique $C^{2,\alpha}(\Sigma)$ solution of \eqref{eq_freezed_dirichlet}, is well defined and by Schauder estimates is also compact.
Observe that if $v \in \mathscr{C}$ is a fixed point of the homotopy $tP$ with $t\in[0,1]$, then $v \in C^{2,\alpha}(\Sigma)$ and its mean curvature is
\begin{align*}
    H_v = t\left(\frac{d\rho_j(v)}{dV_{\sigma_v}}+\g(X_j,N_v)\right) + (1-t)w_v^{-1}H^S(x,v) \in C^{1,\alpha}(\Sigma),
\end{align*}
so by Schauder's theory $v \in C^{3,\alpha}(\Sigma)$. The uniform gradient bound required for the application of the Leray–Schauder fixed point theorem is provided by Bartnik’s estimate, that we rephrase as follows:
\begin{theorem}[\cite{bartnik84}, Corollary 3.4]\label{thm_bartnik_estimates}
    Assume $(\M,\g)$ is a globally hyperbolic spacetime, let $\tau\in C^\infty(\M)$ be a time function and set $T=-\bnabla\tau/|\bnabla\tau|$. Let $M\subseteq\M$ be a compact $C^{3,\alpha}$ spacelike hypersurface with boundary satisfying \ref{cauchy compatto}, and assume that its mean curvature satisfies the following structure conditions  with constant $\Lambda$:
    \begin{align}\tag{MCSC}\label{eq_structure conditions}
        |H|\leq \Lambda w, \qquad |\nabla H| \leq \Lambda(w^2 + w|\second|),
    \end{align}
    where as usual $w = -\g(T,N)$ and $\nabla$, $\abs{\cdot}$ are taken with respect to the induced metric on $M$. Then there exists a constant $C(\g,\tau,\partial M,\Lambda)$ such that $w\leq C$ on $M$.
\end{theorem}

Recalling \eqref{eq_bound_euclidean_norm}, \eqref{eq_bounds_lie} and Proposition \ref{prop_appendix_sec} we have
\begin{align*}
    |H_v| &\leq \abs{\frac{d\rho_j(v)}{dV_{\sigma_v}}}+ \abs{\g(X_j,N_v)} + w_v^{-1}|H^S(x,v)| \\& \leq C_{\rho,j} + \sqrt{2}w_v\norm{X_j} + \norm{H^S}_{C(\overline{D(\Sigma)})}.
\end{align*}
On the other hand, combining \eqref{eq_nablaw_ealtro} and \eqref{stima seconda da sotto} and omitting the subscript $v$ we deduce
\[
|\nabla w^{-1}| = \frac{|\nabla w|}{w^2} \le C\left( w^2 + |\second|w\right). 
\]
Furthermore, for each unit tangent vector $e$ to $M$
\begin{align*}
    |e\g(X_j,N)| &\leq |\g(\bnabla_{e}X_j,N)| + |\second(X_j^{\top},e)| \\
    &\leq \norm{X_j}_1 \norm{e}\norm{N} + |\second||X_j^{\top}| \\
    &\leq C \left(\norm{X_j}_1 w^2 + |\second|\norm{X_j}w\right).
\end{align*}
Putting all together, we have
\begin{align*}
    \abs{\nabla^v H_v} &\leq \abs{ \nabla^v\left(\frac{d\rho_j(v)}{dV_{\sigma_v}}\right)} + \abs{\nabla^v\g(X_j,N_v)} + |\nabla^v (w^{-1}H_v^S)| \\
    &\leq\norm{\frac{d\rho_j(v)}{dV_{\sigma_v}}}_{C^1(\Sigma)} w_v^2 + C \norm{X_j}_1 \left( w_v^2 + |\second_v|w_v \right) \\
    &\leq C_{\rho,j,X} \left( w_v^2 + |\second_v|w_v \right),
\end{align*}
where we used Proposition \ref{prop_appendix_sec} and the fact that, by \ref{cauchy compatto}, the norm $\norm{ H^S}_1$ is bounded by a constant $C$ that only depends on $\Sigma$. We thus conclude that $H_v$ satisfies the structure conditions \eqref{eq_structure conditions} with constant $\Lambda_j = C_{\rho,j,X}$. Theorem \ref{thm_bartnik_estimates} and the Hölder estimates of
Ladyzhenskaya–Ural'tseva (see \cite{GilbargTrudinger1977}) provide the estimates
\begin{align*}
    \norm{w_v}_{L^\infty(\Sigma)} \leq C_{\rho,j,X} \qquad \norm{v}_{C^{1,\alpha}(\Sigma)} \leq C_{\rho,j,X}
\end{align*}
for any fixed point $v \in \mathscr{C}$ of $tP$, uniformly in $t$. Choosing $K = C_{\rho,j,X} +1$ in the definition of $\mathscr{C}$, we see that every fixed point of $tP$ lies in the interior of $\mathscr{C}$ and by Leray-Shauder theory, there exists a fixed point $u_j$ which is a $C^{1,\alpha}$ solution to \eqref{eq_approx_problem}. Standard bootstrapping then yields smoothness.
\end{proof}

Fixed $\eta\in C^1_c(\mathring{\Sigma})$, we therefore have for any $j$
\begin{align}\label{approssimanti}
    \int_\Sigma \eta \frac{\alpha_j^m}{w_j}\diver_{M_j}T \, dV_j = \int_\Sigma w_j\alpha_j^{m-1}\sigma_j(du_j,d\eta) \, dV_j + \int_\Sigma \eta\alpha_j^m \, d\rho_j + \int_\Sigma \eta\alpha^m\g(X_j,N_j) \, dV_j
\end{align}
where $M_j = F_{u_j}(\Sigma)$, $\sigma_j = \sigma_{u_j}$, $\alpha_j = \alpha_{u_j}$, $\rho_j = \rho_j(u_j)$ and so on.

\subsection*{Step 2: Existence of a limit}\label{step2} \emph{There exists} $u\in C(\Sigma)\cap W^{2,2}_{\loc}(\Sigma)$ \emph{such that, up to a subsequence,} $u_j\to u$ \emph{ in } $W^{1,p}(\Sigma)$ \emph{for each} $p \in [1,\infty)$ (\emph{in particular, $u_j \to u$ in $C(\Sigma)$}), \emph{and} $du_j\to du$ \emph{a.e. in $\Sigma$.}

\emph{Furthermore, $u$ satisfies $w_u \in L^1(\Sigma)$ and the second in \eqref{eq_higherinteg}.}

\begin{proof}
The uniform convergence to a function $u\in C(\Sigma)$ follows from Ascoli-Arzelà once one notices that $\{u_j\}$ is equi-Lipschitz and $u_j=\varphi$ in $\partial\Sigma$. By hypothesis, there exists a compact set $E\subseteq \mathring{\Sigma}$ with $\mathscr{H}^1(E) = 0$ such that for any $\Sigma'\Subset \mathring\Sigma\backslash E$ it holds $\rho(u)\in \mathscr{L}^2(\Sigma')$ for all $u\in\Y$. In particular $\tfrac{\rho_j(u_j)}{dV_{\sigma_{u_j}}}$ is $L^2(\Sigma')$ for all $j$ and, by Proposition \ref{prop_appendix_sec}, for every $\delta > 0$ there exists a constant $C = C_{\Sigma',\delta}$ such that for $j>>1$
\begin{align*}
    \norm{\frac{\rho_j(u_j)}{dV_{\sigma_{u_j}}}}_{L^2(\Sigma'_{\delta/2})}  &\leq C \norm{\frac{\rho(u_j)}{dV_{\sigma_{u_j}}}}_{L^2(\Sigma')} \leq \mathcal{I}_2.
\end{align*}
for some constant $\mathcal{I}_2$ depending on $\Sigma',\delta$ but not on $j$. This second inequality follows from the fact that $\Y$ is weakly compact and, by our assumption, $\rho \llcorner \Sigma'$ is bounded. 
Similarly, by the continuity of $\rho: (\Y, \|\cdot \|_{C(\Sigma)}) \to\mathscr{M}(\Sigma)$, Lemma \ref{conseguenze di C} and Proposition \ref{prop_appendix_sec},
\begin{align*}
    \norm{\frac{\rho_j(u_j)}{dV_{\sigma_{u_j}}}}_{L^1(\Sigma)} \le C \norm{\rho_j(u_j)}_{\mathscr{M}(\Sigma)} \leq C_\rho \doteq \mathcal{I}_1
\end{align*}
again for some constant $\mathcal{I}_1$ uniform in $j$. Moreover, by \eqref{C_9} in Lemma \ref{conseguenze di C}
\begin{align*}
    |\g(X_j,N_j)| \leq C_X w_j
\end{align*}
for some $C_X\Lambda\geq 0$. Therefore, by Proposition \ref{stima di energia} and Corollary \ref{Dlnw e D^2u}, if $D_j$ denotes the connection of the metric $\sigma_j =\sigma_{u_j}$ there exist constants $C_1 = C_1(X,\Sigma,\mathcal{I}_1)$ and $C_2 = C_2(\delta,X,\Sigma,\Sigma',\mathcal{I}_1,\mathcal{I}_2)$ such that
\begin{align}
&\int_\Sigma w_j dV_j \leq C_1 \label{eq_energ}\\
& \int_{\Sigma'_\delta} \left(w_j|D^2_j u_j|^2 + w_j^3|D^2_ju_j(D_ju_j,\cdot)|^2+w_j^5 D^2_ju_j(D_ju_j,D_ju_j)^2\right) \, dV_j \leq C_2 \label{eq_secondfund}.
\end{align}
As a consequence of \eqref{eq_W22_norms}, $\{u_j\}$ is bounded in $W_{\loc}^{2,2}(\Sigma')$. Hence, up to a subsequence $u_j\to \tilde u$ strongly in $W^{1,2}_\loc(\Sigma')$, and possibly taking a further subsequence, $u_j$ and $du_j$ converge to $\tilde u$ and $d\tilde u$ $\sigma
_\varphi$-a.e., which by the uniqueness of the limit implies that $\tilde u = u$ on $\Sigma'$. On the other hand, by reflexivity, up to another subsequence $\{u_j\}$ converges weakly in $W^{2,2}_{\loc}(\Sigma')$, and hence also weakly in $W^{1,2}_{\loc}(\Sigma')$, so by uniqueness $u \in W^{2,2}_{\loc}(\Sigma')$. Moreover, since $u_j$ is bounded in $W^{1,\infty}(\Sigma)$, H\"older's inequality and $u_j \to u$ in $W^{1,2}_\loc(\Sigma)$ easily imply that $u_j \to u$ in $W^{1,p}(\Sigma)$ for each $p \in [1,\infty)$. 

To conclude, passing to limits \eqref{eq_energ} yields $w_u \in L^1(\Sigma)$, while combining a truncation argument on $w_j$ with lower-semicontinuity (see \cite[Corollary 4.11]{BIMM} for full details) one can pass to limits in \eqref{eq_secondfund} and deduce the second in \eqref{eq_higherinteg} by the arbitrariness of $\Sigma'$.
\end{proof}
%

\subsection*{Step 3: Convergence of the integral identity} As a last step we shall prove that
\begin{align*}
    \lim_{j\to \infty} \int_\Sigma \eta \frac{\alpha_j^m}{w_j}\diver_{F_j}T \, dV_j 
    &= \int_\Sigma \eta \frac{\alpha^m}{w} \diver_MT \, dV_\sigma \tag{a}\label{eq_a}\\ 
    \lim_{j\to\infty} \int_\Sigma \eta w_j\alpha_j^{m-1}\sigma_j(du_j,d\eta) \, dV_j 
    &= \int_\Sigma \eta w\alpha^{m-1}\sigma(du,d\eta) \, dV_\sigma \tag{b}\label{eq_b}\\
    \lim_{j\to\infty} \int_\Sigma \eta\alpha_j^m \, d\rho_j + \int_\Sigma \eta\alpha_j^m\g(X_j,N_j) \, dV_{\sigma_j} 
    &= \int_\Sigma \eta\alpha^m \, d\rho + \int_\Sigma \eta \alpha^m\g(X,N) \, dV_\sigma \tag{c}\label{eq_c}.
\end{align*}
where $\alpha = \alpha_u$, $\rho = \rho(u)$ and so forth. Recall that, by the Vitali convergence theorem, for any $\{f_j\}\in L^1_\loc(\Sigma)$ and $f$ measurable function we have
\[
\begin{array}{r c l}
f_j \to f \text{ in $L_\loc^1(\Sigma)$} & \iff & 
\begin{array}{l}
f_j \to f \ \text{in measure and} \\
\{|f_j|\}_j \ \text{is locally uniformly integrable in $L^1(\Sigma)$.}
\end{array}
\end{array}
\]
The convergence in measure of the integrands are straightforward and, in fact, we can say more: for any $j$ there exists $v_j\in C(\Sigma)$ such that $dV_j = v_j \, dV_{\sigma}$, and by the uniform convergence $u_j\to u$ established in paragraph \ref{step2} we deduce
\begin{align*}
    \alpha_j \to \alpha, \qquad v_j \to 1 \qquad \text{in $C(\Sigma)$.}
\end{align*}
and furthermore the a.e. convergence $du_j\to du$ ensures that
\begin{align*}
     w_j \to w, \qquad \diver_{M_j} T\to \diver_MT \qquad \text{a.e. in $\Sigma$.}
\end{align*}
Therefore, we only need to care about the local uniform integrability, which in turn follows from the local uniform integrability of $\{w_j\}$ in $\mathring\Sigma$, because the other terms are uniformly bounded. 


By the de la Vallée-Poussin theorem, the local uniform integrability of $\{w_j\}$ on any $\Sigma
'\Subset\mathring\Sigma\backslash E$ is equivalent to the existence of a family of increasing convex functions $f_\epsilon:[0,\infty)\to[0,\infty)$ depending on $\epsilon > 0$ such that
\begin{align*}
    \lim_{t\to\infty} \frac{f(t)}{t} = \infty, \qquad \sup_{j\in\N} \int_{\Sigma_\epsilon'} f_\epsilon(|w_j|) \, dV_\sigma < \infty.
\end{align*}
Choose $f_\epsilon(t) = t\ln (t + 1)$ for any $\epsilon>0$. By Proposition \ref{prop_appendix_sec}, 
\begin{gather*}
    \norm{\frac{\rho_j(u_j)}{dV_{\sigma_{u_j}}}}_{L^1(\Sigma)}\leq \mathcal{I}_1,
    \quad \norm{\frac{d\rho_j(u_j)}{dV_{\sigma_{u_j}}}}_{L^2(\Sigma'_\epsilon)} \leq C_2, \quad  \norm{\frac{\rho(u_j)}{dV_{\sigma_{u_j}}}}_{L^2(\Sigma')} \leq \mathcal{I}_2, \\
    |\g(X_j,N_{u_j})| \leq C_X w_j,
\end{gather*}
thus Proposition \ref{higher integrability} ensures the uniform bound
\begin{align*}
    \int_{\Sigma_\epsilon'} w_j \ln(w_j + 1 ) \, dV_{\sigma_j} \leq C(\epsilon,X,\Sigma',\mathcal{I}_1,\mathcal{I}_2)<\infty.
\end{align*}
This shows that $\{w_j\}$ is locally uniformly integrable over $\Sigma'$, and, by the arbitrariness of $\Sigma'\Subset\mathring\Sigma\backslash E$, in the whole of $\mathring\Sigma\backslash E$. We conclude by invoking Theorem \ref{singolarità rimovibili} to establish local uniform integrability in the entire $\mathring{\Sigma}$.

Now we can prove the limits \eqref{eq_a}, \eqref{eq_b} and \eqref{eq_c}. Notice that, by \eqref{eq_bounds_lie} we have
\begin{align*}
    \eta w_j^{-1}\alpha_j^m |\diver_{M_j}T| \, v_j \leq Cw_j
\end{align*}
hence the integrand in \eqref{eq_a} is locally uniformly integrable and the $L^1$ convergence follows. The proof of \eqref{eq_b} is identical. Concerning the last one, as a consequence of $\rho_j\overset{*}{\rightharpoonup} \rho$ (Proposition \ref{prop_appendix_sec}) and $\alpha_j\to \alpha$ in $C(\Sigma)$ we have
\begin{align*}
    \int_\Sigma \eta\alpha_j^m d\rho_j \longrightarrow \int_\Sigma \eta\alpha^m \, d\rho.
\end{align*}
Moreover, by Proposition \ref{prop_appendix_sec} it holds $\g(X_j,N)\rightharpoonup \g(X,N)$ in $\mathscr{M}(\Sigma)$ and $|\g(X_j,N_j)|\leq C_X w_j$, so
\begin{align*}
    \left|\int_\Sigma \eta\alpha_j^m \g(X_j,N_j)v_j \, dV_j - \int_\Sigma \eta\alpha^m\g(X,N) \, dV_\sigma\right| \leq \\
    \left| \int_\Sigma \left(\alpha_j^mv_j - \alpha^m \right)\g(X_j,N_j) \, dV_\sigma \right| &+ \left| \int_\Sigma \eta\alpha^m \left( \g(X_j,N_j) - \g(X,N) \right) \, dV_\sigma\right| \leq\\
    \leq C_X \norm{\alpha_j^mv_j - \alpha^m}_{L^\infty(\Sigma)}\int_\Sigma w_j \, dV_j &+ o(1), \qquad j\to\infty.
\end{align*}
The energy estimate in Proposition \ref{stima di energia} then yields 
\[
\int_\Sigma \eta\alpha_j^m \g(X_j,N_j)v_j \, dV_j \to \int_\Sigma \eta\alpha^m \g(X,N)v \, dV_\sigma,
\]
which concludes the proof of \eqref{eq_c}.

This concludes the proof of the existence statement in Theorem \ref{main result}. 

\subsection*{Step 4: Absence of light-segments} \textit{The graph of $u$ has no light-segments.}

\begin{proof}
    First, note that since $\mathscr{H}^1(E)=0$ it suffices to prove the absence of light segments within $\Sigma'$.

Suppose by contradiction that there exists a light segment $\gamma = \overline{xy} \Subset \Sigma'$. Write $g_j = g_{u_j}$ and $\sigma_j = \sigma_{u_j}$ for convenience. Take a smooth domain $Q\Subset \Sigma'$ with $\gamma\subset \partial Q$, and a small enough $\epsilon$ such that $Q \Subset \Sigma'_\epsilon$. By the trace inequality in \cite[Corollary 1.1.2]{maz2008theory} and a partition of unity argument as in the proof of Proposition \ref{higher integrability}, applied this time to the Hausdorff measure $\mu = \mathscr{H}^1_{\sigma_j} \llcorner \partial Q$, we get
\begin{align*}
    \int_{\partial Q} f \, ds_j \leq C \|f\|_{W^{1,1}(\Sigma'_{\epsilon/2},\sigma_u)} 
    \qquad \forall \, f \in C^1_c(\Sigma'_{\epsilon/2})
\end{align*}
for some constant $C$ independent of $j$. Here, $s_j$ is the arclength of $\partial Q$ in the metric $\sigma_j$, and we used as usual Lemma \ref{conseguenze di C}. Applying the result to $f = \psi \ln w_j$, where $\psi \in C^\infty_c(\Sigma'_{\epsilon/2})$ is one in $\Sigma'_\epsilon$, by Propositions \ref{stima di energia} and \ref{Dlnw e D^2u} we infer
\begin{align*}
    \int_{\partial Q} \ln w_j \, ds_j \leq C'
\end{align*}
with $C'$ independent of $j$.\\
On the other hand, by the uniform convergence of $u_j$ to $u$ we have 
\[
\mathscr{H}^1_{\sigma}(\gamma) = |u(y)-u(x)| = \lim_{j \to \infty} |u_j(y) - u_j(x)| \le \liminf_{j\to\infty} \int_\gamma |du_j|_{\sigma_j} ds_j \le \liminf_{j \to \infty} \mathscr{H}^1_{\sigma_j}(\gamma)
\]
Whence, from $\mathscr{H}^1_{\sigma_j}(\gamma) \to \mathscr{H}^1_{\sigma}(\gamma)$ we deduce
\[
\fint_\gamma |du_j|_{\sigma_j} ds_j \doteq \frac{1}{\mathscr{H}^1_{\sigma_j}(\gamma)} \int_\gamma |du_j|_{\sigma_j} ds_j\to 1 \qquad \text{as } \, j \to \infty. 
\]
By applying the convex function 
\[
\psi : [0,1) \to [0,\infty), \qquad \psi(t) = \ln \frac{1}{\sqrt{1-t^2}}
\]
we obtain 
\[
\psi\left(\fint_\gamma |du_j|_{\sigma_j} ds_j \right) \to \infty \qquad \text{as } \, j \to \infty.
\]
However, by Jensen's inequality
\[
\psi\left(\fint_\gamma |du_j|_{\sigma_j} ds_j \right) \le \fint_\gamma \psi(|du_j|_{\sigma_j}) ds_j = \fint_\gamma \ln w_j ds_j \le \frac{C'}{\mathscr{H}^1_{\sigma_j}(\gamma)} \le C'', 
\]
and we reached a contradiction.
\end{proof}

\subsection*{Step 5: Higher regularity} Assume that $\Sigma' \Subset \mathring{\Sigma} \backslash E$ is a domain such that $X$ is $C^1$ in $V_{\Sigma'} \doteq \overline{D(\Sigma)} \cap (\Sigma' \times \R)$, and that $\rho\llcorner\Sigma'$ is valued in $\mathscr{C}^1(\overline{\Sigma'})$ and there bounded, say
\begin{align*}
    \norm{\frac{d\rho(u)}{dV_{\sigma_u}}}_{C^1(\overline{\Sigma'})} \leq \Lambda_1
\end{align*}
Then, by Proposition \ref{prop_appendix_sec}, for each domain $\Sigma''\Subset\Sigma'$ we deduce
\begin{align*}
    \norm{\rho_j(u_j)}_{C^1(\overline{\Sigma''})} + \norm{X_j}_{C^1(V_{\Sigma''})} \leq \Lambda_2
\end{align*}
hence the the mean curvatures $H_{u_j}$ satisfy the structure conditions of Bartnik \eqref{eq_structure conditions} with a uniform constant $\Lambda_2$ (see \cite{bartnik88} pag. 150). Then by \cite[Theorem 3.8]{bartnik88} the limit $u$ is $C^{2,\alpha}_\loc$ and spacelike away from the set of light segments of $u$ over $\Sigma'$, which we know to be empty by (i) in Theorem \ref{main result}. Equation \eqref{eq_mean_curvature_sigma_ij} is therefore locally uniformly elliptic in $\Sigma'$, thus the higher regularity of $u$ depending on that of $(\rho,X)$ follows by Schauder's estimates. In particular, if $\rho$ and $X$ are smooth then so is $u$.

\resumetoc

\appendix

\section{Application to the Born-Infeld theory}\label{appe_BI}

\stoptoc

In this section we recall Maxwell's and Born-Infeld's models for electrodynamics. Our goal is is to write down the equation satisfied by the electric potential generated by a fixed charge $\rho$ in a static Lorentzian spacetime $(V,\langle\cdot,\cdot\rangle)$ according to Born-Infeld's theory. We shall see that it can be interpreted as a prescribed mean curvature equation in $V$ endowed with a different Lorentzian metric $\g$, closely related to $\langle \cdot,\cdot \rangle$.

Born-Infeld's theory for electromagnetism only makes sense in $4$-dimensional spacetimes, and for this reason we shall restrict our treatment to ambient dimension $4$. However, the electrostatic equations \eqref{eq_electr_BI} are meaningful in any dimension.  

Recall that the Hodge dual in a pseudo-Riemannian manifold $(V,\langle\cdot, \cdot\rangle)$ is defined by
\begin{align*}
    \omega \wedge \star \eta = \langle \omega,\eta \rangle \, dV
\end{align*}
where $dV$ is the volume element of $\langle\cdot,\cdot\rangle$.
In particular, in a Lorentzian manifold of dimension $n$ it holds
\begin{align*}
    \star \star \omega = (-1)^{n(k-n)+1}\omega \qquad\forall\omega \ \text{$k$-form in $V$.} 
\end{align*}

\subsection*{Classical electrodynamics} Our treatment and conventions follow \cite{misner1973gravitation}. Let $(V,\langle \cdot,\cdot\rangle)$ be a $4$-dimensional spacetime and fix a timelike, future-pointing vector field $J\in\X(V)$ representing the electric current density. The \emph{Maxwell Lagrangian} is given by
\begin{align*}
    \mathsf{L}_{\mathsf{M}} = \int_{V} F \wedge \star F + \int_{V} A(J) \, dV, \qquad F = dA,
\end{align*}
and its Euler-Lagrange equations are 
\begin{align}\label{eq_maxwell}
    \begin{cases}
        dF = 0 \\
    d\star F = \star J_\flat.
    \end{cases}
\end{align}
A differential $2$-form $F$ that obeys \eqref{eq_maxwell} is the electromagnetic field generated by the charge $J$.

For a given observer $T\in \X(V)$, we define \emph{electric field} and \emph{magnetic field} relative to $T$ as
\begin{align*}
    \eps \doteq -\iota_TF \qquad \beta \doteq - \iota_T\star F
\end{align*}
and the \emph{electric displacement field} and \emph{magnetic field strength} according to the decomposition
\begin{align}\label{eq_F_decomposition}
    F = - \eps \wedge T_\flat + B, \qquad \star F = - \beta \wedge T_\flat - E.
\end{align}
Notice that $\iota_TE = \iota_TB = 0$, so, applying the Hodge dual to the first and comparing the two expressions one gets
\begin{align*}
    E =  \star (\eps\wedge T_\flat) \qquad B =  \star(\beta\wedge  T_\flat).
\end{align*}
If we write the current field as
\begin{align*}
    J = J^S + \rho T, \qquad \rho = -\g(J,T),
\end{align*}
since $\g(J^S,T)=0$, the Hodge dual of $J_\flat$ writes as 
\begin{align*}
    \star J_\flat = -j \wedge T_\flat + \rho \star T_\flat
\end{align*}
for some suitable $2$-form $j$ such that $\iota_T j = 0$.

Assume that $T$ is a \emph{synchronizable observer}, that is, $T = -\balpha \bnabla \tau$ for some smooth functions $\balpha>0$ and $\tau$. Then the orthogonal distribution $T^\perp$ is integrable (\cite[Proposition 12.30]{oneill}) and, at least locally, 
\begin{align*}
    \g = -\balpha^2 d\tau^2 + \bar\sigma
\end{align*}
for some symmetric $(2,0)$-tensor $\bar\sigma$ such that $\iota_T\bar\sigma=0$. Assume that the time function $\tau$ splits $V = \R\times S$ with $\{\tau = t\} = \{t\}\times S$ for some fixed $3$-manifold $S$. Denoting by $d_S$ the exterior derivative on $S$, since $B$ and $E$ have vanishing contraction with $T$, we can write
\begin{align*}
    dB = d_S B +  (\lie_{\partial_\tau}B) \wedge d\tau, \qquad dE = d_S E + (\lie_{\partial_\tau}E) \wedge d\tau.
\end{align*}
Here we are making a slight abuse of notation by writing $d_SB$ to denote the exterior derivative in $S$ of the restriction of $B$ to the space slice passing through a given point. Similarly
\begin{align*}
    d\balpha = d_S \balpha + \partial_\tau \balpha \wedge d\tau.
\end{align*}
Writing the Lorentzian volume form $dV$ as
\[
dV = (\eps \wedge T_\flat) \wedge \star (\eps \wedge T_\flat) = \eps \wedge T_\flat \wedge E = \eps \wedge E \wedge T_\flat  
\]
and decomposing $dV = dV_\sigma \wedge T_\flat$, where $dV_\sigma = - \star T_\flat$ is the volume form of $S$, from $\iota_T(\eps \wedge E) = 0$ we deduce that $E = \star_\sigma \eps$, where $\star_\sigma$ is the Hodge dual on the space slice with the induced metric. Proceeding analogously for $B$ we therefore obtain
\begin{align}\label{eq_sigma_hodge}
    E = \star_\sigma \eps, \qquad B = \star_\sigma \beta.
\end{align}
The first equation in \eqref{eq_maxwell} becomes
\begin{align*}
    (d_S(\balpha \eps) + \lie_{\partial_\tau} B) \wedge d\tau + d_S B = 0
\end{align*}
that is,
\begin{align*}
    \begin{cases}
        d_S B = 0 \\
        d_S(\balpha \eps) = -\lie_{\partial_\tau} B. 
    \end{cases}
\end{align*}
The same way the second one becomes
\begin{align}\label{eq_maxwell_non_homogeneous}
    \begin{cases}
        d_S E = -\rho \star T_\flat \\
        d_S(\balpha \beta) = \balpha j + \lie_{\partial_\tau} E.
    \end{cases}
\end{align}
Suppose we are in the electrostatic regime, that is,
\begin{align}\label{elettro_maxw}
    B=0, \qquad j=0, \qquad \lie_{\partial_\tau}\langle\cdot,\cdot\rangle.
\end{align}
A direct computation gives $i_{\partial_\tau} \lie_{\partial_\tau} \bar\sigma = 0$, so inserting into the third of \eqref{elettro_maxw} we get $\partial_\tau\balpha=0$ and $\lie_{\partial_\tau}\bar\sigma
=0$. Hence, $V$ is a warped product
\begin{align}\label{eq_static_spacetime}
    V = \R\times S  \qquad \langle \cdot,\cdot \rangle = -(\pi^*\alpha)^2 d\tau^2 + \pi^*\sigma
\end{align}
for a positive function $\alpha\in C^\infty(S)$ and a Riemannian metric $\sigma$ on $S$. Moreover, by the second in \eqref{eq_maxwell_non_homogeneous} we have $\lie_{\partial_\tau}E=0$, 
and by \eqref{eq_sigma_hodge} $\lie_{\partial_\tau}\eps = 0$. Thus, $E$ and $\eps$ can be seen as fields on $S$. In this setting
Maxwell's equations reduce to
\begin{align}\label{eq_maxwell_electrostatic}
    \begin{cases}
        d_S(\alpha \eps) = 0 \\
        d_S E = -\rho \star T_\flat.
    \end{cases}
\end{align}
Using the second in \eqref{eq_maxwell_electrostatic} and \eqref{eq_sigma_hodge} we have $\star_\sigma d_S \star_\sigma \eps = \rho \star_\sigma dV_\sigma = \rho$, that is
\begin{align*}
    \diver_\sigma \eps^\sharp = \rho
\end{align*}
where the musical isomorphism is performed this time with respect to $\sigma$. Assuming $S$ simply connected, the first equation allows us to write $du =\alpha \eps $ for some smooth function $u$ in $S$, hence, if $Du$ is its gradient with respect to $\sigma$ we obtain the electrostatic equation (Poisson's equation) in a static spacetime:
\begin{align*}
    \diver_\sigma (\alpha^{-1}Du) = \rho.
\end{align*}

\subsection*{Born-Infeld electrodynamics} 

In their proposal to overcome the infinity energy problem occurring in Maxwell's model, Born and Infeld in \cite{BI, BI_2} suggested to use the Lagrangian 
\begin{align}\label{Lagra_BI}
    \mathsf{L}_{\mathsf{BI}}(A) = \int_{V} \left( 1 - \sqrt{1+ P - \frac{Q^2}{2}} \right) \, dV + \int_{V} A(J) \, dV
\end{align}
where
\begin{align*}
    P = \langle F,F \rangle,  \qquad Q = \langle F, \star F\rangle , \qquad F = dA.
\end{align*}
(we follow the sign convention in \cite{Yang}, opposite to that in \cite{BI_2}). They also indicated the  Lagrangian
\begin{align}\label{Lagra_B}
    \mathsf{L}_{\mathsf{B}}(A) = \int_{V} \left( 1 - \sqrt{1+ P} \right) \, dV + \int_{V} A(J) \, dV,
\end{align}
first proposed in \cite{born}, as a possible alternative model, see \cite[(2.27) and (2.28)]{BI_2}. However, later works in \cite{boillat,plebanski} pointed out that $\mathsf{L}_{\mathsf{BI}}$ has the distinctive feature of generating an electrodynamics free from the phenomenon of birefringence (cf. \cite{Kiessling-legacy}). While the term $Q$ in $\mathsf{L}_{\mathsf{BI}}$ forces to restrict to $4$-dimensional $V$, the Lagrangian $\mathsf{L}_{\mathsf{B}}$ is meaningful in any dimension, and the following computations show that the resulting electrostatics lead to the same equation, which is \eqref{eq_born_infeld_electrostatics_intro}. 

%


The Euler Lagrange equations of $\mathsf{L}_{\mathsf{BI}}$ are
\begin{align}
    \begin{cases}
        dF = 0 \\
        d\star G = \star J_\flat
    \end{cases}
\end{align}
where
\begin{align*}
    G = w \left( F - Q\star F\right) \qquad w = \frac{1}{\sqrt{1 + P - \frac{Q^2}{2}}}.
\end{align*}
For a fixed observer $T$ decompose $F$ and its Hodge dual as in \eqref{eq_F_decomposition} and set $h = -\iota_T\star G$ and $D$ so that
\begin{align*}
    \star G = - h \wedge T_\flat - D.
\end{align*}
It is immediate to see that
\begin{align*}
    h = w(\beta + Q\eps), \qquad D = w(E - QB).
\end{align*}
Since Born-Infeld equations have the same structure as Maxwell's, as in the previous section we see that, for a synchronizable observer, they are equivalent to the system
\begin{align}\label{eq_four_BI_equations}
    \begin{cases}
    d_S D = -\rho \star T_\flat \\
    d_S B =0 \\
    d_S(\balpha\eps) = - \lie_{\partial_\tau} B \\
    d_S(\balpha h) = \balpha j + \lie_{\partial_\tau} D.
    \end{cases}
\end{align}
As before, in the case of a static spacetime \eqref{eq_static_spacetime}, the Born-Infeld equation for electrostatics (that is, $j=B=0$) reduce to
\begin{align}\label{eq_electr_BI}
    \begin{cases}
        d_S(\alpha\eps)=0 \\
        d_S D = -\rho \star T_\flat.
    \end{cases}
\end{align}
Moreover, $Q=0$ and $ P = -\abs{\eps}_\sigma^2$, the electric displacement field $D$ is simply $D = wE$ and by the third in \eqref{eq_four_BI_equations}, if $S$ is simply connected $\eps = \alpha^{-1}d_Su$ for some $u\in C^\infty(S)$. Applying again \eqref{eq_sigma_hodge} we finally obtain the Born-Infeld equation for electrostatics in a static spacetime:
\begin{align}\label{eq_born_infeld_electrostatics}
    \diver_\sigma \left(\frac{\alpha^{-1}Du}{\sqrt{1 - \alpha^{-2}|Du|^2}}\right) = \rho.
\end{align}
A computation similar to those performed in Section \ref{sec_prelim} (see also (2.17) in \cite{bartnik84}) show that the mean curvature of the graph $M_u\subseteq \R\times S$ of a function $u\in C^\infty(S)$ with respect to the metric $\g = -(\pi^*f)^2 \, d\tau^2 + \pi^*\sigma$, with $f\in C^\infty(S)$, is given by
\begin{align*}
    H_u = \diver_\sigma \left(\frac{fDu}{\sqrt{1 - f^{2}|Du|^2}}\right) + \g(N,\bnabla_{\tilde T}\tilde T).
\end{align*}
Here $\tilde T = -f\bnabla\tau = f^{-1}\partial_\tau$, $\bnabla$ is the connection of $\g$ and $N$ is the $\g$-normal future pointing normal of $M_u$. Choosing $f = \alpha^{-1}$, and taking into account the observations on the dimension of $V$ at the beginning of this subsection, we can state the following:

\begin{prop}
    Let $(S,\sigma)$ be a $m$-dimensional Riemannian manifold, let $\alpha\in C^\infty(S)$ be a positive function, and consider on the product $V=\R\times S$ the Lorentzian warped metrics
    \begin{align*}
        \langle \cdot,\cdot\rangle = -(\pi^*\alpha)^2d \tau^2 + \pi^*\sigma, \qquad \g = -(\pi^*\alpha)^{-2}d \tau^2 + \pi^*\sigma,
    \end{align*}
    where $\pi:\R\times S\to S$ is the canonical projection. For a given charge distribution $\rho \in C^\infty(S)$, let $u$ be the electrostatic potential generated by $\rho$ according to Born-Infeld's theory in $(V, \langle \cdot ,\cdot \rangle)$, that is, $u$ solves \eqref{eq_born_infeld_electrostatics}. Then, the graph $M_u$ of $u$ in $(V,\g)$ is a spacelike hypersurface whose mean curvature is 
    \begin{align*}
        H_u = \rho + \g(X,N_u), \qquad \text{where} \quad X = \bnabla_{\tilde T}\tilde T, \quad \tilde T = \alpha \partial_\tau.
    \end{align*}    
\end{prop}

    
Fix now a bounded domain $\Omega\Subset S$ and a spacelike function $\varphi \in C^\infty(\overline\Omega)$, and consider the Born-Infeld electrostatic problem
\begin{align}
    \begin{cases}\label{eq_born_infeld}
        \diver_\sigma \left(\frac{\alpha^{-1}Du}{\sqrt{1 - \alpha^{-2}|Du|^2}}\right) = \rho &\quad \text{in $\Omega$} \\
        u = \varphi &\quad\text{on $\partial\Omega$}.
    \end{cases}
\end{align}
By writing 
\[
\g = -(\pi^*\alpha)^{-2} \big( d \tau^2 + \pi^* \hat \sigma\big), \qquad \text{with } \, \hat \sigma = \alpha^2 \sigma, 
\]
and recalling Remark \ref{rem_cauchycompact_intro}, as a consequence of Theorem \eqref{main result} we obtain the following existence, uniqueness and regularity result for \eqref{eq_born_infeld} which generalizes \cite[Theorem 1.11]{BIMM}.

\begin{theorem}\label{te_forBI}
    Let $(S,\sigma)$ be a $2$-dimensional Riemannian manifold, let $\Omega\Subset S$ be a bounded smooth domain and fix a Radon measure $\rho = \rho_{\rm S} + \tfrac{d\rho}{dV_\sigma}\, dV_\sigma\in\mathscr{M}(\overline\Omega)$. Assume that there exists a compact subset $E\Subset \Omega$ such that 
    \[
    \mathscr{H}_\sigma^1(E) = 0, \qquad \supp \rho_{\rm S} \subseteq E, \qquad \tfrac{d\rho}{dV_\sigma}\in L^2_{\loc}(\Omega\backslash E). 
    \]
    Denote by $D$ and $|\cdot |$ the connection and norm of $(S,\sigma)$. Then, for any $\varphi\in C^\infty(\overline{\Omega})$ with $|D\varphi|< \alpha$ in $\overline\Omega$ there exists a unique solution $u$ to \eqref{eq_born_infeld}. The solution also satisfies the following properties.
    \begin{enumerate}
        \item $u$ has no light segments, namely, 
        \[
        |u(y)-u(x)|<\operatorname{d}_{\hat\sigma}(y,x) \qquad \forall \, x,y \in \Omega.
        \]
        where $\hat\sigma = \alpha^{2}\sigma$. 
        \item The energy density $w = \frac{1}{\sqrt{1- \alpha^{-2}\abs{Du}^2}}$ satisfies  
        \[
        \begin{array}{l}
        \disp w \in L^1(\Omega), \qquad w \ln(1+w) \in L^1_\loc(\Omega \backslash E), \\[0.3cm]
        \disp w\abs{\hat D^2u}_{\hat \sigma}^2 + w^3 \abs{\hat D^2u(\hat Du)}_{\hat \sigma}^2 + w^5 \big[\hat D^2u( \hat Du, \hat Du)\big]^2\in L^1_\loc(\Omega \backslash E),
        \end{array}
        \]
        where $\hat D$ is the connection of $\hat \sigma$. In particular, $w  \abs{D^2 u}^2, \abs{D\log w} \in L^1_\loc(\Omega \backslash E)$.
    \end{enumerate}
    If, in addition, there exists a domain  $\Omega'\Subset\Omega\backslash E$ such that $\tfrac{d\rho}{dV_\sigma}\in C^1(\Omega')$ then $\mathscr{S}\cap \Omega' = \emptyset$ (so $|Du|< \alpha$ in $\Omega'$) and $u \in C^{2,\beta}_\loc(\Omega')$.
\end{theorem}

Uniqueness is guaranteed by the fact that, in this setting, \eqref{eq_born_infeld} is the Euler Lagrange equation of the convex functional
\begin{align*}
    I_\rho (u) = \int_{\Omega} \left( 1 - \sqrt{1 - \alpha^{-2}|Du|^2} \right) \alpha \, dV_\sigma + \int_\Omega u \, d\rho,
\end{align*}
thus comparison holds (see \cite[Lemma 2.12]{bpd}).

\resumetoc

\section{Proof of Proposition \ref{prop_appendix_sec}}\label{Appe_approx}

\stoptoc

\subsection*{Convolution} The first step is to construct a convolution operator for Radon measures on $\Sigma$. Recall that 
\begin{align*}
    \mathscr{C}^k(\Sigma) &= \left\{\mu\in\cM(\Sigma) \ : \ \mu << dV_{\sigma_u} \  \text{and} \ \frac{d\mu}{dV_{\sigma_u}}\in C^k(\Sigma), \quad \forall u\in\Y \cap C^k(\Sigma) \right\} \\
    \mathscr{L}^p(\Sigma') &= \left\{\mu\in\cM(\Sigma) \ : \ \mu << dV_{\sigma_u} \  \text{and} \ \frac{d\mu}{dV_{\sigma_u}}\in L^p(\Sigma'), \quad \forall u\in\Y \right\}
\end{align*}

Recall that given $\Sigma' \subseteq \Sigma$ we define
\[
\Sigma'_\delta = \big\{x \in \Sigma' \ : \ \di_{\sigma_\varphi}(x,\partial \Sigma) > \delta\big\}.
\]

\begin{lemma}\label{lem_convolution_op}
    Assume that \ref{cauchy compatto} holds. Then, there exists $\epsilon_0>0$ and a family of operators $\{\Upsilon_\epsilon\}_{\epsilon<\epsilon_0}$
    \begin{align*}
        \Upsilon_\epsilon \ : \ \cM(\Sigma) &\longrightarrow \cM(\Sigma) \\
        \mu &\longmapsto \mu_\epsilon
    \end{align*}
    such that 
    \begin{enumerate}
        \item $\Upsilon_\epsilon(\cM(\Sigma))\subseteq \mathscr{C}^\infty(\Sigma)$ for each $\epsilon\in(0,\epsilon_0$);
        \item\label{item_convolution_convergence} for each $\mu\in\cM(\Sigma)$, $\mu_\epsilon\overset{*}{\rightharpoonup} \mu$ in $\cM(\Sigma)$ as $\epsilon\to 0$;
        \item\label{item_measure_norm_bound} there exists a constant $c$ independent of $\epsilon$ such that
        \begin{align*}
            \norm{\mu_\epsilon}_{\cM(\Sigma)} \leq c \norm{\mu}_{\cM(\Sigma)} \qquad \forall \mu\in\cM(\Sigma)
        \end{align*}
        \item\label{item_C1_norm_bound} for each $\epsilon\in(0,\epsilon_0)$ there exists a constant $C_{\epsilon}$ such that for every $\mu\in\cM(\Sigma)$
        \begin{align*}
            \norm{\frac{d\mu_\epsilon}{dV_{\sigma_u}}}_{C^1(\Sigma)} \leq C_{\epsilon}\norm{\mu}_{\cM(\Sigma)} \qquad u\in\Y\cap C^1(\Sigma) ;
        \end{align*}
        and if $\mu\in\mathscr{C}^1(\Sigma')$ for some $\Sigma'\subseteq\Sigma$, then for every $\delta>0$ there exists a constant $C_{\delta}$ such that 
        \begin{align*}
            \norm{\frac{d\mu_\epsilon}{dV_{\sigma_u}}}_{C^1(\Sigma'_\delta)} \leq C_{\delta} \norm{\frac{d\mu}{dV_{\sigma_u}}}_{C^1(\Sigma')} \qquad \forall u\in\Y\cap C^1(\Sigma)
        \end{align*}
        for $\epsilon$ small enough.
        \item\label{item_Lp_norm_bound} if $\mu\in\mathscr{L}^p(\Sigma')$ for some domain $\Sigma' \subseteq \Sigma$, then for each $\delta>0$ there exists a constant $C_{p,\delta}$ such that
        \begin{align*}
            \norm{\frac{d\mu_\epsilon}{dV_{\sigma_u}}}_{L^p(\Sigma'_\delta)} \leq C_{p,\delta}\norm{\frac{d\mu}{dV_{\sigma_u}}}_{L^p(\Sigma')} \qquad\forall u\in\Y
        \end{align*}
        for $\epsilon$ small enough. 
    \end{enumerate}
\end{lemma}
\begin{proof}
    Extend $\Sigma$ to a smooth, spacelike hypersurface $\Sigma^1$ with $\Sigma \Subset \mathring{\Sigma}^1$. Choose also bounded open sets $V_i\Subset U_i\Subset \mathring{\Sigma}^1$, $1\leq i\leq n$, such that
    \begin{align*}
        \Sigma \subset V \doteq \bigcup_{i=1}^n V_i, \qquad \text{$U_i$ supports a chart $\varphi_i:U_i\to \R^m$ with $\varphi_i(U_i) = \R^m$.}
    \end{align*}
    Let $\{\zeta_i\}_i$ be a partition of unity subordinated to $V_i$ and set $\tilde\zeta_i = \zeta_i \circ \varphi_i^{-1}\in C_c(\R^m)$. Since $\supp \tilde\zeta \cap \varphi_i(\Sigma\cap U_i)\Subset \varphi_i(V_i)$, there exists $\epsilon_0>0$ such that
    \begin{align}\label{eq_lemma_conv_1}
        \overline{B_\epsilon(\supp\tilde\zeta_i\cap\varphi_i(V_i))} \Subset \varphi_i(V_i) \qquad\forall\epsilon \le \epsilon_0, \quad \forall 1\leq i\leq n.
    \end{align}
    where $B_\epsilon$ is taken with respect to the Euclidean metric on $\R^m$.
    We regard any given $\mu\in\cM(\Sigma)$ as an element of $\cM(V)$ by pushing forward via the inclusion $\Sigma\hookrightarrow V$ and let $\mu_i = (\varphi_i)_*(\zeta_i\mu)\in\cM(\R^m)$ so that
    \begin{align*}
        \int_{\R^m} f \, d\mu_i \doteq \int_{U_i} (\varphi_i^*f)\zeta_i \, d\mu\qquad \forall f\in C_c(\R^m)
    \end{align*}
    and $\norm{\mu_i}_{\cM(\R^m)} = \norm{\zeta_i \mu}_{\cM(\Sigma)} \le \norm{ \mu}_{\cM(\Sigma)}$. Fix a regularization kernel $\{\Phi_\epsilon\}$ in $\R^m$ as usual:
    \begin{align*}
        \Phi_\epsilon (x) = \frac{1}{\epsilon^n}\Phi\left(\frac{x}{\epsilon}\right) \qquad\Phi\in C_c^\infty(\R^m) \qquad\supp\Phi = \overline{B_1(0)} \qquad \norm{\Phi}_{L^1(\R^m)} = 1
    \end{align*}
    and let $\Phi_\epsilon * \mu_i$ the measure on $\R^m$ defined by
    \begin{align*}
        \int_{\R^m}f \, d(\Phi_\epsilon * \mu_i) \doteq \int_{\R^m} f(x)\int_{\R^m} \Phi_\epsilon(x-y) \, d\mu_i(y) \, dx,
    \end{align*}
    that is, the classical regularization of $\mu_i$. Notice that, by \eqref{eq_lemma_conv_1}, $\Phi_\epsilon*\mu_i$ is well defined and supported in $\varphi_i(V_i)$ for $\epsilon<\epsilon_0$. We then define $\Upsilon_\epsilon$ by
    \begin{align*}
        \Upsilon_\epsilon\mu = \mu_\epsilon \doteq \left(\sum_{i=1}(\varphi_i^{-1})_*(\Phi_\epsilon*\mu_i)\right) \llcorner \, \Sigma\in\cM(\Sigma).
    \end{align*}
    
    If we fix $u\in\Y\cap C^\infty(\Sigma)$ then by the change of variables formula we have
    \begin{align*}
        \frac{d(\varphi_i^{-1})_*(\Phi_\epsilon*\mu_i)}{dV_{\sigma_u}} = \frac{1}{\sqrt{\det\sigma_u^{\varphi_i}}}\left(\frac{d(\Phi_\epsilon * \mu_i)}{dx}\circ\varphi_i\right) \qquad \text{on } \, \Sigma, 
    \end{align*}
    where $\det \sigma_u^{\varphi_i}:U_i\to\R$ is the determinant of the matrix associated to $\sigma_u$ via the chart $\varphi_i$. Therefore, 
    \[
    \frac{d(\varphi_i^{-1})_*(\Phi_\epsilon*\mu_i)}{dV_{\sigma_u}}
    \]
    inherits the same regularity as $u$ on $\Sigma$, so by definition $\mu_\epsilon\in\mathscr{C}^\infty(\Sigma)$ with density
    \begin{align*}
        \frac{d\mu_\epsilon}{dV_{\sigma_u}} = \sum_{i=1}^n\frac{1}{\sqrt{\det\sigma_u^{\varphi_i}}}\left(\frac{d(\Phi_\epsilon*\mu_i)}{dx}\circ\varphi_i\right)\Bigg\vert_{\Sigma}.
    \end{align*}
    
    To prove \eqref{item_C1_norm_bound}, first notice that we have the following estimate for the Euclidean gradient of the density of $\Phi_\epsilon*\mu_i$
    \begin{align*}
        \abs{D_{\R^m}\left(\frac{d(\Phi_\epsilon*\mu_i)}{dx}\right)(y)} &= \abs{\int_{\R^m} D_{\R^m}\Phi_\epsilon(y-z) \, d\mu_i(z)} \\
        &\leq \int_{\R^m}\epsilon^{-n-1}\abs{D_{\R^m}\Phi\left(\tfrac{y-z}{\epsilon}\right)} \, d\mu_i(z) \\
        &\leq \epsilon^{-1}\norm{D_{\R^m}\Phi}_{L^1(\R^m)}\norm{\mu_i}_{\cM(\R^m)} \\
        &\leq \epsilon^{-1}\norm{D_{\R^m}\Phi}_{L^1(\R^m)}\norm{\mu}_{\cM(\Sigma)}.
    \end{align*}
    Also, as $u\in\Y \cap C^1(\Sigma)$ and \ref{cauchy compatto} holds, the $C^1(\Sigma)$ norm of $\sqrt{\det\sigma_{u}^{\varphi_i}}$ is bounded by a constant that depends only on $\Sigma$ and $\tau$, and as a consequence
    \begin{align*}
        \abs{D_{\sigma_u}\left(\frac{d\mu_\epsilon}{dV_{\sigma_u}}\right)} \leq Cn\epsilon^{-1} \norm{D_{\R^m}\Phi}_{L^1(\R^m)}\norm{\mu}_{\cM(\Sigma)} \doteq C_{\epsilon} \norm{\mu}_{\cM(\Sigma)}.
    \end{align*}
    If moreover $\mu\in \mathscr{C}^1(\Sigma')$ then $\mu_i$ has $C^1$ density and, since $\varphi_i(\Sigma_\delta'\cap V_i)\Subset \varphi_i(\Sigma'\cap V_i)$, by standard properties of the convolution on $\R^m$, for every $\delta>0$ we have
    \begin{align*}
        \norm{\partial_i\left(\frac{d(\Phi_\epsilon*\mu_i)}{dx}\right)}_{C(\varphi_i(\Sigma'_\delta\cap V_i))} &= \norm{\Phi_\epsilon*\partial_i\frac{d\mu_i}{dx}}_{C(\varphi_i(\Sigma'_\delta\cap V_i))} \\ &\leq \norm{\Phi_\epsilon}_{L^1(\R^m)}\norm{\frac{d\mu_i}{dx}}_{C^1(\varphi_i(\Sigma'\cap V_i))}
    \end{align*}
    for $\epsilon$ small enough. Now since by Lemma \eqref{conseguenze di C} we have an a priori $C^1$ bound on $\sqrt{\det\sigma_u^{\varphi_i}}$ for any $u\in\Y \cap C^1(\Sigma)$ and $1\leq i\leq n$, the second estimate in \eqref{item_C1_norm_bound} follows.
    
    Similarly, 
    \begin{align*}
        \norm{\frac{d\mu_\epsilon}{dV_{\sigma_u}}}^p_{L^p(\Sigma'_\delta,\sigma_u)} &= \int_{\Sigma'_\delta}\abs{\sum_{i=1}^n \frac{1}{\sqrt{\det\sigma_u^{\varphi_i}}}\left(\frac{d(\Phi_\epsilon * \mu_i)}{dx}\circ\varphi_i\right) }^p \, dV_{\sigma_u} \\
        &\leq C_p\sum_{i=1}^n \norm{\frac{d(\Phi_\epsilon*\mu_i)}{dx}}_{L^p(\varphi_i(\Sigma_\delta'\cap V_i))}^p \\ 
        &\leq C_{p,\delta}\sum_{i=1}^n \norm{\frac{d\mu_i}{dx}}^p_{L^p(\varphi_i(\Sigma'\cap V_i))} \\
        &\leq C_{p,\delta}\norm{\frac{d\mu}{dV_{\sigma_u}}}_{L^p(\Sigma')}^p
    \end{align*}
    and also \eqref{item_Lp_norm_bound} is shown.
    
    Now we prove \eqref{item_convolution_convergence}. Given a test function $\eta\in C(\Sigma)$, consider an extension of $\eta$ to a function in $C_c(V)$ with the same norm. By definition, setting $\tilde\eta_i=\eta\circ\varphi_i^{-1}$ we have
    \begin{align*}
        \int_\Sigma \eta \, d\mu_\eps = \sum_{i=1}^n \int_{\R^m} \tilde\eta_i(x) \int_{\R^m} \Phi_\epsilon(x-y) \, d\mu_i(y) \, dx
    \end{align*}
    and
    \begin{align*}
        \int_\Sigma \eta \, d\mu = \sum_{i=1}^n \int_\Sigma \eta\zeta_i \, d\mu = \sum_{i=1}^n \int_{\R^m} \tilde\eta_i \, d\mu_i,
    \end{align*}
    hence
    \begin{align*}
        \abs{\int_\Sigma \eta \, d\mu_\epsilon - \int_\Sigma \eta \, d\mu} &\leq \sum_{i=1}^n\abs{\int_{\R^m} \tilde\eta(x) \int_{\R^m}\Phi_\epsilon(x-y) \, d\mu_i(y)\, dx - \int_{\R^m}\tilde\eta(y) \, d\mu_i(y) } \\
        &= \sum_{i=1}^n\abs{   \int_{\R^m} \tilde\eta_i(y + z) \int_{\R^m}\Phi_\epsilon(z) \, d\mu_i(y) \, dz -\int_{\R^m}\int_{\R^m}\tilde\eta_i(y) \Phi_\epsilon(z) \, \mu_i(y)\, dz } \\
        &=\sum_{i=1}^n \abs{\int_{\R^m} \int_{\R^m} \Phi_\epsilon(z)[\tilde\eta_i(y+z) - \tilde\eta_i(y)] \, \, d\mu_i(y) \, dz }.
    \end{align*}
    Now, being $\tilde\eta_i$ continuous on the compact set $\overline{B_{\epsilon_0}(\supp\tilde\zeta_i \cap \varphi_i(\Sigma \cap U_i))}$ it is uniformly continuous there, that is, for every $\delta>0$ there exists $\epsilon = \epsilon(\delta,\eta)\in(0,\epsilon_0)$ such that, for every $y\in \supp\tilde\zeta_i \cap \varphi(\Sigma\cap U_i)$ and $z\in\R^m$ one has
    \begin{align*}
        |z|<\epsilon \ \Longrightarrow \ |\tilde\eta_i(y+z) -\tilde\eta_i(y)| \leq \delta
    \end{align*}
    for every $1\leq i\leq n$. We have thus showed that
    \begin{align}\label{eq_limit}
        \forall\delta>0 \quad \exists\epsilon(\delta,\eta)>0 \quad : \quad \epsilon<\epsilon(\delta,\eta) \quad \Longrightarrow \quad
        \abs{\int_\Sigma \eta \, d\mu - \int_\Sigma \eta \, d\mu_\epsilon} \leq \delta n\norm{\mu}_{\mathscr{M}(\Sigma)},
    \end{align}
    and by the arbitrariness of $\delta>0$ we have $\mu_\epsilon \overset{*}{\rightharpoonup}\mu$ as $\epsilon\to 0$.
    
    Finally, to show \eqref{item_measure_norm_bound} notice that
    \begin{align*}
        \abs{\int_\Sigma \eta \, d\mu_\epsilon} \leq \sum_{i=1}\abs{\int_{\R^m} \tilde\eta_i(x) \int_{\R^m} \Phi_\epsilon(x - y) \, d\mu_i(y) \, dx } \leq n\norm{\eta}_{C(\Sigma)} \norm{\mu}_{\cM(\Sigma)}
    \end{align*}
    and thus $\norm{\mu_\epsilon}_{\cM(\Sigma)}\leq n\norm{\mu}_{\cM(\Sigma)}$ for each $\epsilon$.
\end{proof}

Now we are ready to prove Proposition \ref{prop_appendix_sec}, that we rewrite for the convenience of the reader.

\begin{prop}\label{prop_appendix}
    Assume \ref{cauchy compatto}. Let $\rho: (\Y, \|\cdot \|_{C(\Sigma)}) \to\cM(\Sigma)$ be a continuous map. There exists a sequence of functions
    \begin{align*}
        \rho_j \ : \ \Y \longrightarrow \mathscr{C}^\infty(\Sigma)
    \end{align*}
    such that the following hold for $j>>1$.
    \begin{enumerate}[label=(\roman*)]
        \item
        For any $\{u_j\}\subseteq\Y$ we have
        \begin{align*}
            u_j \to u \quad \text{in $C(\Sigma)$} \quad \Longrightarrow \quad \rho_j(u_j)\overset{*}{\rightharpoonup} \rho(u) \quad \text{in $\cM(\Sigma)$}.
        \end{align*}
        \item
        There exists a constant $C_\rho$ such that
        \begin{align*}
            \norm{\rho_j(u)}_{\cM(\Sigma)} \leq C_\rho \qquad \forall u\in \Y.
        \end{align*}
        \item
        For each $j$ there exists a constant $C_{\rho,j}$ such that
        \begin{align*}
            \norm{\frac{d\rho_j(u)}{dV_{\sigma_u}}}_{C^1(\Sigma)} \leq C_{\rho,j} \qquad \forall u\in\Y\cap C^1(\Sigma)
        \end{align*}
        and if $\rho$ is valued in $\mathscr{C}^1(\Sigma')$ for some $\Sigma'\Subset\Sigma$, then  for every $\delta>0$ there exists a constant $C_{\Sigma',\delta}$ such that
        \begin{align*}
            \norm{\frac{d\rho_j(u)}{dV_{\sigma_u}}}_{C^1(\Sigma'_\delta)} \leq C_{\Sigma',\delta} \norm{\frac{d\rho(u)}{dV_{\sigma_u}}}_{C^1(\Sigma')} \qquad \forall u\in\Y\cap C^1(\Sigma).
        \end{align*}
        \item
        If $\rho$ is valued in $\mathscr{L}^p(\Sigma')$ for some $\Sigma' \Subset \Sigma$ and $p\in[1,\infty)$, then for every $\delta>0$ there exists a constant $C_{p,\Sigma',\delta}$ such that 
        \begin{align*}
            \norm{\frac{d\rho_j(u)}{dV_{\sigma_u}}}_{L^p(\Sigma'_\delta)} \leq C_{p,\Sigma',\delta}\norm{\frac{d\rho(u)}{dV_{\sigma_u}}}_{L^p(\Sigma')} \qquad\forall u\in\Y.
        \end{align*}
    \end{enumerate}
    Furthermore if $X\in\X(\overline{D(\Sigma)})$ is a continuous vector field, there exists a sequence of smooth vector fields $\{X_j\}$ and a constant $\Lambda\geq 0$ such that 
    \begin{align}\label{eq_convergence_Jj}
        X_j\to X \quad \text{in $C(\overline{D(\Sigma)})$}
    \end{align}
    and
    \begin{align}\label{eq_Hp_Lambda_app}
        \abs{\g(X_j,N_{u})} \leq \Lambda w_{u} \qquad \forall u\in\Y
    \end{align}
    where $w_{u} = - \g(T,N_{u})$ and $N_{u}$ is the future pointing unit normal to the graph of $u$. Moreover, if $X$ is $C^1$ on a compact subset $K\subseteq\overline{D(\Sigma)}$ then the $C^1$ norm of each $X_j$ on $K$ satisfies
    \begin{align}
        \norm{X_j}_{1,K} \leq \Lambda.
    \end{align}
        
\end{prop}

\begin{proof}
    Fix a sequence $\epsilon_j \to 0^+$ and define
    \[
    \rho_j = \Upsilon_{\epsilon_j} \circ \rho,
    \]
    where $\Upsilon_{\epsilon_j}$ is the operator in Lemma \ref{lem_convolution_op}. Since $\rho$ is continuous and $\Y$ is compact in $C(\Sigma)$, so is $\rho(\Y)$ and thus $\norm{\rho(u)}_{\cM(\Sigma)}\leq C_\rho$ for some constant $C_\rho$ and every $u\in\Y$. Hence, \textit{(ii)} follows by using \eqref{item_measure_norm_bound} in Lemma \ref{lem_convolution_op}.
    
    Having fixed $\eta\in C(\Sigma)$, by the continuity of $\rho$ and by \eqref{eq_limit} applied to $\mu = \rho(u_j)$ (so that $\mu_{\epsilon_j}=\rho_j(u_j)$) we deduce that  for $\delta>0$ there exists $j_0 = j_0(\delta,\eta)$ such that, for any $j\geq j_0$ it holds
    \begin{align*}
        \abs{\int_\Sigma \eta \, d\rho_j(u_j) - \int_\Sigma \eta \, d\rho(u)} &\leq \abs{\int_\Sigma \eta \, d\rho(u_j) - \int_\Sigma \eta \, d\rho(u)} + \abs{\int_\Sigma \eta \, d\rho_j(u_j) - \int_\Sigma \eta \, d\rho(u_j)} \\
        &\leq \delta + \delta n \norm{\rho(u_j)}_{\cM(\Sigma)} \\
        &\leq \delta(1 + nC_\rho)
    \end{align*}
    and also \textit{(i)} is proved. 

    Items $(3),(4),(5)$ immediately follow by Lemma \ref{lem_convolution_op}. 

    For the last statements, take $\{X_j\}$ to be any approximation of $X$ by Riemannian convolution in the metric $\ee$ defined in \eqref{eq_eucl_metric}. While \eqref{eq_convergence_Jj} follows by standard properties of Riemannian convolution, assertion \eqref{eq_Hp_Lambda_app} follows by \eqref{eq_bound_euclidean_norm}
    and the fact that $\norm{X_j}$ is uniformly bounded in $\overline{D(\Sigma)}$.
\end{proof}

\resumetoc

\printbibliography

@article {mccannsamann,
    AUTHOR = {McCann, Robert J. and S\"amann, Clemens},
     TITLE = {A {L}orentzian analog for {H}ausdorff dimension and measure},
   JOURNAL = {Pure Appl. Anal.},
  FJOURNAL = {Pure and Applied Analysis},
    VOLUME = {4},
      YEAR = {2022},
    NUMBER = {2},
     PAGES = {367--400},
}

@article{Bernal:2005aa,
	author = {Antonio N. Bernal and Miguel S{\'a}nchez},
    journal = {Commun.Math.Phys.},
	pages = {43-50},
	title = {Smoothness of time functions and the metric splitting of globally hyperbolic spacetimes},
    doi = {https://doi.org/10.1007/s00220-005-1346-1},
    eprint = {gr-qc/0401112},
	url = {https://arxiv.org/pdf/gr-qc/0401112.pdf},
	volume = {257},
	year = {2005}
    }

@article{gerhardt83,
    author = {Claus Gerhard},
    title = {$H$-surfaces in Lorentzian manifolds},
    journal = {Commun.Math.Phys.},
    year = {1983},
    volume = {89},
    pages = {523–553}
}

@article {klyachin_desc,
    AUTHOR = {Klyachin, A. A.},
     TITLE = {Description of a set of entire solutions with singularities of
              the equation of maximal surfaces},
   JOURNAL = {Mat. Sb.},
  FJOURNAL = {Matematicheski\u i\ Sbornik},
    VOLUME = {194},
      YEAR = {2003},
    NUMBER = {7},
     PAGES = {83--104},
      ISSN = {0368-8666,2305-2783},
   MRCLASS = {35J60 (35B05 35B40 53A40)},
  MRNUMBER = {2020379},
MRREVIEWER = {Sergey\ G.\ Pyatkov},
       DOI = {10.1070/SM2003v194n07ABEH000753},
       URL = {https://doi-org.pros2.lib.unimi.it/10.1070/SM2003v194n07ABEH000753},
}

@article {flaherty,
    AUTHOR = {Flaherty, F. J.},
     TITLE = {The boundary value problem for maximal hypersurfaces},
   JOURNAL = {Proc. Nat. Acad. Sci. U.S.A.},
  FJOURNAL = {Proceedings of the National Academy of Sciences of the United
              States of America},
    VOLUME = {76},
      YEAR = {1979},
    NUMBER = {10},
     PAGES = {4765--4767},
      ISSN = {0027-8424},
   MRCLASS = {53C50 (58D30 83C99)},
  MRNUMBER = {545611},
MRREVIEWER = {Paul\ E.\ Ehrlich},
       DOI = {10.1073/pnas.76.10.4765},
       URL = {https://doi-org.pros2.lib.unimi.it/10.1073/pnas.76.10.4765},
}

@article {ab,
    AUTHOR = {Audounet, Jacques and Bancel, Daniel},
     TITLE = {Sets with spatial boundary and existence of submanifolds with
              prescribed mean extrinsic curvature of a {L}orentzian
              manifold},
   JOURNAL = {J. Math. Pures Appl. (9)},
  FJOURNAL = {Journal de Math\'ematiques Pures et Appliqu\'ees. Neuvi\`eme
              S\'erie},
    VOLUME = {60},
      YEAR = {1981},
    NUMBER = {3},
     PAGES = {267--283},
      ISSN = {0021-7824,1776-3371},
   MRCLASS = {53C50 (53C40 58D17 83C99)},
  MRNUMBER = {633005},
MRREVIEWER = {Mauro\ Francaviglia},
}

@article {mt,
    AUTHOR = {Marsden, Jerrold E. and Tipler, Frank J.},
     TITLE = {Maximal hypersurfaces and foliations of constant mean
              curvature in general relativity},
   JOURNAL = {Phys. Rep.},
  FJOURNAL = {Physics Reports. A Review Section of Physics Letters},
    VOLUME = {66},
      YEAR = {1980},
    NUMBER = {3},
     PAGES = {109--139},
      ISSN = {0370-1573,1873-6270},
   MRCLASS = {83C99 (57R30 58D25 58D30)},
  MRNUMBER = {598585},
MRREVIEWER = {F.\ J.\ Flaherty},
       DOI = {10.1016/0370-1573(80)90154-4},
       URL = {https://doi-org.pros2.lib.unimi.it/10.1016/0370-1573(80)90154-4},
}

@article {huisken_yau,
    AUTHOR = {Huisken, Gerhard and Yau, Shing-Tung},
     TITLE = {Definition of center of mass for isolated physical systems and
              unique foliations by stable spheres with constant mean
              curvature},
   JOURNAL = {Invent. Math.},
  FJOURNAL = {Inventiones Mathematicae},
    VOLUME = {124},
      YEAR = {1996},
    NUMBER = {1-3},
     PAGES = {281--311},
      ISSN = {0020-9910,1432-1297},
   MRCLASS = {53C20 (83C99)},
  MRNUMBER = {1369419},
MRREVIEWER = {Alan\ D.\ Rendall},
       DOI = {10.1007/s002220050054},
       URL = {https://doi-org.pros2.lib.unimi.it/10.1007/s002220050054},
}

@book {choquetbruhat,
    AUTHOR = {Choquet-Bruhat, Yvonne},
     TITLE = {General relativity and the {E}instein equations},
    SERIES = {Oxford Mathematical Monographs},
 PUBLISHER = {Oxford University Press, Oxford},
      YEAR = {2009},
     PAGES = {xxvi+785},
      ISBN = {978-0-19-923072-3},
   MRCLASS = {83-02 (35Q75 83C05)},
  MRNUMBER = {2473363},
MRREVIEWER = {Hans-Peter\ K\"unzle},
}

@article{pryce,
    author = {Pryce, Maurice Henry Lecorney},
    title = {The two-dimensional electrostatic solutions of Born's new field equations},
    journal = {Proc. Camb. Phil. Soc.},
    year = {1935}, 
    volume = {31}, 
    pages = {50--68},
}

@article {fls,
    AUTHOR = {Fern\'andez, Isabel and L\'opez, Francisco J. and Souam,
              Rabah},
     TITLE = {The space of complete embedded maximal surfaces with isolated
              singularities in the 3-dimensional {L}orentz-{M}inkowski
              space},
   JOURNAL = {Math. Ann.},
  FJOURNAL = {Mathematische Annalen},
    VOLUME = {332},
      YEAR = {2005},
    NUMBER = {3},
     PAGES = {605--643},
      ISSN = {0025-5831,1432-1807},
   MRCLASS = {58D10 (53A35)},
  MRNUMBER = {2181764},
MRREVIEWER = {Juan\ A.\ Aledo},
       DOI = {10.1007/s00208-005-0642-6},
       URL = {https://doi-org.pros2.lib.unimi.it/10.1007/s00208-005-0642-6},
}

@Article{ 
  BIMM,
  author = { Byeon, Jaeyoung and Ikoma, Norihisa and Malchiodi, Andrea and Mari, Luciano},
  title = { Existence and regularity for prescribed Lorentzian mean curvature hypersurfaces, and the Born Infeld model },
  journal = { Annals of PDE },
  year = { 2024 },
  doi = { 10.1007/s40818-023-00167-4 },
  volume = { 10 },
  number = { 4 },
  pages = { 86 },
  URL = { http://cvgmt.sns.it/paper/5393/ },
  note = { cvgmt preprint}
}

@Article{BI,
  author = {Born, Max and Infeld, Leopold},
  title = {Foundations of the new field theory},
  journal = {Proc. Roy. Soc. London Ser. A},
  year = {1934},
  doi = {},
  volume = {144},
  number = {},
  pages = {425--451},
}

@article{BI_2,
  author  = {Born, Max and Infeld, Leopold},
  title   = {Foundations of the New Field Theory},
  journal = {Nature},
  year    = {1933},
  volume  = {132},
  number  = {3348},
  pages   = {1004}
}

@article{boillat, 
    author= {Boillat, Guy},
    title={Nonlinear electrodynamics: Lagrangians and equations of motion}, 
    journal={J. Math. Phys}, 
    volume={11},
    year={1970},
    pages={941--951},
}

@book{plebanski,
    author = {Pleba\'nski, Jerzy},
    title = {Lecture notes on nonlinear electrodynamics},
    publisher = {NORDITA},
    year = {1970},
}

@article {KM95,
    AUTHOR = {Klyachin, A. A. and Miklyukov, V. M.},
     TITLE = {The existence of solutions with singularities of the equation
              of maximal surfaces in a {M}inkowski space},
   JOURNAL = {Mat. Sb.},
  FJOURNAL = {Matematicheski\u i\ Sbornik},
    VOLUME = {184},
      YEAR = {1993},
    NUMBER = {9},
     PAGES = {103--124},
      ISSN = {0368-8666,2305-2783},
   MRCLASS = {53A10 (53C21 53C50)},
  MRNUMBER = {1257338},
MRREVIEWER = {Guo\ Ying\ Jiang},
       DOI = {10.1070/SM1995v080n01ABEH003515},
       URL = {https://doi-org.pros2.lib.unimi.it/10.1070/SM1995v080n01ABEH003515},
}

@incollection {Birula,
    AUTHOR = {Bia\l ynicki-Birula, Iwo},
     TITLE = {Nonlinear electrodynamics: variations on a theme by {B}orn and
              {I}nfeld},
 BOOKTITLE = {Quantum theory of particles and fields},
     PAGES = {31--48},
 PUBLISHER = {World Sci. Publishing, Singapore},
      YEAR = {1983},
      ISBN = {9971-950-77-4},
   MRCLASS = {78A25 (81L05)},
  MRNUMBER = {772780},
}

@article{born,
  author  = {Born, Max},
  title   = {Modified Field Equations with a Finite Radius of the Electron},
  journal = {Nature},
  year    = {1933},
  volume  = {132},
  number  = {3329},
  pages   = {282}
}

@article {Kiessling-legacy,
    AUTHOR = {Kiessling, Michael K.-H.},
     TITLE = {Electromagnetic field theory without divergence problems. {I}.
              {T}he {B}orn legacy},
   JOURNAL = {J. Statist. Phys.},
  FJOURNAL = {Journal of Statistical Physics},
    VOLUME = {116},
      YEAR = {2004},
    NUMBER = {1-4},
     PAGES = {1057--1122},
      ISSN = {0022-4715,1572-9613},
   MRCLASS = {81-03 (78A35 81V10)},
  MRNUMBER = {2082203},
MRREVIEWER = {T.\ Erber},
       DOI = {10.1023/B:JOSS.0000037250.72634.2a},
       URL = {https://doi-org.pros2.lib.unimi.it/10.1023/B:JOSS.0000037250.72634.2a},
}

@article {Yang,
    AUTHOR = {Yang, Yisong},
     TITLE = {Nonlinear problems inspired by the {B}orn-{I}nfeld theory of
              electrodynamics},
   JOURNAL = {Adv. Nonlinear Stud.},
  FJOURNAL = {Advanced Nonlinear Studies},
    VOLUME = {24},
      YEAR = {2024},
    NUMBER = {1},
     PAGES = {222--246},
      ISSN = {1536-1365,2169-0375},
   MRCLASS = {35Q75 (78A25 83C22 83C57 83F05)},
  MRNUMBER = {4727590},
       DOI = {10.1515/ans-2023-0123},
       URL = {https://doi-org.pros2.lib.unimi.it/10.1515/ans-2023-0123},
}

@article {boniaco_1,
    AUTHOR = {Bonheure, Denis and Iacopetti, Alessandro},
     TITLE = {On the regularity of the minimizer of the electrostatic
              {B}orn-{I}nfeld energy},
   JOURNAL = {Arch. Ration. Mech. Anal.},
  FJOURNAL = {Archive for Rational Mechanics and Analysis},
    VOLUME = {232},
      YEAR = {2019},
    NUMBER = {2},
     PAGES = {697--725},
      ISSN = {0003-9527,1432-0673},
   MRCLASS = {49N60 (78A30)},
  MRNUMBER = {3925529},
MRREVIEWER = {Eugen\ Viszus},
       DOI = {10.1007/s00205-018-1331-4},
       URL = {https://doi-org.pros2.lib.unimi.it/10.1007/s00205-018-1331-4},
}

@article {boniaco_2,
    AUTHOR = {Bonheure, Denis and Iacopetti, Alessandro},
     TITLE = {A sharp gradient estimate and {$W^{2,q}$} regularity for the
              prescribed mean curvature equation in the
              {L}orentz-{M}inkowski space},
   JOURNAL = {Arch. Ration. Mech. Anal.},
  FJOURNAL = {Archive for Rational Mechanics and Analysis},
    VOLUME = {247},
      YEAR = {2023},
    NUMBER = {5},
     PAGES = {Paper No. 87, 44},
      ISSN = {0003-9527,1432-0673},
   MRCLASS = {35J60 (53C42 53C50)},
  MRNUMBER = {4631021},
MRREVIEWER = {Shi-Zhong\ Du},
       DOI = {10.1007/s00205-023-01910-8},
       URL = {https://doi-org.pros2.lib.unimi.it/10.1007/s00205-023-01910-8},
}

@article {ecker,
    AUTHOR = {Ecker, Klaus},
     TITLE = {Area maximizing hypersurfaces in {M}inkowski space having an
              isolated singularity},
   JOURNAL = {Manuscripta Math.},
  FJOURNAL = {Manuscripta Mathematica},
    VOLUME = {56},
      YEAR = {1986},
    NUMBER = {4},
     PAGES = {375--397},
      ISSN = {0025-2611,1432-1785},
   MRCLASS = {53C50 (49F10 53B30 53C42)},
  MRNUMBER = {860729},
MRREVIEWER = {R.\ Osserman},
       DOI = {10.1007/BF01168501},
       URL = {https://doi-org.pros2.lib.unimi.it/10.1007/BF01168501},
}

@article {fsuy,
    AUTHOR = {Fujimori, Shoichi and Saji, Kentaro and Umehara, Masaaki and
              Yamada, Kotaro},
     TITLE = {Singularities of maximal surfaces},
   JOURNAL = {Math. Z.},
  FJOURNAL = {Mathematische Zeitschrift},
    VOLUME = {259},
      YEAR = {2008},
    NUMBER = {4},
     PAGES = {827--848},
      ISSN = {0025-5874,1432-1823},
   MRCLASS = {53A10 (53C42)},
  MRNUMBER = {2403743},
MRREVIEWER = {Pablo\ Mira},
       DOI = {10.1007/s00209-007-0250-0},
       URL = {https://doi-org.pros2.lib.unimi.it/10.1007/s00209-007-0250-0},
}

@article {various,
    AUTHOR = {Fujimori, S. and Kim, Y. W. and Koh, S.-E. and Rossman, W. and
              Shin, H. and Takahashi, H. and Umehara, M. and Yamada, K. and
              Yang, S.-D.},
     TITLE = {Zero mean curvature surfaces in {$\mathbf{L}^3$} containing a
              light-like line},
   JOURNAL = {C. R. Math. Acad. Sci. Paris},
  FJOURNAL = {Comptes Rendus Math\'ematique. Acad\'emie des Sciences. Paris},
    VOLUME = {350},
      YEAR = {2012},
    NUMBER = {21-22},
     PAGES = {975--978},
      ISSN = {1631-073X,1778-3569},
   MRCLASS = {53A10},
  MRNUMBER = {2996778},
MRREVIEWER = {Francisco\ Mil\'an},
       DOI = {10.1016/j.crma.2012.10.024},
       URL = {https://doi-org.pros2.lib.unimi.it/10.1016/j.crma.2012.10.024},
}

@article {uy,
    AUTHOR = {Umehara, Masaaki and Yamada, Kotaro},
     TITLE = {Maximal surfaces with singularities in {M}inkowski space},
   JOURNAL = {Hokkaido Math. J.},
  FJOURNAL = {Hokkaido Mathematical Journal},
    VOLUME = {35},
      YEAR = {2006},
    NUMBER = {1},
     PAGES = {13--40},
      ISSN = {0385-4035},
   MRCLASS = {53A10 (53C42 53C50)},
  MRNUMBER = {2225080},
MRREVIEWER = {Pablo\ Mira},
       DOI = {10.14492/hokmj/1285766302},
       URL = {https://doi-org.pros2.lib.unimi.it/10.14492/hokmj/1285766302},
}

@incollection {uy_light_0,
    AUTHOR = {Umehara, Masaaki and Yamada, Kotaro},
     TITLE = {Surfaces with light-like points in {L}orentz-{M}inkowski
              3-space with applications},
 BOOKTITLE = {Lorentzian geometry and related topics},
    SERIES = {Springer Proc. Math. Stat.},
    VOLUME = {211},
     PAGES = {253--273},
 PUBLISHER = {Springer, Cham},
      YEAR = {2017},
      ISBN = {978-3-319-66290-9},
   MRCLASS = {53A10 (53A35)},
  MRNUMBER = {3777999},
MRREVIEWER = {Erhan\ G\"uler},
       DOI = {10.1007/978-3-319-66290-9\_14},
       URL = {https://doi-org.pros2.lib.unimi.it/10.1007/978-3-319-66290-9_14},
}

@article {uy_light,
    AUTHOR = {Umehara, M. and Yamada, K.},
     TITLE = {Hypersurfaces with light-like points in a {L}orentzian
              manifold},
   JOURNAL = {J. Geom. Anal.},
  FJOURNAL = {Journal of Geometric Analysis},
    VOLUME = {29},
      YEAR = {2019},
    NUMBER = {4},
     PAGES = {3405--3437},
      ISSN = {1050-6926,1559-002X},
   MRCLASS = {53A10 (35K93 53B30)},
  MRNUMBER = {4015443},
MRREVIEWER = {Rahul\ Kumar\ Singh},
       DOI = {10.1007/s12220-018-00118-7},
       URL = {https://doi-org.pros2.lib.unimi.it/10.1007/s12220-018-00118-7},
}

@article {klyachin_mixed,
    AUTHOR = {Klyachin, V. A.},
     TITLE = {Surfaces of zero mean curvature of mixed type in {M}inkowski
              space},
   JOURNAL = {Izv. Ross. Akad. Nauk Ser. Mat.},
  FJOURNAL = {Izvestiya Rossiiskoi Akademii Nauk. Seriya Matematicheskaya},
    VOLUME = {67},
      YEAR = {2003},
    NUMBER = {2},
     PAGES = {5--20},
      ISSN = {1607-0046,2587-5906},
   MRCLASS = {53A10 (35J60 53A35)},
  MRNUMBER = {1972991},
       DOI = {10.1070/IM2003v067n02ABEH000425},
       URL = {https://doi-org.pros2.lib.unimi.it/10.1070/IM2003v067n02ABEH000425},
}

@article {koba,
    AUTHOR = {Kobayashi, Osamu},
     TITLE = {Maximal surfaces with conelike singularities},
   JOURNAL = {J. Math. Soc. Japan},
  FJOURNAL = {Journal of the Mathematical Society of Japan},
    VOLUME = {36},
      YEAR = {1984},
    NUMBER = {4},
     PAGES = {609--617},
      ISSN = {0025-5645,1881-1167},
   MRCLASS = {53A10 (53C50)},
  MRNUMBER = {759417},
MRREVIEWER = {Martin\ A.\ Magid},
       DOI = {10.2969/jmsj/03640609},
       URL = {https://doi-org.pros2.lib.unimi.it/10.2969/jmsj/03640609},
}

@article {auy,
    AUTHOR = {Akamine, Shintaro and Umehara, Masaaki and Yamada, Kotaro},
     TITLE = {Space-like maximal surfaces containing entire null lines in
              {L}orentz-{M}inkowski 3-space},
   JOURNAL = {Proc. Japan Acad. Ser. A Math. Sci.},
  FJOURNAL = {Japan Academy. Proceedings. Series A. Mathematical Sciences},
    VOLUME = {95},
      YEAR = {2019},
    NUMBER = {9},
     PAGES = {97--102},
      ISSN = {0386-2194},
   MRCLASS = {53A10 (35J93)},
  MRNUMBER = {4026357},
MRREVIEWER = {Irene\ I.\ Onnis},
       DOI = {10.3792/pjaa.95.97},
       URL = {https://doi-org.pros2.lib.unimi.it/10.3792/pjaa.95.97},
}

@article {esturome,
    AUTHOR = {Estudillo, Francisco J. M. and Romero, Alfonso},
     TITLE = {Generalized maximal surfaces in {L}orentz-{M}inkowski space
              {$L^3$}},
   JOURNAL = {Math. Proc. Cambridge Philos. Soc.},
  FJOURNAL = {Mathematical Proceedings of the Cambridge Philosophical
              Society},
    VOLUME = {111},
      YEAR = {1992},
    NUMBER = {3},
     PAGES = {515--524},
      ISSN = {0305-0041,1469-8064},
   MRCLASS = {53A10 (53C50)},
  MRNUMBER = {1151327},
MRREVIEWER = {Cornelia-Livia\ Bejan},
       DOI = {10.1017/S0305004100075587},
       URL = {https://doi-org.pros2.lib.unimi.it/10.1017/S0305004100075587},
}

@article {galvez_jimenez_mira,
    AUTHOR = {G\'alvez, Jos\'e{} A. and Jim\'enez, Asun and Mira, Pablo},
     TITLE = {Isolated singularities of the prescribed mean curvature
              equation in {M}inkowski 3-space},
   JOURNAL = {Ann. Inst. H. Poincar\'e{} C Anal. Non Lin\'eaire},
  FJOURNAL = {Annales de l'Institut Henri Poincar\'e{} C. Analyse Non
              Lin\'eaire},
    VOLUME = {35},
      YEAR = {2018},
    NUMBER = {6},
     PAGES = {1631--1644},
      ISSN = {0294-1449,1873-1430},
   MRCLASS = {35J93 (35J60 53C42)},
  MRNUMBER = {3846238},
MRREVIEWER = {Theodora\ Bourni},
       DOI = {10.1016/j.anihpc.2018.01.004},
       URL = {https://doi-org.pros2.lib.unimi.it/10.1016/j.anihpc.2018.01.004},
}

@article{haa,
    author = {Haarala, Akseli},
    title = {The electrostatic Born-Infeld equations with integrable charge densities}, journal = {available at arXiv:2006.08208},
    year = {2020},            
}

@article {bpd,
    AUTHOR = {Bonheure, Denis and d'Avenia, Pietro and Pomponio, Alessio},
     TITLE = {On the electrostatic {B}orn-{I}nfeld equation with extended
              charges},
   JOURNAL = {Comm. Math. Phys.},
  FJOURNAL = {Communications in Mathematical Physics},
    VOLUME = {346},
      YEAR = {2016},
    NUMBER = {3},
     PAGES = {877--906},
      ISSN = {0010-3616,1432-0916},
   MRCLASS = {78A30 (35J93 35Q60)},
  MRNUMBER = {3537339},
MRREVIEWER = {Benedetta\ Pellacci},
       DOI = {10.1007/s00220-016-2586-y},
       URL = {https://doi-org.pros2.lib.unimi.it/10.1007/s00220-016-2586-y},
}

@article {bcf,
    AUTHOR = {Bonheure, Denis and Colasuonno, Francesca and F\"oldes, Juraj},
     TITLE = {On the {B}orn-{I}nfeld equation for electrostatic fields with
              a superposition of point charges},
   JOURNAL = {Ann. Mat. Pura Appl. (4)},
  FJOURNAL = {Annali di Matematica Pura ed Applicata. Series IV},
    VOLUME = {198},
      YEAR = {2019},
    NUMBER = {3},
     PAGES = {749--772},
      ISSN = {0373-3114,1618-1891},
   MRCLASS = {35J93 (35B40 35B65 35J62 78A30)},
  MRNUMBER = {3954391},
MRREVIEWER = {David\ James\ Hartley},
       DOI = {10.1007/s10231-018-0796-y},
       URL = {https://doi-org.pros2.lib.unimi.it/10.1007/s10231-018-0796-y},
}

@article {kiessling,
    AUTHOR = {Kiessling, Michael K.-H.},
     TITLE = {On the quasi-linear elliptic {PDE}
              {$-\nabla\cdot(\nabla{u}/\sqrt{1-|\nabla{u}|^2})=4\pi\sum_ka_k\delta_{s_k}$}
              in physics and geometry},
   JOURNAL = {Comm. Math. Phys.},
  FJOURNAL = {Communications in Mathematical Physics},
    VOLUME = {314},
      YEAR = {2012},
    NUMBER = {2},
     PAGES = {509--523. Corrigendum on: Comm. Math. Phys. 364 (2018), no. 2, 825–833.},
      ISSN = {0010-3616,1432-0916},
   MRCLASS = {58J05 (35J62)},
  MRNUMBER = {2958962},
MRREVIEWER = {Yisong\ Yang},
       DOI = {10.1007/s00220-012-1502-3},
       URL = {https://doi-org.pros2.lib.unimi.it/10.1007/s00220-012-1502-3},
}

@article {impera,
    AUTHOR = {Impera, Debora},
     TITLE = {Comparison theorems in {L}orentzian geometry and applications
              to spacelike hypersurfaces},
   JOURNAL = {J. Geom. Phys.},
  FJOURNAL = {Journal of Geometry and Physics},
    VOLUME = {62},
      YEAR = {2012},
    NUMBER = {2},
     PAGES = {412--426},
}

@article {alias_hurtado_palmer,
    AUTHOR = {Al\'ias, Luis J. and Hurtado, Ana and Palmer, Vicente},
     TITLE = {Geometric analysis of {L}orentzian distance function on
              spacelike hypersurfaces},
   JOURNAL = {Trans. Amer. Math. Soc.},
  FJOURNAL = {Transactions of the American Mathematical Society},
    VOLUME = {362},
      YEAR = {2010},
    NUMBER = {10},
     PAGES = {5083--5106},
}

@article {erke_garcia_kupe,
    AUTHOR = {Erkeko\u glu, F. and Garc\'ia-R\'io, E. and Kupeli, D. N.},
     TITLE = {On level sets of {L}orentzian distance function},
   JOURNAL = {Gen. Relativity Gravitation},
  FJOURNAL = {General Relativity and Gravitation},
    VOLUME = {35},
      YEAR = {2003},
    NUMBER = {9},
     PAGES = {1597--1615},
}

@article{bartnik84,
    author = {Robert Bartnik},
    title = {Existence of Maximal Surfaces
in Asymptotically Flat Spacetimes},
    journal = {Commun.Math.Phys.},
    year = {1984},
    volume = {94},
    pages = {155–175}
}

@article{BS,
    author = {Robert Bartnik and Leon Simon},
    title = {Spacelike hypersurfaces with prescribed boundary values and mean curvature},
    journal = {Commun.Math.Phys.},
    year = {1982},
    volume = {87},
    pages = {131–152},
    URL = {https://doi.org/10.1007/BF01211061}
}

@article{bartnik88,
    author = {Robert Bartnik},
    title = {Regularity of variational maximal surfaces},
    journal = {Acta Math.},
    year = {1988},
    volume = {161},
    pages = {145–181}
}

@book{aubin,
    author = {Thierry Aubin},
    title = {Nonlinear Analysis on Manifolds. Monge-Ampère Equations},
    publisher = {Springer New York, NY},
    year = {1982}
}

@book{lee2011topological,
  author    = {John M. Lee},
  title     = {Introduction to Topological Manifolds},
  series    = {Graduate Texts in Mathematics},
  volume    = {202},
  edition   = {2},
  publisher = {Springer},
  address   = {New York},
  year      = {2011}
}

@book {oneill,
    AUTHOR = {O'Neill, Barrett},
     TITLE = {Semi-{R}iemannian geometry},
    SERIES = {Pure and Applied Mathematics},
    VOLUME = {103},
      NOTE = {With applications to relativity},
 PUBLISHER = {Academic Press, Inc. [Harcourt Brace Jovanovich, Publishers],
              New York},
      YEAR = {1983},
     PAGES = {xiii+468},
}

@book {he,
    AUTHOR = {Hawking, S. W. and Ellis, G. F. R.},
     TITLE = {The large scale structure of space-time},
    SERIES = {Cambridge Monographs on Mathematical Physics},
    VOLUME = {No. 1},
 PUBLISHER = {Cambridge University Press, London-New York},
      YEAR = {1973},
     PAGES = {xi+391},
}

@book{GilbargTrudinger1977,
  author    = {David Gilbarg and Neil S. Trudinger},
  title     = {Elliptic Partial Differential Equations of Second Order},
  publisher = {Springer},
  year      = {1977},
  series    = {Grundlehren der Mathematischen Wissenschaften},
  volume    = {224},
  address   = {Berlin},
}

@book {evans,
    AUTHOR = {Evans, Lawrence C.},
     TITLE = {Partial differential equations},
    SERIES = {Graduate Studies in Mathematics},
    VOLUME = {19},
   EDITION = {Second},
 PUBLISHER = {American Mathematical Society, Providence, RI},
      YEAR = {2010},
     PAGES = {xxii+749}
}

@book{chavel2006riemannian,
  title={Riemannian Geometry: A Modern Introduction},
  author={Chavel, I.},
  isbn={9781139452571},
  series={Cambridge Studies in Advanced Mathematics},
  url={https://books.google.it/books?id=3Gjp4vQ_mPkC},
  year={2006},
  publisher={Cambridge University Press}
}

@article{greene_wu,
    AUTHOR = {Greene, R. E. and Wu, H.},
     TITLE = {{$C\sp{\infty }$}\ approximations of convex, subharmonic, and
              plurisubharmonic functions},
   JOURNAL = {Ann. Sci. \'Ecole Norm. Sup. (4)},
  FJOURNAL = {Annales Scientifiques de l'\'Ecole Normale Sup\'erieure.
              Quatri\`eme S\'erie},
    VOLUME = {12},
      YEAR = {1979},
    NUMBER = {1},
     PAGES = {47--84},
}

@book{maz2008theory,
  title={Theory of Sobolev Multipliers: With Applications to Differential and Integral Operators},
  author={Maz'ya, V. and Shaposhnikova, T.O.},
  isbn={9783540694922},
  lccn={2008932182},
  series={Grundlehren der mathematischen Wissenschaften},
  url={https://books.google.it/books?id=QN8uP6Mn0yQC},
  year={2008},
  publisher={Springer Berlin Heidelberg}
}

@book{lee:smooth,
  author    = {John M. Lee},
  title     = {Introduction to Smooth Manifolds},
  series    = {Graduate Texts in Mathematics},
  volume    = {218},
  edition   = {2},
  publisher = {Springer},
  year      = {2013},
  doi       = {10.1007/978-1-4419-9982-5},
  isbn      = {978-1-4419-9981-8}
}

@book{misner1973gravitation,
  title        = {Gravitation},
  author       = {Misner, Charles W. and Thorne, Kip S. and Wheeler, John Archibald},
  year         = {1973},
  publisher    = {W. H. Freeman},
  address      = {San Francisco}
}
\nocite{*}

\vspace{0.5cm}

\noindent \textbf{Conflict of Interests.} The authors have no conflict of interest. \\[0.2cm]
\noindent \textbf{Data availability statement.} No data was generated by this research.

\vspace{0.5cm}

\end{document}